\documentclass[11pt,a4paper]{article}

\usepackage[left=3cm, top=2.5cm,bottom=2.5cm,right=3cm]{geometry}
\usepackage{mathtools,amssymb,amsthm,mathrsfs,calc,graphicx,xcolor,cleveref,dsfont,tikz,pgfplots,bm, bbm,pdflscape}
\usepackage{tikz-cd}

\usepackage[british]{babel}
\usepackage{amsfonts}              
\usepackage[T1]{fontenc}

\newcommand{\ind}{\mathbbm{1}}

\newtheorem {theorem}{Theorem}[section]
\newtheorem {proposition}[theorem]{Proposition}

\newtheorem {corollary}[theorem]{Corollary}

{\theoremstyle{definition}

}

{
\theoremstyle{remark}
\newtheorem {remark}[theorem]{Remark}
}

\def\EE{\mathbb{E}}

\def\GG{\mathbb{G}}

\def\KK{\mathbb{K}}

\def\NN{\mathbb{N}}

\def\PP{\mathbb{P}}

\def\RR{\mathbb{R}}
\def\RRd1{\mathbb{R}^{d+1}}
\def\SS{\mathbb{S}}
\def\SSd{\mathbb{S}^d}

\def\ZZ{\mathbb{Z}}

\def\bE{\mathbf{E}}

\def\bP{\mathbf{P}}

\def\cA{\mathcal{A}}
\def\cB{\mathcal{B}}

\def\cF{\mathcal{F}}

\def\cH{\mathscr{H}}

\def\dint{\textup{d}}

\def\SO{\textup{SO}}
\def\pos{\textup{pos}}

\def\lin{\textup{lin}}
\def\cv{\breve{v}}

\makeatletter
\let\@fnsymbol\@alph
\makeatother

\begin{document}

\title{\bfseries Faces in random great hypersphere tessellations}

\author{Zakhar Kabluchko\footnotemark[1] ~and Christoph Th\"ale\footnotemark[2]}

\date{}
\renewcommand{\thefootnote}{\fnsymbol{footnote}}

\footnotetext[1]{Institut f\"ur Mathematische Stochastik, Westf\"alische Wilhelms-Universit\"at M\"unster, Germany. Email: zakhar.kabluchko@uni-muenster.de}

\footnotetext[2]{Fakult\"at f\"ur Mathematik, Ruhr-Universit\"at Bochum, Germany. Email: christoph.thaele@rub.de}

\maketitle

\begin{abstract}
\noindent  The concept of typical and weighted typical spherical faces for tessellations of the $d$-dimensional unit sphere, generated by $n$ independent random great hyperspheres distributed according to a non-degenerate directional distribution, is introduced and studied. Probabilistic interpretations for such spherical faces are given and their directional distributions are determined. Explicit formulas for the expected $f$-vector, the expected spherical Querma\ss integrals and the expected spherical intrinsic volumes are found in the isotropic case. Their limiting behaviour as $n\to\infty$ is discussed and compared to the corresponding notions and results in the Euclidean case. The expected statistical dimension and a problem related to intersection probabilities of spherical random polytopes is investigated.
\end{abstract}


\section{Introduction}

The analysis of random tessellations in $\RR^d$ and the resulting random polytopes has a long tradition in stochastic geometry. Particularly attractive, at least from a mathematical point of view, is the class of Poisson hyperplane tessellations for which many explicit results are available, see e.g.\ \cite{Matheron,MilesFlats,SW} for representative overviews. This class of models has recently found also interesting applications in compressed sensing \cite{BaccelliOReilly,BilykLacey}. This paper deals with the natural analogue of such tessellations in spherical spaces of constant positive curvature $+1$. More precisely, we are dealing with random tessellations of the $d$-dimensional unit sphere $\SS^d\subset\RR^{d+1}$ generated by $n\in\NN$ independent random great hyperspheres. Such tessellations were previously investigated in \cite{ArbeiterZaehle,CoverEfron,HugSchneiderConicalTessellations,MilesSphere}. These works deal with mean value relations for geometric characteristics of these tessellations as well as with their so-called typical and, occasionally, also with their weighted typical cells. These cells arise as follows. The typical cell is a random cell which is selected uniformly from the (almost surely finite) collection of all cells, while the weighted typical cell is its size biased version, where size is measured by the $d$-dimensional spherical Lebesgue measure. In contrast to these previous works we are not only interested in the cells of random great hypersphere tessellations, but also in their lower-dimensional faces. This gives rise to new questions, mainly related to the notion of direction, which in the spherical set-up typically have different answers compared to their Euclidean counterparts. Motivated by the approach in \cite{SchneiderWeightedFaces}, which deals with Poisson hyperplane tessellations in $\RR^d$, for any $k\in\{0,1,\ldots,d\}$ we introduce two different types of random $k$-dimensional spherical faces associated with a great hypersphere tessellation, the $k$-dimensional typical spherical face and the $k$-dimensional weighted typical spherical face. In the equivalent language of conical random tessellations, the typical spherical $k$-face generalizes to lower dimensions the concept of the Schl\"afli random cone studied in \cite{CoverEfron,HugSchneiderConicalTessellations,KabluchkoTemesvariThaeleCones}. We investigate the relation between and provide probabilistic interpretations of these lower-dimensional spherical random polytopes. In particular, we determine explicitly several of their key characteristic quantities in the isotropic case, that is, if the distribution of the underlying great hyperspheres is the uniform distribution on the space of great hyperspheres. This is essentially based on the connection of weighted typical spherical $k$-faces by spherical polarity to random convex hulls on half-spheres, which in turn can be analysed using the language of random beta polytopes, see \cite{kabluchko_poisson_zero,KabluchkoTemesvariThaele,KabluchkoThaeleZaporozhets,MarynychKabluchkoTemesvariThaele}. To be precise, for each $k\in\{0,1,\ldots,d\}$, we determine explicitly the expected number $\ell$-dimensional spherical faces, $\ell\in\{0,1,\ldots,k-1\}$, of the typical and the weighted typical $k$-dimensional spherical face. We also provide fully explicit formulas for the expected spherical Querma\ss integrals and the expected spherical intrinsic volumes. We also deal with the expected statistical dimension and study a problem from geometric probability, dealing with intersection probabilities of spherical random polytopes.

Our paper continues and adds to a recent active line of research on random geometric systems in non-Euclidean spaces. Recent works directly linked with this text are the articles \cite{BaranyHugReitznerSchneider,kabluchko_poisson_zero,MarynychKabluchkoTemesvariThaele} on spherical convex hulls on half-spheres, the papers \cite{DeussHoerrmannThaele,GodlandKabluchko,HeroldHugThaele,HugReichenbacher,HugSchneiderConicalTessellations,HugThaele} dealing with different types of hyperplane or splitting tessellations in spherical (and hyperbolic) spaces or the publications \cite{HugReichenbacher,KabluchkoThaele_VoronoiSphere} about Voronoi tessellations on the sphere. Let us also mention here the work \cite{SchneiderKinematicCones}, which studies intersection probabilities for deterministic and random cones, and \cite{HoldenPeresZhai}, where random tessellations of the $2$-dimensional sphere generated by a gravitational allocation scheme are investigated.

\medspace

The remaining parts of this text are structured as follows. In Section \ref{sec:HypersphereTess} we recall some preliminaries, introduce random great hypersphere tessellations and formally define the notions of typical and weighted typical spherical $k$-faces. Probabilistic interpretations of such faces in terms of intersections are provided in Section \ref{sec:ProbInterpretation}. We also show there that the weighted typical spherical $k$-face is the size biased version of the typical spherical $k$-face. The directional distribution of both types of faces as well as the distribution of faces with given direction are determined in Section \ref{sec:DirDistr}, whereas in Section \ref{sec:Iso} we concentrate (mainly) on the isotropic case. Especially, we provide there explicit formulas for the expected $f$-vector, the expected spherical Querma\ss integrals and the expected spherical intrinsic volumes of typical and weighted typical spherical $k$-faces. We also study their asymptotic behaviour, determine their statistical dimension and analyse a question related to intersection probabilities of spherical random polytopes.

\section{Great hypersphere tessellations and typical spherical faces}\label{sec:HypersphereTess}

\subsection{Preliminaries}

\paragraph{Spaces of subspheres and polytopes.} We fix a space dimension $d\geq 1$ and consider the $d$-dimensional unit sphere $\SS^d\subset\RR^{d+1}$. For $k\in\{0,1,\ldots,d\}$ we denote by $\GG_s(d,k)$ the \textit{spherical Grassmannian} of $k$-dimensional great subspheres of $\SS^d$  and refer to the elements of $\GG_s(d,d-1)$ as great hyperspheres. Also observe that $\GG_s(d,d)=\SS^d$. Each of the spaces $\GG_s(d,k)$ carries a unique rotation invariant Haar probability measure, $\nu_k$, see \cite[Chapter 6.5]{SW}. Following the convention in \cite{SW}, we also define the constant
$$
\omega_{k+1} = \cH^k(S) = {2\pi^{k/2}\over\Gamma({k\over 2})},
$$
where $S\in\GG_s(d,k)$ is arbitrary and $\cH^k(\,\cdot\,)$ denotes the $k$-dimensional Hausdorff measure.

We write $\KK_s(d)$ for the space of spherical convex subsets of $\SS^d$, where we recall that a subset of $\SSd$ is convex if it is the intersection of $\SSd$ with a closed convex cone in $\RRd1$ different from $\{0\}$. Moreover, we use the symbol $\PP_s(d)$ to indicate the space of spherical polytopes in $\SS^{d}$.
By such a polytope we understand the intersection of $\SS^{d}$ with a polyhedral cone in $\RR^{d+1}$. 
The intersection of $\SS^d$ with an $(\ell+1)$-dimensional face of the polyhedral cone generating a spherical polytope $P\in\PP_s(d)$ is called an \textit{$\ell$-dimensional spherical face} of $P$, $\ell\in\{0,1,\ldots,k\}$.

The Borel $\sigma$-fields on $\KK_s(d)$ and $\PP_s(d)$ generated by the spherical Hausdorff distance are denoted by $\cB(\KK_s(d))$ and $\cB(\PP_s(d))$, respectively. More generally, if $E$ is a topological space then $\cB(E)$ will denote the Borel $\sigma$-field on $E$ generated by the given topology.

\begin{figure}
\centering
\includegraphics[width=\columnwidth]{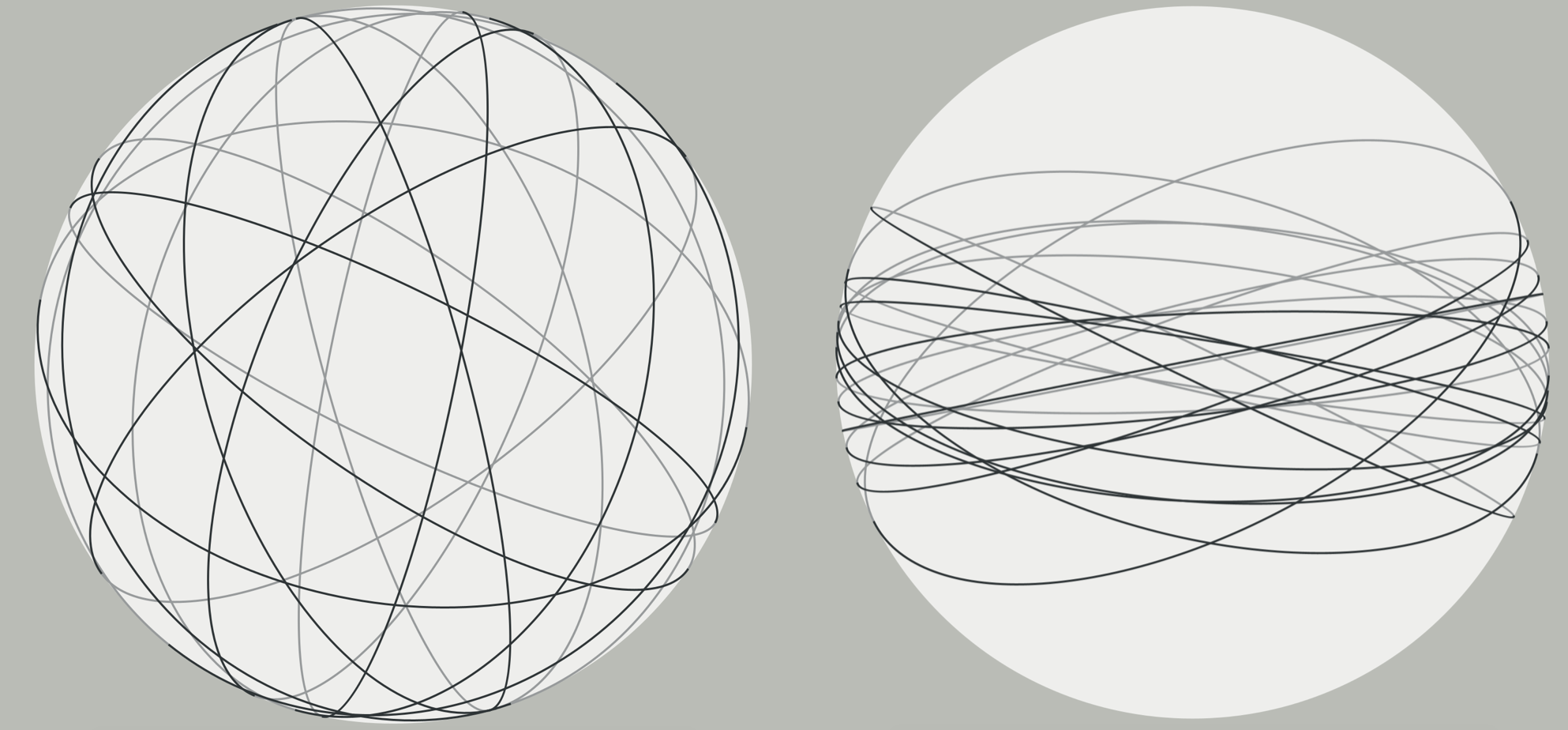}
\caption{Two realizations of great hypersphere tessellation of $\SS^2$ with different directional distributions $\kappa$. In the left picture $\kappa$ is the uniform distribution on $\GG_s(2,1)$, while in the right picture the great hyperspheres are concentrated close to the equator.}
\end{figure}

\paragraph{Great hypersphere tessellations.} In this paper, $\kappa^\circ$ will denote an even probability measure on $\SS^{d}$ with the property that $\kappa^\circ(S)=0$ for any great hypersphere $S\in\GG_s(d,d-1)$. The image measure of $\kappa^\circ$ under the orthogonal complement map $\perp:\SS^{d}\to\GG_s(d,d-1),u\mapsto u^\perp\cap\SS^{d}$ is denoted by $\kappa$. It is a probability measure on the space $\GG_s(d,d-1)$ of great hyperspheres of $\SS^d$. We let $n\in\NN$ and consider a binomial point process $\xi_n$ on $\GG_s(d,d-1)$ with intensity measure $n\kappa$. That is,
$$
\xi_n=\{S_1,\ldots,S_n\}
$$
with independent random great hyperspheres $S_1,\ldots,S_n$ all having distribution $\kappa$. We assume that all random elements we consider are defined on a probability space $(\Omega,\cA,\bP)$ and denote by $\bE$ expectation (integration) with respect to $\bP$. We note that our assumption on $\kappa$ (or $\kappa^\circ$) implies that $\kappa$ is \textit{non-degenerate} in the sense that with probability one the great hyperspheres $S_1,\ldots,S_n$ are in general position, which means that with probability one for any $1\leq i_1<\ldots< i_{d-k}\leq n$, $k\in\{1,\ldots,d\}$, we have that $S_{i_1}\cap\ldots\cap S_{i_{d-k}}\in\GG_s(d,k)$. The great hyperspheres from $\xi_n$ partition $\SS^{d}$ into a random collection of spherical polytopes, which are referred to as \textit{cells} in the sequel. By a classical result of Steiner (for $\SS^2$) and Schl\"afli (for general $\SS^d$) the random collection almost surely consists of
\begin{equation}\label{eq:Cnd}
C(n,d) = 2\sum_{r=0}^{d}{n-1\choose r}
\end{equation}
spherical random polytopes $P_1,\ldots,P_{C(n,d)}$, say. We call
$$
T_{n,d}=T_{d}(S_1,\ldots,S_n)=\{P_1,\ldots,P_{C(n,d)}\}
$$
a random \textit{great hypersphere tessellation} of $\SS^{d}$ with \textit{intensity} $n$ and \textit{directional distribution} $\kappa$. Let us remark that in this generality the model has previously been considered in \cite{HugSchneiderConicalTessellations}, while the works \cite{ArbeiterZaehle,MilesSphere} only deal with the special case where $\kappa$ is the uniform distribution $\nu_{d-1}$ on $\GG_s(d,d-1)$ (or, equivalently, $\kappa^\circ$ is the uniform distribution on $\SS^{d}$). We refer to the latter situation as the \textit{isotropic} case, which we will intensively study in Section \ref{sec:Iso}.

For a spherical polytope $P\in\PP_s(d)$ and $k\in\{0,1,\ldots,d\}$ we write $\cF_k(P)$ for the set of its $k$-dimensional spherical faces, called spherical $k$-faces for short. The cardinality of $\cF_k(P)$ is denoted by $f_k(P)=|\cF_k(P)|$. We also write
$$
\cF_k(T_{n,d})=\cF_k(S_1,\ldots,S_n)=\bigcup_{P\in T_{n,d}}\cF_k(P)
$$
for the set of spherical $k$-faces of $T_{n,d}$, which is generated by $S_1,\ldots,S_n$. Using the non-degeneracy property of $\kappa$ it easily follows that almost surely
\begin{equation}\label{eq:Cndk}
|\cF_k(T_{n,d})| = {n\choose d-k}C(n-d+k,k) =: C(n,d,k),
\end{equation}
see \cite[Equation (16)]{HugSchneiderConicalTessellations}.

\begin{remark}
We would like to emphasize that although we are using the same notation as in \cite{HugSchneiderConicalTessellations,SW}, we are working on $\SSd$, while in \cite{HugSchneiderConicalTessellations,SW} the $(d-1)$-dimensional unit sphere $\SS^{d-1}$ is considered. In particular, this implies that the sum in the definition \eqref{eq:Cnd} of constants $C(n,d)$ runs up to $d$ in our case and not to $d-1$ as in \cite{HugSchneiderConicalTessellations,SW}. This should be kept in mind when our results are compared to those in the literature.
\end{remark}

\subsection{Typical and weighted typical spherical faces}

We assume the same set-up as in the previous section and fix $k\in\{0,1,\ldots,d\}$. In order to avoid the discussion of degenerate cases, we assume that the number $n$ of great hyperspheres satisfies $n\geq d-k$.

To obtain the \textit{typical spherical $k$-face} $Z_{n,d}^{(k)}$ of the great hypersphere tessellation $T_{n,d}$ we choose uniformly at random one of the $C(n,d,k)$ spherical $k$-faces of $T_{n,d}$ (note that by our assumption on $n$ we have that $\cF_k(T_{n,d})\neq\varnothing$). The distribution of $Z_{n,d}^{(k)}$ is formally given by
\begin{equation}\label{eq:DistributionTypicalKface}
\begin{split}
\bP(Z_{n,d}^{(k)}\in A) &= \int_{\GG_s(d,d-1)^n}{1\over C(n,d,k)}\sum_{F\in\cF_k(S_1,\ldots,S_n)}{\bf 1}\{F\in A\}\,\kappa^n(\dint(S_1,\ldots,S_n)),
\end{split}
\end{equation}
where $A\in\cB(\PP_s(d,k))$. We also use the convention to drop the upper index if $k=d$, that is, we write $Z_{n,d}$ instead of $Z_{n,d}^{(d)}$ for the typical cell of $T_{n,d}$.

\begin{remark}
Taking $k=d$ in the previous definition we get back the spherical random polytope whose conical version was studied in \cite{HugSchneiderConicalTessellations,KabluchkoTemesvariThaeleCones} under the name  Sch\"afli random cone.
\end{remark}

To introduce the {weighted typical spherical $k$-face} of $T_{n,d}$ one considers the $k$-skeleton ${\rm skel}_k(T_{n,d})$ of $T_{n,d}$, by which we mean the random closed set on $\SS^{d}$ (in the sense of \cite[Chapter 2]{SW}) consisting of the union of all spherical $k$-faces of cells of $T_{n,d}$, that is,
\begin{align}\label{eq:Skeleton}
{\rm skel}_k(T_{n,d})=\bigcup_{P\in T_{n,d}}\bigcup_{F\in\cF_k(P)}F .
\end{align}
By our assumption on $n$, ${\rm skel}_k(T_{n,d})\neq\varnothing$ and it is easy to verify that the non-degeneracy assumption on $\kappa$ implies that almost surely $\cH^k({\rm skel}_k(T_{n,d}))={n\choose d-k}\omega_{k+1}\in(0,\infty)$. We can thus choose a random point $v$ on ${\rm skel}_k(T_{n,d})$ according to the normalized $k$-dimensional Hausdorff measure. This point almost surely lies in the relative interior of a unique spherical $k$-face $F_v=F_v(T_{n,d})$ of $T_{n,d}$. By definition, this is what we mean by the \textit{weighted typical spherical $k$-face} $W_{n,d}^{(k)}$ of $T_{n,d}$, where the term `weight' always refers to the Hausdorff measure $\cH^k$. Its distribution is formally given by
\begin{equation}\label{eq:DistributionWeightedTypicalKface}
\begin{split}
&\bP(W_{n,d}^{(k)}\in A) \\
&\quad= \int_{\GG_s(d,d-1)^n}{1\over\cH^k({\rm skel}_k(T_{n,d}))}\int_{{\rm skel}_k(T_{n,d})}{\bf 1}\{F_v\in A\}\,\cH^k(\dint v)\kappa^n(\dint(S_1,\ldots,S_n)),
\end{split}
\end{equation}
where $A\in\cB(\PP_s(d,k))$. As above, we use the convention to drop the upper index if $k=d$, that is, we write $W_{n,d}$ instead of $W_{n,d}^{(d)}$ for the weighted typical cell of $T_{n,d}$.

\begin{remark}
In the special case $k=d$ we just have ${\rm skel}_{d}(T_{n,d})=\SS^{d}$ and $W_{n,d}^{(d)}$ is thus the almost surely uniquely determined cell of $T_{n,d}$ which contains a uniform random point on the sphere $\SS^{d}$. The conical versions of such cells appeared in \cite[Section 5]{HugSchneiderConicalTessellations} in the isotropic case.
\end{remark}

The notion of typical spherical $k$-faces and weighted typical spherical $k$-faces parallels in spirit the concepts known from the Euclidean case (see \cite{SchneiderWeightedFaces} and \cite{HugSchneiderFacesDirection}). However, a stationary random tessellation in $\RR^d$ has with probability one infinitely many cells, so that Palm distributions need to be used to introduce the corresponding notions. The compactness of the spherical space $\SS^{d}$ allows a much more direct approach, since the number of $k$-faces of $T_{n,d}$ is almost surely finite. In addition, it should be noted that the common stationarity assumption in the Euclidean case would translate naturally into an isotropy assumption for random tessellations on $\SS^{d}$, see \cite{ArbeiterZaehle,MilesSphere}. In our set-up we work, however, with a general directional distribution, which is is not compatible with the invariance assumption required for Palm calculus. This is the reason why we do not work with Palm distributions for random measures on the sphere in order to define $Z_{n,d}^{(k)}$ and $W_{n,d}^{(k)}$.

\section{Probabilistic interpretation of typical spherical faces}\label{sec:ProbInterpretation}

\subsection{Interpretation of $W_{n,d}^{(k)}$ and $Z_{n,d}^{(k)}$ via intersections}

In this section we consider a great hypersphere tessellation $T_{n,d}$ of $\SS^{d}$ of intensity $n\geq d-k$ and with non-degenerate directional distribution $\kappa$, which is driven by a binomial point process $\xi_n$ on $\GG_s(d,d-1)$ with intensity measure $n\kappa$. By $Z_{n,d}^{(k)}$ and $W_{n,d}^{(k)}$ we denote the typical and the weighted typical spherical $k$-face of $T_{n,d}$, respectively, where $k\in\{0,1,\ldots,d\}$ is a fixed dimension parameter.

Our first result is a description of the weighted typical spherical $k$-face $W_{n,d}^{(k)}$ as the intersection of the weighted typical cell of $T_{n-d+k,d}$ with a $k$-dimensional random great subsphere. This can be seen as  the spherical analogue of \cite[Theorem 1]{SchneiderWeightedFaces}. Below, we also derive a similar description for the typical spherical $k$-face as well. Before we present the result, we introduce for $k\in\{0,1,\ldots,d-1\}$ the measure $\kappa_k$ on $\GG_s(d,k)$ by putting
\begin{equation}\label{eq:Kappak}
\kappa_k(C) = \int_{\GG_s(d,d-1)^{d-k}}{\bf 1}\{S_1\cap\ldots\cap S_{d-k}\in C\}\,\kappa^{d-k}(\dint(S_1,\ldots,S_{d-k}))
\end{equation}
for sets $C\in\cB(\GG_s(d,k))$. It can be interpreted as the directional distribution of the $k$-th intersection process of $\xi_n$, which arises by intersecting any $d-k$ great hyperspheres from $\xi_n$. Clearly, $\kappa_{d-1}=\kappa$.

\begin{theorem}\label{thm:WeightFaceIntersection}
	Let $d\geq 1$, $k\in\{0,1,\ldots,d\}$ and consider a great hypersphere tessellation of $\SS^d$ with non-degenerate directional distribution $\kappa$ and intensity $n\geq d-k$.
	Let $h:\KK_s(d)\to\RR$ be a non-negative measurable function. Then
	\begin{align*}
	\bE h(W_{n,d}^{(k)}) &= {1\over\omega_{k+1}}\int_{\GG_s(d,k)}\int_S\bE h(Z_v(T_{n-d+k,d})\cap S)\,\cH^k(\dint v)\kappa_k(\dint S),
	\end{align*}
	where $Z_v(T_{n-d+k,d})$ stands for the almost surely unique cell of $T_{n-d+k,d}$ containing $v$ in its relative interior.
\end{theorem}
\begin{proof}
We start by recalling the definition \eqref{eq:Skeleton} of the $k$-skeleton of $T_{n,d}$. By construction of $T_{n,d}$ we have that
\begin{align}\label{eq:Skeleton2}
{\rm skel}_k(T_{n,d})= \bigcup_{1\leq i_1<\ldots<i_{d-k}\leq n}S_{i_1}\cap\ldots\cap S_{i_{d-k}}.
\end{align}
Since all great hyperspheres $S_1,\ldots,S_n$ generating $T_{n,d}$ are identically distributed and almost surely in general position, this together with the definition \eqref{eq:DistributionWeightedTypicalKface} of $W_{n,d}^{(k)}$ yields
\begin{align*}
&\bE h(W_{n,d}^{(k)})\\
&= \int_{\GG_s(d,d-1)^n}{1\over\cH^k({\rm skel}_k(T_{n,d}))}\int_{{\rm skel}_k(T_{n,d})}h(F_v(T_{n,d}))\,\cH^k(\dint v)\kappa^n(\dint(S_1,\ldots,S_n))\\
&= \int_{\GG_s(d,d-1)^n} {1\over\sum\limits_{1\leq i_1<\ldots<i_{d-k}\leq n}\cH^k(S_{i_1}\cap\ldots\cap S_{i_{d-k}})}\\
&\hspace{1cm}\times\sum_{1\leq i_1<\ldots<i_{d-k}\leq n}\int_{S_{i_1}\cap\ldots\cap S_{i_{d-k}}} h(F_v(T_{n,d}))\,\cH^k(\dint v)\kappa^n(\dint(S_1,\ldots,S_n))\\
&=\int_{\GG_s(d,d-1)^{d-k}}{1\over\omega_{k+1}}\int_{S_1\cap\ldots\cap S_{d-k}}\int_{\GG_s(d,d-1)^{n-d+k}}h(F_v(T_{n,d}))\\
&\hspace{1cm}\times\kappa^{n-d+k}(\dint(S_{d-k+1},\ldots,S_n))\cH^k(\dint v)\kappa^{d-k}(\dint(S_1,\ldots,S_{d-k}))\\
&=\int_{\GG_s(d,d-1)^{d-k}}{1\over\omega_{k+1}}\int_{S_1\cap\ldots\cap S_{d-k}}\int_{\GG_s(d,d-1)^{n-d+k}}h(F_v(T_{n-d+k,d}\cap(S_1\cap\ldots\cap S_{d-k})))\\
&\hspace{1cm}\times\kappa^{n-d+k}(\dint(S_{d-k+1},\ldots,S_n))\cH^k(\dint v)\kappa^{d-k}(\dint(S_1,\ldots,S_{d-k})),
\end{align*}
where $T_{n-d+k,d}=T_d(S_{d-k+1},\ldots,S_n)$ is the tessellation of $\SS^d$ generated by the $n-d+k$ great hyperspheres $S_{d-k+1},\ldots,S_n$. Applying now the definition of the measure $\kappa_k$ yields
\begin{align*}
\bE h(W_{n,d}^{(k)}) &= {1\over\omega_{k+1}}\int_{\GG_s(d,k)}\int_S \bE h(F_v(T_{n-d+k,d}\cap S)\,\cH^k(\dint v)\kappa_k(\dint S).
\end{align*}
The result follows now by observing that $F_v(T_{n-d+k,d}\cap S)=Z_v(T_{n-d+k,d})\cap S$.
\end{proof}

\begin{figure}
	\centering
	\begin{tikzpicture}[]
	\pgftext{\includegraphics[width=250pt]{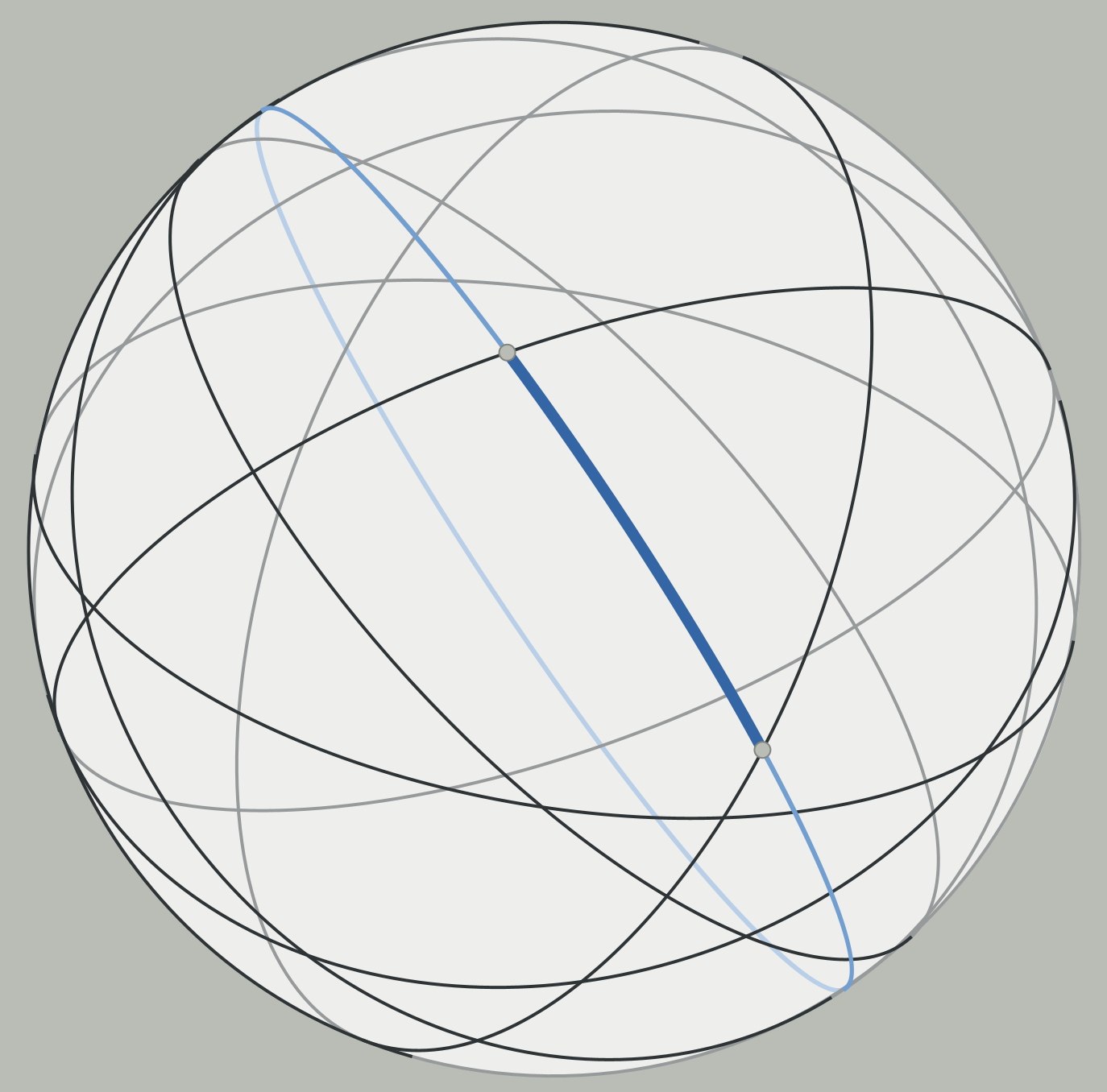}} at (0pt,0pt);
	\node at (-3.5,3.6) {$T_{n-d+k,d}$};
	\node at (0,0) {$S$};
	\end{tikzpicture}
	\caption{Construction of $Z_{n,d}^{(k)}$ or $W_{n,d}^{(k)}$ by intersection of $T_{n-d+k,d}$ with a random great subsphere $S$.}
\end{figure}

We also note the following corollary in the isotropic case.  To present it, we define the $k$-dimensional great subsphere $E_k=\SS^{d}\cap\EE_k$, where $\EE_k$ is the $(k+1)$-dimensional linear subspace spanned by the last $k+1$ vectors of the standard orthonormal basis of $\RR^{d+1}$. Moreover, we let $e=(0,\ldots,0,1)$ be the north pole of $\SS^{d}$. Clearly, $e\in E_k$.

\begin{corollary}\label{cor:CellSection}
	Let $d\geq 1$, $k\in\{0,1,\ldots,d\}$ and consider an isotropic great hypersphere tessellation of $\SS^d$ with intensity $n\geq d-k$. Let $h:\KK_s(d)\to\RR$ be a rotation invariant, non-negative measurable function. Then
	$$
	\bE h(W_{n,d}^{(k)}) = \bE h(Z_e(T_{n-d+k,d})\cap E_k) = \bE h(W_{n-d+k,k}).
	$$
\end{corollary}
\begin{proof}
	For any $v\in\SS^d$ we let
	$$
	\Theta_v=\{\varrho\in\SO(d+1):\varrho e=v\}
	$$
	be the set of rotations that rotate $e$ to $v$. The set $\Theta_{e}$ naturally has the structure of a compact group and we let $\nu_{e}$ be the rotation invariant Haar probability measure on $\Theta_{e}$. For general $v\in\SS^d$ we let $\nu_v=\nu_{e}\circ\varrho^{-1}$ be a probability measure on the co-set $\Theta_v$, where $\varrho\in\Theta_v$ is an arbitrary element (the resulting measure can be shown to be independent of the choice of $\varrho$). Moreover, for $S\in\GG_s(d,k)$ we let
	$$
	\Theta_S=\{\varrho\in\SO(d+1):\varrho E_k=S\}
	$$
	be the set of rotations that rotate $E_k$ to $S$. Similarly as above, we define the invariant probability measure $\nu_S$ on $\Theta_S$.
	
	Using Theorem \ref{thm:WeightFaceIntersection}, the rotation invariance of $T_{n-d+k,d}$, $\kappa$ (which carries over to $\kappa_k$) and $h$ we see that
	\begin{align*}
	\bE h(W_{n,d}^{(k)}) &= {1\over\omega_{k+1}}\int_{\GG_s(d,k)}\int_S\bE h(Z_v(T_{n-d+k,d})\cap S)\,\cH^k(\dint v)\kappa_k(\dint S)\\
	&={1\over\omega_{k+1}}\int_{\GG_s(d,k)}\int_{S}\int_{\Theta_{v}}\bE h(Z_{\varrho^{-1}v}(\varrho^{-1}T_{n-d+k,d})\cap \varrho^{-1}S)\,\nu_v(\dint\varrho)\cH^k(\dint v)\kappa_k(\dint S)\\
	&={1\over\omega_{k+1}}\int_{\GG_s(d,k)}\int_S\int_{\Theta_{v}}\int_{\Theta_{\varrho^{-1}S}}\bE h(Z_e(T_{n-d+k,d})\cap \varsigma^{-1}\varrho^{-1}S)\\
	&\hspace{5cm}\times\nu_{\varrho^{-1}S}(\dint\varsigma)\nu_v(\dint\varrho)\cH^k(\dint v)\kappa_k(\dint S)\\
	&={1\over\omega_{k+1}}\int_{\GG_s(d,k)}\int_S\int_{\Theta_{v}}\int_{\Theta_{\varrho^{-1}S}}\bE h(Z_e(T_{n-d+k,d})\cap E_k)\\
	&\hspace{5cm}\times\nu_{\varrho^{-1}S}(\dint\varsigma)\nu_v(\dint\varrho)\cH^k(\dint v)\kappa_k(\dint S)\\
	&=\bE h(Z_e(T_{n-d+k,d})\cap E_k).
	\end{align*}
	This proves the claim.
\end{proof}

Next, we derive the analogue of Theorem \ref{thm:WeightFaceIntersection} for the typical spherical $k$-face $Z_{n,d}^{(k)}$ of $T_{n,d}$. The result is, however, slightly different, since the typical cell of $T_{n-d+k,d}$ is not necessarily hit by an independent random great subsphere $S\in\GG_s(d,k)$ with distribution $\kappa_k$. Instead, we have to consider the typical cell of the spherical sectional tessellation $T_{n-d+k,S}=T_{n-d+k,d}\cap S$ within the great subsphere $S$.

\begin{theorem}\label{thm:TypicalIntersection}
	Let $d\geq 1$, $k\in\{0,1,\ldots,d\}$ and consider a great hypersphere tessellation of $\SS^d$ with non-degenerate directional distribution $\kappa$ and intensity $n\geq d-k$.
	Let $h:\KK_s(d)\to\RR$ be a non-negative measurable function. Then
	$$
	\bE h(Z_{n,d}^{(k)}) = \int_{\GG_s(d,k)}\bE h(Z_{n-d+k,S})\,\kappa_k(\dint S),
	$$
	where $Z_{n-d+k,S}$	stands for the typical cell of $T_{n-d+k,S}$.
\end{theorem}
\begin{proof}
Using the definition \eqref{eq:DistributionTypicalKface} of $Z_{n,d}^{(k)}$ and the construction of $T_{n,d}$, we see that
\begin{align*}
 \bE h(Z_{n,d}^{(k)}) &= \int_{\GG_s(d,d-1)^n}{1\over C(n,d,k)}\sum_{F\in\cF_k(S_1,\ldots,S_n)}h(F)\,\kappa^n(\dint(S_1,\ldots,S_n))\\
 &= \int_{\GG_s(d,d-1)^n}{1\over C(n,d,k)}\sum_{1\leq i_1<\ldots<i_{d-k}\leq n}\sum_{F\in\cF_k(S_1,\ldots,S_n)\atop F\subseteq S_{i_1}\cap\ldots\cap S_{i_{d-k}}}h(F)\,\kappa^n(\dint(S_1,\ldots,S_n)).
\end{align*}
Since the great hyperspheres $S_1,\ldots,S_n$ generating $T_{n,d}$ are identically distributed, we obtain
\begin{align*}
 \bE h(Z_{n,d}^{(k)}) &= \int_{\GG_s(d,d-1)^{d-k}}{{n\choose d-k}\over C(n,d,k)}\int_{\GG_s(d,d-1)^{n-d+k}}\sum_{F\in\cF_k(T_{n-d+k,S_1\cap\ldots\cap S_{d-k}})}h(F)\\
 &\qquad\qquad\times\kappa^{n-d+k}(\dint(S_{d-k+1},\ldots,S_n))\kappa^{d-k}(\dint(S_1,\ldots,S_{d-k})),
\end{align*}
where
$$
T_{n-d+k,S_1\cap\ldots\cap S_{d-k}}=T_d(S_{d-k+1},\ldots,S_n)\cap(S_1\cap\ldots\cap S_{d-k})
$$
is the spherical sectional random tessellation within $S_1\cap\ldots\cap S_{d-k}$ arising by intersecting $S_{1}\cap\ldots\cap S_{{d-k}}$ with the $n-d+k$ independent great hyperspheres $S_{d-k+1},\ldots,S_n$. Recalling \eqref{eq:Cndk} we see that
$$
{{n\choose d-k}\over C(n,d,k)} = {1\over C(n-d+k,k)}.
$$
Moreover, $C(n-d+k,k)$ is almost surely the number of cells of $T_{n-d+k,S_1\cap\ldots\cap S_{d-k}}$, since $\kappa$ is non-degenerate. So, using once again the definition \eqref{eq:DistributionTypicalKface}, but this time applied to $T_{n-d+k,S_1\cap\ldots\cap S_{d-k}}$, as well as the definition of the measure $\kappa_k$, this leads to
\begin{align*}
 \bE h(Z_{n,d}^{(k)}) &=\int_{\GG_s(d,k)}{1\over C(n-d+k,k)}\int_{\GG_s(d,d-1)^{n-d+k}}\sum_{F\in\cF_k(T_{n-d+k,S})}h(F)\\
 &\qquad\qquad\times\kappa^{n-d+k}(\dint(S_{d-k+1},\ldots,S_n))\kappa_k(\dint S)\\
 &= \int_{\GG_s(d,k)}\bE h(Z_{n-d+k,S})\,\kappa_k(\dint S).
\end{align*}
The proof is thus complete.
\end{proof}

In the isotropic case we have the following corollary, which is the analogue of Corollary \ref{cor:CellSection} for the typical spherical $k$-face. For this, recall the definition of the great subsphere $E_k$.

\begin{corollary}\label{cor:TypicalIso}
	Let $d\geq 1$, $k\in\{0,1,\ldots,d\}$ and consider an isotropic great hypersphere tessellation of $\SS^d$ with intensity $n\geq d-k$. Let $h:\KK_s(d)\to\RR$ be  rotation invariant, non-negative measurable function. Then
	$$
	\bE h(Z_{n,d}^{(k)}) = \bE h(Z_{n-d+k,k}),
	$$
	where $Z_{n-d+k,k}=Z_{n-d+k,E_k}$.
\end{corollary}
\begin{proof}
Using the same notation as in the proof of Corollary \ref{cor:CellSection} we conclude from Theorem \ref{thm:TypicalIntersection} and the assumed isotropy that
\begin{align*}
\bE h(Z_{n,d}^{(k)}) &= \int_{\GG_s(d,k)}\bE h(Z_{n-d+k,S})\,\kappa_k(\dint S)\\
&=\int_{\GG_s(d,k)}\int_{\Theta_S}\bE h(Z_{n-d+k,\varrho^{-1}S})\,\nu_S(\dint\varrho)\kappa_k(\dint S)\\
&= \int_{\GG_s(d,k)}\int_{\Theta_S}\bE h(Z_{n-d+k,E_k})\,\nu_S(\dint\varrho)\kappa_k(\dint S)\\
&=\bE h(Z_{n-d+k,E_k}).
\end{align*}
This completes the argument.
\end{proof}

\subsection{Interpretation of $W_{n,d}^{(k)}$ via size biasing}

Our next result describes the relation between $Z_{n,d}^{(k)}$ and $W_{n,d}^{(k)}$. It shows that the weighted typical spherical $k$-face $W_{n,d}^{(k)}$ is the size biased version of the typical spherical $k$-face $Z_{n,d}^{(k)}$, where size is measured by the $k$-dimensional Hausdorff measure. This can be regarded as the spherical analogue of \cite[Equation (10)]{SchneiderWeightedFaces}, which in turn generalizes \cite[Theorem 10.4.1]{SW}.

\begin{theorem}\label{thm:WeightedvsTypical}
Let $d\geq 1$, $k\in\{0,1,\ldots,d\}$ and consider a great hypersphere tessellation of $\SS^d$ with non-degenerate directional distribution $\kappa$ and intensity $n\geq d-k$.
For non-negative measurable functions $h:\KK_s(d)\to\RR$ one has that
$$
\bE h(W_{n,d}^{(k)}) = {1\over\bE\cH^k(Z_{n,d}^{(k)})}\bE[h(Z_{n,d}^{(k)})\cH^k(Z_{n,d}^{(k)})].
$$
\end{theorem}
\begin{proof}
We start by noting that, since $\kappa$ is non-degenerate, every spherical $k$-face of $T_{n,d}$ arises from the intersection of $d-k$ district great hyperspheres with the remaining elements from $\xi_n$. Since $Z_{n,d}^{(k)}$ is uniformly distributed among the $C(n,d,k)$ spherical $k$-faces of $T_{n,d}$, and since the $n$ great hyperspheres are identically distributed this leads to
\begin{align*}
\bE\cH^k(Z_{n,d}^{(k)}) &= {1\over C(n,d,k)}\bE\sum_{F\in\cF_k(T_{n,d})}\cH^k(F)\\
&={{n\choose d-k}\over C(n,d,k)}\int_{\GG_s(d,d-1)^n}\sum_{F\in\cF_k(T_{n-d+k,S_{1}\cap\ldots\cap S_{{d-k}}})}\cH^k(F)\,\kappa^n(\dint(S_1,\ldots,S_n)).
\end{align*}
Since $T_{n-d+k,S_1\cap\ldots\cap S_{d-k}}$ is a tessellation, the sum over $F\in \cF_k(T_{n-d+k,S_1\cap\ldots\cap S_{d-k}})$ of $\cH^k(F)$ is just $\omega_{k+1}$, independently of $S_{1},\ldots, S_{{d-k}}$. Since $\kappa$ is a probability measure, it follows that
\begin{equation}\label{eq:MeasureZk}
\bE\cH^k(Z_{n,d}^{(k)}) = {\omega_{k+1}\over C(n,d,k)}{n\choose d-k},
\end{equation}
independently of $\kappa$.

We continue by using the definition \eqref{eq:DistributionTypicalKface} of $Z_{n,d}^{(k)}$ to see that
\begin{align*}
&\bE[h(Z_{n,d}^{(k)})\cH^k(Z_{n,d}^{(k)})]\\
&\qquad=\int_{\GG_s(d,d-1)^n}{1\over C(n,d,k)}\sum_{F\in\cF_k(S_1,\ldots,S_n)} h(F)\cH^k(F)\,\kappa^n(\dint(S_1,\ldots,S_n)).
\end{align*}
Splitting the integral, arguing as above and using Fubini's theorem, this is equal to
\begin{align*}
&{{n\choose d-k}\over C(n,d,k)}\int_{\GG_s(d,d-1)^{d-k}}\int_{\GG_s(d,d-1)^{n-d+k}}\sum_{F\in\cF_k(T_{n-d+k,S_1\cap\ldots\cap S_{d-k}})}h(F)\cH^k(F)\\
&\qquad\qquad\times\kappa^{n-d+k}(\dint(S_{d-k+1},\ldots,S_n))\kappa^{d-k}(\dint(S_1,\ldots,S_{d-k}))\\
&={{n\choose d-k}\over C(n,d,k)}\int_{\GG_s(d,d-1)^{d-k}}\int_{\GG_s(d,d-1)^{n-d+k}}\sum_{F\in\cF_k(T_{n-d+k,S_1\cap\ldots\cap S_{d-k}})}\int_Fh(F)\\
&\qquad\qquad\times{\bf 1}\{v\in F\}\,\cH^k(\dint v)\kappa^{n-d+k}(\dint(S_{d-k+1},\ldots,S_n))\kappa^{d-k}(\dint(S_1,\ldots,S_{d-k}))\\
&={{n\choose d-k}\over C(n,d,k)}\int_{\GG_s(d,d-1)^{d-k}}\int_{S_1\cap\ldots\cap S_{d-k}}\int_{\GG_s(d,d-1)^{n-d+k}}\sum_{F\in\cF_k(T_{n-d+k,S_1\cap\ldots\cap S_{d-k}})}h(F)\\
&\qquad\qquad\times{\bf 1}\{v\in F\}\,\kappa^{n-d+k}(\dint(S_{d-k+1},\ldots,S_n))\cH^k(\dint v)\kappa^{d-k}(\dint(S_1,\ldots,S_{d-k})),
\end{align*}
where we recall that
$$
T_{n-d+k,S_1\cap\ldots\cap S_{d-k}}=T_d(S_{d-k+1},\ldots,S_n)\cap(S_1\cap\ldots\cap S_{d-k})
$$
is the spherical sectional random tessellation within $S_1\cap\ldots\cap S_{d-k}$ arising by intersecting $S_{1}\cap\ldots\cap S_{{d-k}}$ with the $n-d+k$ independent great hyperspheres $S_{d-k+1},\ldots,S_n$.

Next, we need to observe that for $\kappa^{d-k}$-almost all $(S_1,\ldots,S_{d-k})$ and $\kappa^{n-d+k}$-almost all $(S_{d-k+1},\ldots,S_n)$ we can only have ${\bf 1}\{v\in F\}=1$ for exactly one element from the set $\cF_k(T_{n-d+k,S_1\cap\ldots\cap S_{d-k}})$, which we denote by $F_v=F_v(S_1,\ldots,S_n)$. This yields
\begin{align*}
&\bE[h(Z_{n,d}^{(k)})\cH^k(Z_{n,d}^{(k)})]\\
&={{n\choose d-k}\over C(n,d,k)}\int_{\GG_s(d,d-1)^{d-k}}\int_{S_1\cap\ldots\cap S_{d-k}}\int_{\GG_s(d,d-1)^{n-d+k}}h(F_v)\\
&\qquad\qquad\qquad\times\kappa^{n-d+k}(\dint(S_{d-k+1},\ldots,S_n))\cH^k(\dint v)\kappa^{d-k}(\dint(S_1,\ldots,S_{d-k})).
\end{align*}
Division by $\bE\cH^k(Z_{n,d}^{(k)})$ leads in view of \eqref{eq:MeasureZk} and the representations \eqref{eq:Skeleton} and \eqref{eq:Skeleton2} of the $k$-skeleton ${\rm skel}_k(T_{n,d})$ to
\begin{align*}
&{1\over\bE\cH^k(Z_{n,d}^{(k)})}\bE[h(Z_{n,d}^{(k)})\cH^k(Z_{n,d}^{(k)})]\\
&=\int_{\GG_s(d,d-1)^{d-k}}{1\over\omega_{k+1}}\int_{S_1\cap\ldots\cap S_{d-k}}\int_{\GG_s(d,d-1)^{n-d+k}}h(F_v)\\
&\qquad\qquad\times\kappa^{n-d+k}(\dint(S_{d-k+1},\ldots,S_n))\cH^k(\dint v)\kappa^{d-k}(\dint(S_1,\ldots,S_{d-k}))\\
&=\int_{\GG_s(d,d-1)^n}{1\over\cH^k({\rm skel}_k(T_{n,d}))}\int_{{\rm skel}_k(T_{n,d})}h(F_v)\,\cH^k(\dint v)\kappa^n(\dint(S_1,\ldots,S_n))\\
&=\bE h(W_{n,d}^{(k)}),
\end{align*}
where in the last step we used the definition \eqref{eq:DistributionWeightedTypicalKface} of $W_{n,d}^{(k)}$. This completes the argument.
\end{proof}

\section{Directional distributions and spherical faces with given directions}\label{sec:DirDistr}

In this section we consider the distribution of the direction of the typical and the weighted typical spherical $k$-face of $T_{n,d}$.  While these distributions are fundamentally different in the Euclidean case (see \cite{HugSchneiderFacesDirection}), we will see the new phenomenon that in the spherical case these distributions coincide. The reason behind this behaviour is that the sum of the weights of all spherical $k$-faces of the tessellation lying in a common $k$-dimensional subsphere is some constant, which is independent of the given subsphere. Since this is true for the sum of the constant weights $1$ and also for the sum of the $k$-dimensional Hausdorff measures, the directional distributions coincide. This is in contrast to the Euclidean case, where instead of the sum over all spherical $k$-faces lying in a common $k$-dimensional subspace one considers the corresponding intensities. These intensities strongly depend on the direction of the subspace, which leads to different results, see \cite{HugSchneiderFacesDirection} and \cite[Chapter 10.3]{SW}.

For $K\in\KK_s(d)$ with $\dim\lin(K)=k+1$ for some $k\in\{0,1,\ldots,d-1\}$ (here, $\lin(\,\cdot\,)$ stands for the linear hull of the argument set taken with respect to $\RRd1$) we write
$$
D(K)=\lin(K)\cap\SS^{d}\in\GG_s(d,k)
$$
for the \textit{direction} of a $k$-dimensional spherical convex set $K$. Moreover, recall the definition of the measure $\kappa_k$ from \eqref{eq:Kappak}.

\begin{theorem}\label{thm:DirectionalDistribution}
Let $d\geq 1$, $k\in\{0,1,\ldots,d-1\}$ and consider a great hypersphere tessellation of $\SS^d$ with non-degenerate directional distribution $\kappa$ and intensity $n\geq d-k$.
	For $C\in\cB(\GG_s(d,k))$ one has that
	$$
	\bP(D(Z_{n,d}^{(k)})\in C)=\bP(D(W_{n,d}^{(k)})\in C)=\kappa_k(C).
	$$
\end{theorem}
\begin{proof}
We use Theorem \ref{thm:WeightFaceIntersection} with $h(P)={\bf 1}\{D(P)\in C\}$ to see that
\begin{align*}
\bP(D(W_{n,d}^{(k)})\in C) = {1\over\omega_{k+1}}\int_{\GG_s(d,k)}\int_S\bP(D(Z_v(T_{n-d+k,d})\cap S)\in C)\,\cH^k(\dint v)\kappa_k(\dint S).
\end{align*}
Since $D(Z_v(T_{n-d+k})\cap S)=S$ this yields
\begin{align*}
\bP(D(W_{n,d}^{(k)})\in C) = {1\over\omega_{k+1}}\int_{\GG_s(d,k)}\int_S{\bf 1}\{S\in C\}\,\cH^k(\dint v)\kappa_k(\dint S)=\kappa_k(C).
\end{align*}
Similarly, for the typical spherical $k$-face $Z^{(k)}$ we use Theorem \ref{thm:TypicalIntersection} to see that
\begin{align*}
\bP(D(Z_{n,d}^{(k)})\in C) &= \int_{\GG_s(d,k)}\bP(D(Z_{n-d+k,S})\in C)\,\kappa_k(\dint S)\\
& = \int_{\GG_s(d,k)}{\bf 1}\{S\in C\}\,\kappa_k(\dint S)=\kappa_k(C),
\end{align*}
since $D(Z_{n-d+k,S})=S$. This completes the argument.
\end{proof}

Since $\KK_s(d)$ is a Polish space (with respect to the topology generated by the spherical Hausdorff distance) the regular conditional distribution of $W_{n,d}^{(k)}$, given the direction $D(W_{n,d}^{(k)})=S$ is well defined for $S\in\GG_s(d,k)$. We will denote this conditional distribution by
$$
\bP(W_{n,d}^{(k)}\in A\,|\,D(W_{n,d}^{(k)})=S),\qquad A\in\cB(\KK_s(d)),\ S\in\GG_s(d,k).
$$
The next result yields an interpretation of this distribution as the distribution of the cell in the spherical sectional tessellation $T_{n-d+k,S}$ that contains a uniform random point.

\begin{corollary}
	Let $d\geq 1$, $k\in\{0,1,\ldots,d-1\}$ and consider a great hypersphere tessellation of $\SS^d$ with non-degenerate directional distribution $\kappa$ and intensity $n\geq d-k$.
For $\kappa_k$-almost all $S\in\GG_s(d,k)$ and $A\in\cB(\KK_s(d))$ one has that
$$
\bP(W_{n,d}^{(k)}\in A\,|\,D(W_{n,d}^{(k)})=S) = \bP(Z_U(T_{n-d+k,S})\in A),
$$
where $U$ is a uniform random point in $S$, which is independent of $T_{n-d+k,S}$.
\end{corollary}
\begin{proof}
From Theorem \ref{thm:WeightFaceIntersection} it follows that, for any $C\in\cB(\GG_s(d,k))$,
\begin{equation}\label{eq:ProofCond1}
\begin{split}
&\bP(W_{n,d}^{(k)}\in A,D(W_{n,d}^{(k)})\in C)\\
&= {1\over\omega_{k+1}}\int_{\GG_s(d,k)}\int_S\bP(Z_v(T_{n-d+k,S})\in A,D(Z_v(T_{n-d+k,S}))\in C)\,\cH^k(\dint v)\kappa_k(\dint S)\\
&= {1\over\omega_{k+1}}\int_{C}\int_S\bP(Z_v(T_{n-d+k,S})\in A)\,\cH^k(\dint v)\kappa_k(\dint S),
\end{split}
\end{equation}
where we used that $D(Z_v(T_{n-d+k,S}))=S$. Moreover, since by Theorem \ref{thm:DirectionalDistribution} the distribution of $D(W_{n,d}^{(k)})$ is given by $\kappa_k$, we have that
\begin{align}\label{eq:ProofCond2}
\bP(W_{n,d}^{(k)}\in A,D(W_{n,d}^{(k)})\in C) &= \int_C\bP(W_{n,d}^{(k)}\in A\,|\,D(W_{n,d}^{(k)})=S)\,\kappa_k(\dint S)
\end{align}
Combining \eqref{eq:ProofCond1} and \eqref{eq:ProofCond2} yields that, for $\kappa_k$- almost all $S\in\GG_s(d,k)$,
\begin{align*}
\bP(W_{n,d}^{(k)}\in A|D(W_{n,d}^{(k)})=S) &= {1\over\omega_{k+1}}\int_S\bP(Z_v(T_{n-d+k,S})\in A)\,\cH^k(\dint x)\\
&=\bP(Z_U(T_{n-d+k,S})\in A),
\end{align*}
where the last step follows from the definition and the independence of $U$.
\end{proof}

A similar result also holds for the typical spherical $k$-face with a given direction. As above, we denote this distribution by
$$
\bP(Z_{n,d}^{(k)}\in A\,|\,D(Z_{n,d}^{(k)})=S),\qquad A\in\cB(\KK_s(d)),\ S\in\GG_s(d,k).
$$

\begin{corollary}
Let $d\geq 1$, $k\in\{0,1,\ldots,d-1\}$ and consider a great hypersphere tessellation of $\SS^d$ with non-degenerate directional distribution $\kappa$ and intensity $n\geq d-k$.
For $\kappa_k$-almost all $S\in\GG_s(d,k)$ and $A\in\cB(\KK_s(d))$ one has that
$$
\bP(Z_{n,d}^{(k)}\in A\,|\,D(Z_{n,d}^{(k)})=S) = \bP(Z(T_{n-d+k,S})\in A).
$$
\end{corollary}
\begin{proof}
Using Theorem \ref{thm:DirectionalDistribution} for the typical spherical $k$-face we see that for any $C\in\cB(\GG_s(d,k))$,
\begin{align*}
\bP(Z_{n,d}^{(k)}\in A,D(Z_{n,d}^{(k)})\in C) &= \int_C\bP(Z_{n,d}^{(k)}\in A\,|\,D(Z_{n,d}^{(k)})=S)\,\kappa_k(\dint S)
\end{align*}
and from Theorem \ref{thm:TypicalIntersection} we have that
\begin{align*}
\bP(Z_{n,d}^{(k)}\in A,D(Z_{n,d}^{(k)})\in C) &= \int_{\GG_s(d,k)}\bP(Z(T_{n-d+k,S})\in A,D(Z(T_{n-d+k,S}))\in C)\,\kappa_k(\dint S)\\
&= \int_{C}\bP(Z(T_{n-d+k,S})\in A)\,\kappa_k(\dint S),
\end{align*}
since $D(Z(T_{n-d+k,S}))=S$. Combination of both identities yields that, for $\kappa_k$- almost all $S\in\GG_s(d,k)$,
$$
\bP(Z_{n,d}^{(k)}\in A\,|\,D(Z_{n,d}^{(k)})=S) = \bP(Z(T_{n-d+k,S})\in A).
$$
The argument is thus complete.
\end{proof}

\section{Explicit results in the isotropic case}\label{sec:Iso}

In this section, if not otherwise stated, we assume that the spherical random great hypersphere tessellation $T_{n,d}$ is isotropic, which means that $\kappa$ is the uniform distribution on $\GG_s(d,d-1)$.

\subsection{The $f$-vector of typical and weighted typical spherical faces}

To present an explicit formula for the expected number $\bE f_\ell(W_{n,d}^{(k)})$ of $\ell$-dimensional spherical faces of the weighted typical spherical $k$-face of a great hypersphere tessellation $T_{n,d}$ we need to introduce some notation taken from \cite{kabluchko_poisson_zero}. If $[x^\ell]P(x)$ denotes the coefficient of $x^\ell$ in a polynomial (or, more generally, a Laurent series) $P(x)$, we define for $m\in\{0,1,\ldots,\}$ and $k\in\ZZ$ the quantity
\begin{equation}\label{eq:Aml}
A[m,\ell] = \begin{cases}
[x^\ell]Q_m(x) &: \ell\ \text{even}\\
[x^\ell]\big(\tanh\big({\pi\over 2x}\big)Q_m(x)\big) &: \ell\ \text{odd},m\ \text{even}\\
[x^\ell]\big({\rm cotanh}\big({\pi\over 2x}\big)Q_m(x)\big) &: \ell\ \text{odd},m\ \text{odd}
\end{cases}
\end{equation}
for $m\in\{0,1,2,\ldots\}$ and $\ell\in\ZZ$, where $Q_m(x)$ is defined as $Q_0(x)=Q_1(x)=1$ and
$$
Q_m(x) = (1+(m-1)^2x^2)(1+(m-3)^2x^2)(1+(m-5)^2x^2)\cdots
$$
for $m\in\{2,3,\ldots\}$, and the last factor in this product is either $1+x^2$ or $1+2^2x^2$. Note that $A[m,\ell]=0$ whenever $\ell>m$.
Furthermore, let $B\{m,\ell\}$ be given by
\begin{align}\label{eq:def_B}
B\{m,\ell\} = {1\over(\ell-1)!(m-\ell)!}\int_0^\pi (\sin x)^{\ell-1}x^{m-\ell}\,\dint x
\end{align}
for $m\in\NN$ and $\ell\in\{1,\ldots,m\}$. Following \cite[Equation (3.16)]{kabluchko_poisson_zero} we also put $B\{m,0\} = {\pi^m\over m!}$ and $B\{m,\ell\}=0$ for $m\in\NN$ and $\ell\in\{m+1,m+2,\ldots\}$. This allows us to state the following result, which yields an explicit formula for $\bE f_\ell(W_{n,d}^{(k)})$. Some particular values for small $d$ and $n$ are collected in Appendix \ref{app:FaceNumbers}, see also Figure \ref{fig:FTypicalWeighted}.

\begin{theorem}\label{thm:fVector}
Let $d\geq 1$, $k\in\{0,1,\ldots,d\}$ and consider an isotropic great hypersphere tessellation of $\SS^d$ with intensity $n\geq d+1$. For $\ell\in\{0,1,\ldots,k-1\}$ the expected number of spherical $\ell$-faces of the weighted typical spherical $k$-face is given by
\begin{align*}
\bE f_\ell(W_{n,d}^{(k)}) &= {(n-d+k)!\,\pi^{d-\ell-n}\over (k-\ell)!}\\
&\qquad\times\sum_{s=0}^{\lfloor{\ell\over 2}\rfloor}B\{n-d+k,k-2s\}
(k-2s-1)^2A[k-2s-2,k-\ell-2],
\end{align*}
where the term $0^2 A[-1,-1]$ has to be interpreted as $2/\pi$, if it appears.
\end{theorem}
\begin{proof}
From Corollary \ref{cor:CellSection} (applied twice) it follows that for all $\ell\in \{0,\ldots,k\}$,
$$
\bE f_\ell(W_{n,d}^{(k)}) = \bE f_\ell(Z_e(T_{n-d+k,k})) = \bE f_\ell (W_{n-d+k, k}),
$$
where $Z_e(T_{n-d+k,k})$ is the north pole cell of the isotropic great hypersphere tessellation $T_{n-d+k,k}$ in $\SS^{k}$, and $W_{n-d+k, k}$ is the weighted typical cell of the same tessellation.

To proceed, let $Y_1,\ldots,Y_n$ be independent and uniformly distributed random points on the lower half-sphere $\SS_-^{d}=\{x=(x_1,\ldots,x_{d+1})\in\SS^{d}:x_{d+1}<0\}$. The spherical random polytope $D_{n,d}=\pos(Y_1,\ldots,Y_n)\cap\SS^{d}$ is connected to the weighted cell $W_{n,d}=W_{n,d}^{(d)}$ of the spherical random tessellation $T_{n,d}$ by spherical polarity. In particular, this implies that
$$
\bE f_\ell(W_{n,d}) = \bE f_{d-\ell-1}(D_{n,d})
$$
for all $\ell\in\{0,1,\ldots,d-1\}$, see \cite[Remark 1.8]{KabluchkoThaele_VoronoiSphere}.
Altogether, we have
$$
\bE f_\ell(W_{n,d}^{(k)}) = \bE f_\ell (W_{n-d+k, k}) = \bE f_{k-\ell-1}(D_{n-d+k, k})
$$
for all $\ell\in \{0,1,\ldots, k-1\}$.
The quantities on the right hand side of this identity were explicitly determined in \cite[Theorem 2.2]{kabluchko_poisson_zero} as follows:
\begin{equation}\label{eq:theo:spherical_polytope_f_vector}
\bE f_{k} (D_{n,d}) = \frac{n!\pi^{k+1-n}}{(k+1)!} \sum_{\substack{s=0,1,\ldots \\ d-2s\geq k+1}} B\{n, d-2s\}
(d-2s-1)^2 A[d-2s-2,k-1]
\end{equation}
for all $d\geq 1$, $n\geq d+1$, and all $k\in \{0,\ldots,d-1\}$. The term $0^2 A[-1,-1]$, if it appears,  has to be interpreted as $2/\pi$ according to Remark~2.3 in~\cite{kabluchko_poisson_zero}. Applying this result to $\bE f_{k-\ell-1}(D_{n-d+k, k})$, which requires to replace $d$ by $k$,  $n$ by $n-d+k$ and $k$ by $k-\ell-1$ in~\eqref{eq:theo:spherical_polytope_f_vector}, yields the desired formula.
\end{proof}

\begin{figure}
	\centering
	\includegraphics[width=0.45\columnwidth]{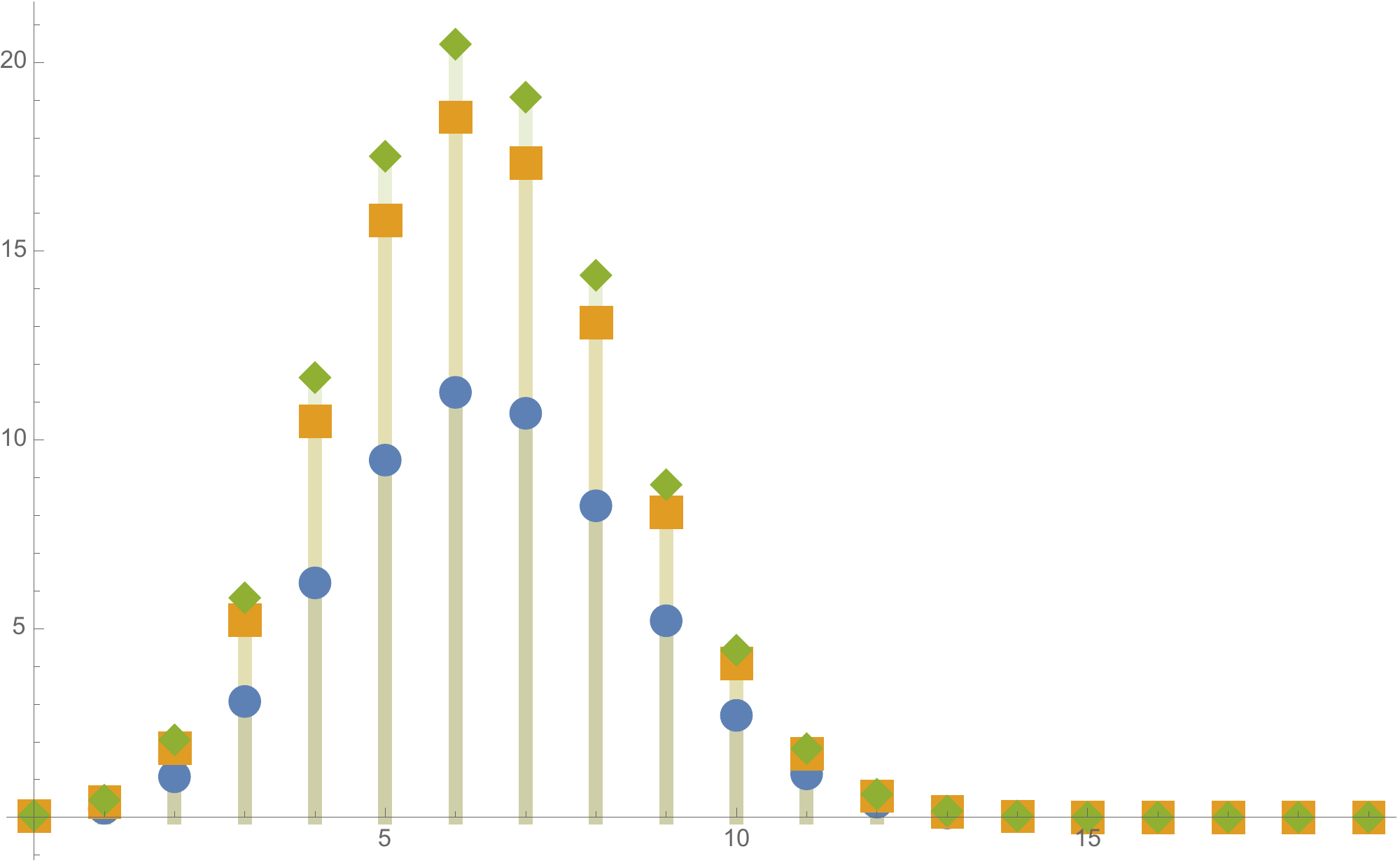}\quad
	\includegraphics[width=0.45\columnwidth]{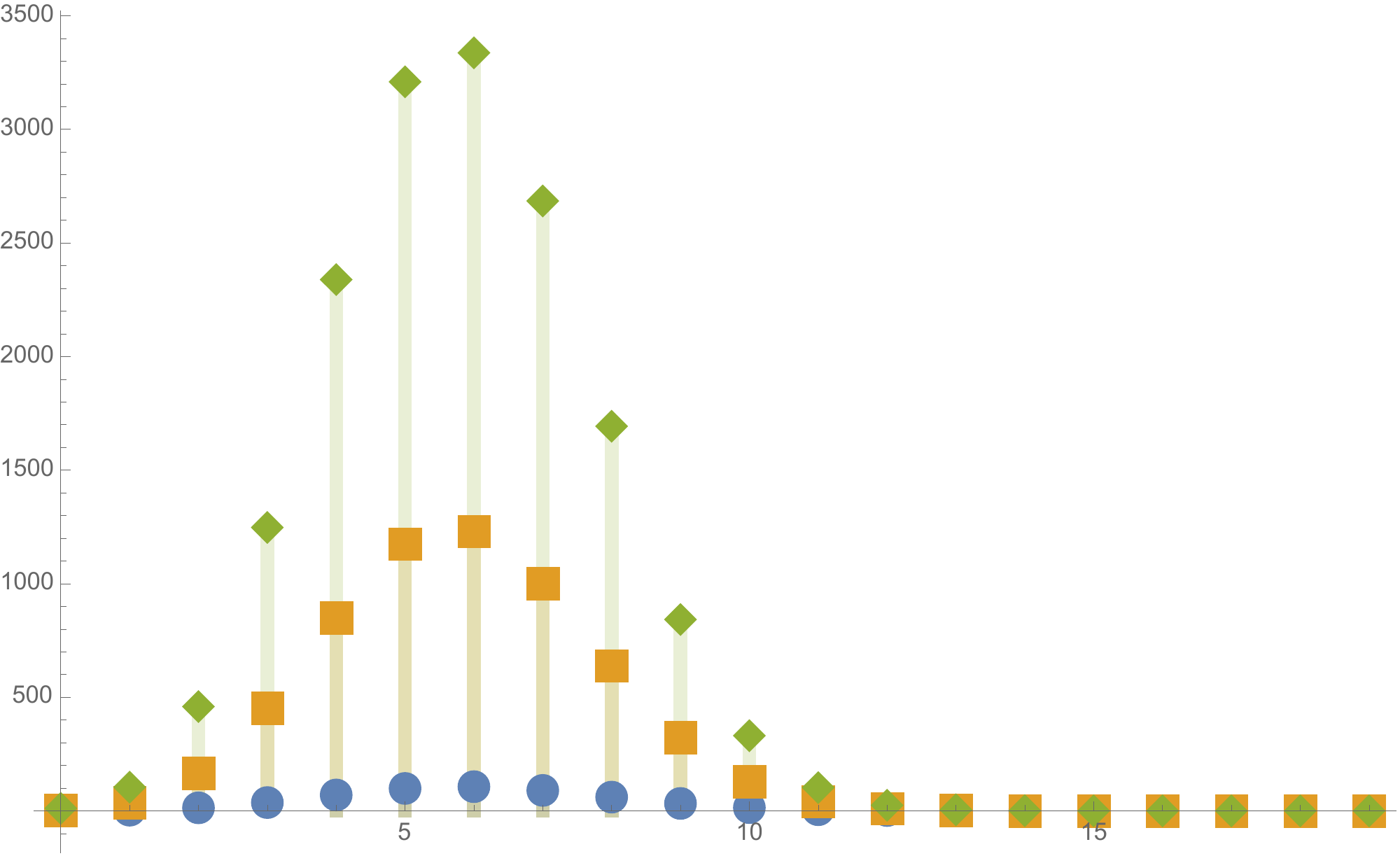}
	\caption{Normalized expected spherical face numbers $\bE f_\ell(\,\cdot\,)/10^7$ of $Z_{n,d}$ (left) and $W_{n,d}$ (right) as a function of $\ell$ for $d=19$ and $n=40$ (blue dots), $n=60$ (orange squares) and $n=80$ (green diamonds).}
	\label{fig:FTypicalWeighted}
\end{figure}

Using Theorem \ref{thm:TypicalIntersection} together with the explicit formula for the $f$-vector of a Schl\"afli random cone from \cite{CoverEfron} (see also \cite{HugSchneiderConicalTessellations}), we can derive an explicit formula for $\bE f_\ell(Z_{n,d}^{(k)})$. In contrast to the previous result, this formula is available for arbitrary non-degenerate directional distributions $\kappa$ and is of purely combinatorial nature.

\begin{theorem}\label{thm:fvectortypical}
Let $d\geq 1$, $k\in\{0,1,\ldots,d\}$ and consider a great hypersphere tessellation of $\SS^d$ with non-degenerate directional distribution $\kappa$ and intensity $n\geq d-k$. For $\ell\in\{0,1,\ldots,k-1\}$ the expected number of $\ell$-faces of the typical spherical $k$-face spherical is given by
$$
\bE f_\ell(Z_{n,d}^{(k)}) = {2^{k-\ell}{n-d+k\choose k-\ell}C(n-d+\ell,\ell)\over C(n-d+k,k)}.
$$
\end{theorem}
\begin{proof}
Theorem \ref{thm:WeightedvsTypical} yields that
$$
\bE f_\ell(Z_{n,d}^{(k)}) = \int_{\GG_s(d,k)}\bE f_\ell(Z(T_{n-d+k,S}))\,\kappa_k(\dint S).
$$
Following the terminology in \cite{HugSchneiderConicalTessellations}, the $k$-dimensional spherical random polytope $Z(T_{n-d+k,S})$ in $S$ is the spherical version of the $(\kappa,n-d+k)$-Schl\"afli random cone in $S$. Its $f$-vector is explicitly known from \cite{CoverEfron} or \cite[Corollary 4.1]{HugSchneiderConicalTessellations}. From this it follows that, independently of $S$,
\begin{align*}
\bE f_\ell(Z(T_{n-d+k,S})) = {2^{k-\ell}{n-d+k\choose k-\ell}C(n-d+\ell,\ell)\over C(n-d+k,k)}.
\end{align*}
Since $\kappa_k$ is a probability measure, the result follows.
\end{proof}

From the particular values for $\bE f_\ell(W_{n,d}^{(k)})$ and $\bE f_\ell(Z_{n,d}^{(k)})$ one can verify numerically that $\bE f_\ell(W_{d+1,d}^{(k)})=\bE f_\ell(Z_{d+1,d}^{(k)})$, whereas $\bE f_\ell(W_{n,d}^{(k)})>\bE f_\ell(Z_{n,d}^{(k)})$ whenever $n\geq d+2$ (see Appendix \ref{app:FaceNumbers}). While we were not able to prove such an inequality in full generality, we show a weaker inequality for the expected vertex numbers. Since the argument requires tools we only develop in the next section, we postpone the proof. We remark that the corresponding inequality for the expected vertex number of the typical and the weighted typical $k$-face of a Poisson hyperplane tessellation in the Euclidean space $\RR^d$ is trivial, since the expected vertex number of the typical $k$-face is a lower bound for that of the weighted typical $k$-face, see \cite[Theorem 2]{SchneiderWeightedFaces}.

\begin{proposition}\label{thm:MonotoneF0}
Let $d\geq 1$, $k\in\{1,\ldots,d\}$ and consider an isotropic great hypersphere tessellation of $\SS^d$ with intensity $n\geq d+1$. Then
$$
\bE f_0(W_{d+1,d}^{(k)})=\bE f_0(Z_{d+1,d}^{(k)})\qquad\text{and}\qquad\bE f_0(W_{n,d}^{(k)})\geq {C(n-d+k,k)\over 2^k C(n-d,k)}\bE f_0(Z_{n,d}^{(k)})
$$
for all $n\geq d+2$.
\end{proposition}

\begin{remark}
	It can be checked that the prefactor ${C(n-d+k,k)\over 2^k C(n-d,k)}$ in Proposition \ref{thm:MonotoneF0} is always strictly less than $1$ and tends to zero, as $n\to\infty$, for any fixed $d$ and $k$.
\end{remark}

\subsection{An Efron-type identity and spherical Querma\ss integrals}

Let $K\subset\RR^d$ be a convex body with volume one. For $n\geq d+1$ let $K_n$ be the convex hull of $n$ uniformly distributed random points in $K$ and write $V(K_n)$ for the volume ($d$-dimensional Lebesgue measure) and $f_0(K_n)$ for the number of vertices of $K_n$. \textit{Efron's identity} relates these two quantities as follows:
$$
\bE V(K_n) = 1-{\bE f_0(K_{n+1})\over n+1},
$$
see \cite[Equation (8.12)]{SW}. While this equality does not admit an extension to relationships for the number of $k$-dimensional faces of $K_{n+1}$ for $k\in\{1,\ldots,d-1\}$, such identities were established in \cite[Section 5]{SchneiderWeightedFaces} for the volume-weighted cell of a stationary and isotropic Poisson hyperplane tessellation in $\RR^d$ (see also \cite{HoermannHugReitznerThaele} for generalizations). In this case the expected number of $k$-dimensional faces is linked to the expected $(d-k)$th Euclidean intrinsic volume or Euclidean Querma\ss integral. For the $\cH^{d}$-weighted cell $W_{n,d}=W_{n,d}^{(d)}$ of an isotropic great hyperspheres tessellation $T_{n,d}$ a similar relationship was established in \cite[p.\ 414]{HugSchneiderConicalTessellations} (see also \cite[Theorem 2.7]{MarynychKabluchkoTemesvariThaele}). It says that, for any $\ell\in\{0,1,\ldots,d\}$ one has that
\begin{align}\label{eq:EfronHugSchneider}
\bE f_{d-\ell}(W_{n,d}) = 2{n\choose\ell}\bE U_\ell(W_{n-\ell,d}).
\end{align}
Here, the functionals $U_\ell$ are the spherical analogues of the Euclidean Querma\ss integrals mentioned above and given by
$$
U_\ell(P) = {1\over 2}\int_{\GG_s(d,d-\ell)}{\bf 1}\{P\cap S\}\,\nu_{d-\ell}(\dint S),
$$
whenever $P\in\PP_s(d)$ is not a great subsphere of $\SS^{d}$ (if $S\in\GG_s(d,k)$ then $U_\ell(S)=1$ if $k-\ell\geq 0$ and even, and $U_\ell(S)=0$ if $k-\ell<0$ or odd).
Our next result generalizes the Efron-type identity \eqref{eq:EfronHugSchneider} to weighted spherical $k$-faces.

\begin{theorem}\label{thm:Efron}
	For $d\geq  1$  consider an isotropic great hypersphere tessellation $T_{n,d}$ of $\SS^d$ with $n\geq d+1$.
	Then, for all $k\in\{0,1,\ldots,d\}$ and $\ell\in\{0,1,\ldots,k\}$ it holds that
	$$
	\bE f_{k-\ell}(W_{n,d}^{(k)}) = 2{n-d+k\choose\ell}\bE U_\ell(W_{n-\ell,d}^{(k)}).
	$$
\end{theorem}
\begin{proof}
	We apply Corollary \ref{cor:CellSection} to the rotation invariant function $h(P)=f_{k-\ell}(P)$. This yields
	$$
	\bE f_{k-\ell}(W_{n,d}^{(k)}) = \bE f_{k-\ell}(Z_e(T_{n-d+k,d})\cap E_k) = \bE f_{k-\ell}(W_{n-d+k,k}),
	$$
	where we recall that $W_{n-d+k-\ell,k}$ stands for the weighted typical cell of $T_{n-d+k,k}$.
	Applying now \eqref{eq:EfronHugSchneider} to this cell leads to
	$$
	\bE f_{k-\ell}(W_{n-d+k,k}) = 2{n-d+k\choose\ell}\bE U_\ell(W_{n-d+k-\ell,k}).
	$$
	Finally, we apply again Corollary \ref{cor:CellSection} to the rotation invariant function $h(P)=U_\ell(P)$ to conclude that
	$$
	\bE U_\ell(W_{n-d+k-\ell,k}) = \bE U_\ell(W_{n-\ell,d}^{(k)}).
	$$
	Putting together these three identities yields the result.
\end{proof}

A combination of Theorem \ref{thm:Efron} with Theorem \ref{thm:fVector} also yields explicit formulas for the expected spherical Querma\ss integrals $\bE U_\ell(W_{n,d}^{(k)})$ of $W_{n,d}^{(k)}$. These quantities can also be determined for the typical spherical $k$-face even for a general directional distribution (for $k=d$ this is known from \cite{HugSchneiderConicalTessellations}). Note that always $\bE U_0(W_{n,d}^{(k)})=1/2$ and $\bE U_0(Z_{n,d}^{(k)})=1/2$ even almost surely without the expectations. Some particular values for small $d$ and $n$ are collected in Appendix \ref{app:Quermass}, see also Figure \ref{fig:UTypicalWeighted}.

\begin{figure}
	\centering
	\includegraphics[width=0.45\columnwidth]{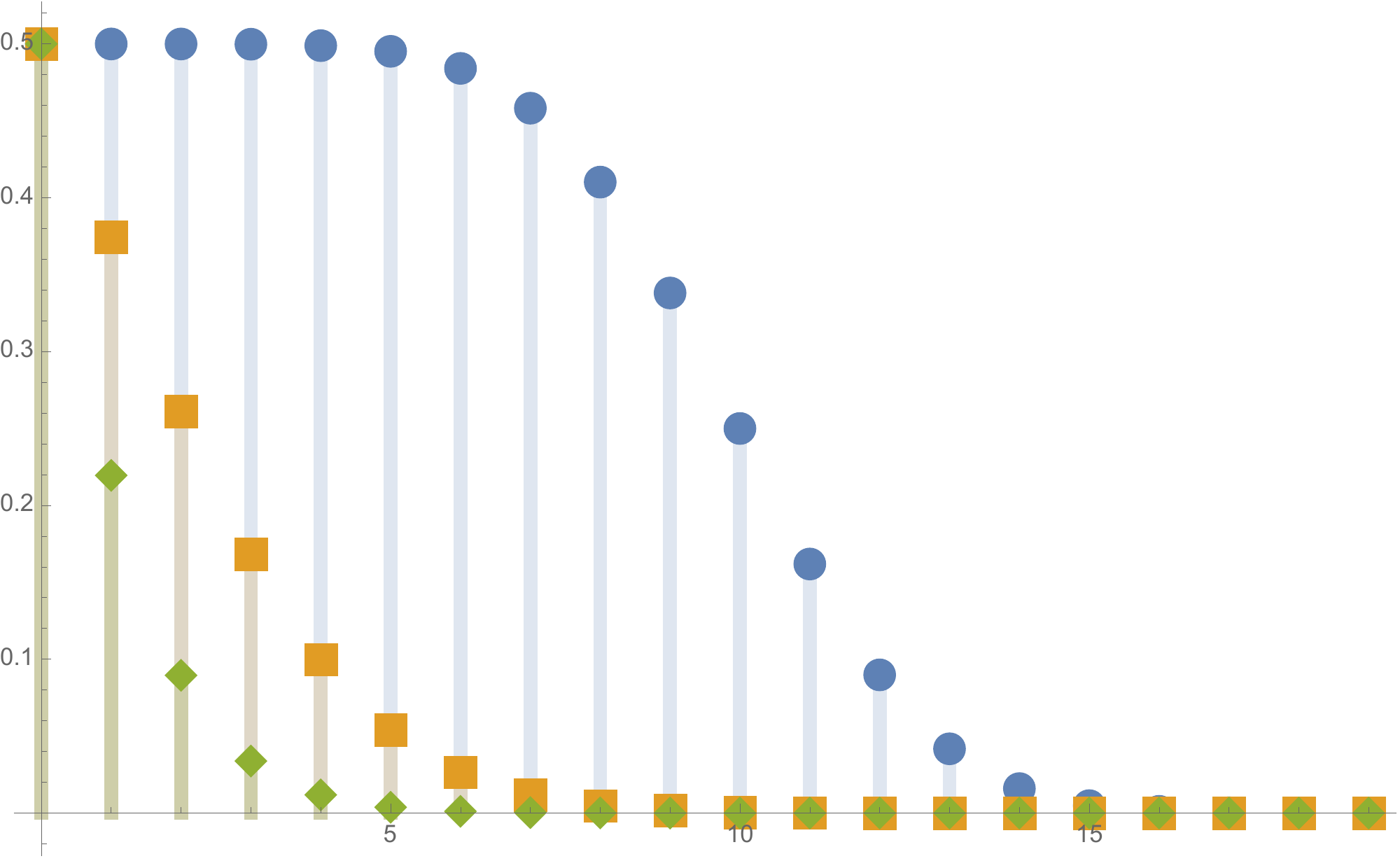}\quad
	\includegraphics[width=0.45\columnwidth]{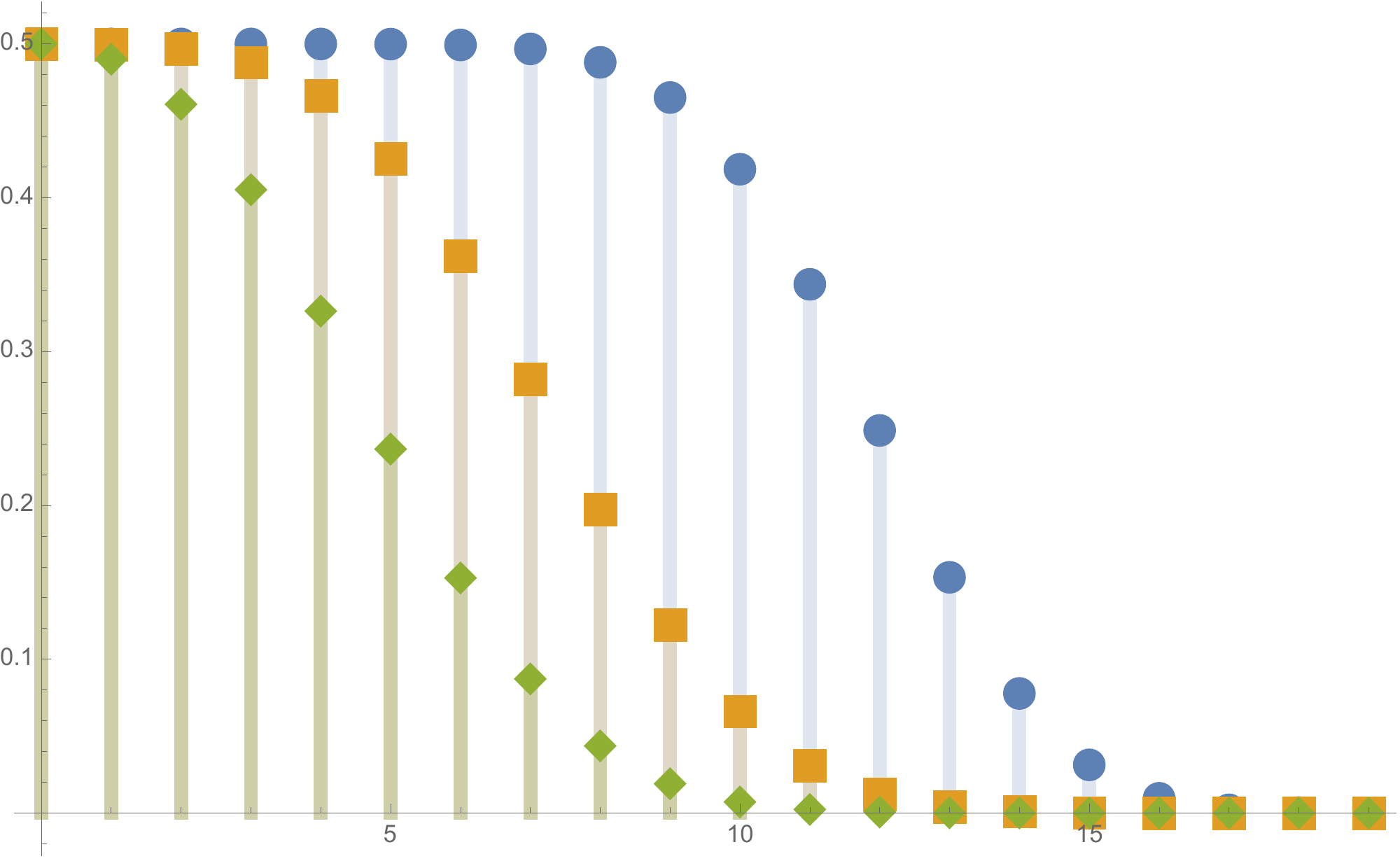}
	\caption{The expected spherical Querma\ss integrals $\bE U_\ell(\,\cdot\,)$ of $Z_{n,d}$ (left) and $W_{n,d}$ (right) as a function of $\ell$ for $d=19$ and $n=20$ (blue dots), $n=40$ (orange squares) and $n=60$ (green diamonds).}
	\label{fig:UTypicalWeighted}
\end{figure}

\begin{corollary}\label{cor:U}
Let $d\geq 1$, $k\in\{0,1,\ldots,d\}$ and consider a great hypersphere tessellation of $\SS^d$ with non-degenerate directional distribution $\kappa$ and intensity $n\geq d-k$. For $\ell\in\{0,1,\ldots,k\}$ one has that
\begin{align}\label{eq:UTypical}
\bE U_\ell(Z_{n,d}^{(k)}) &= {C(n-d+k,k-\ell)\over 2C(n-d+k,k)}.
\end{align}
In the isotropic case, that is, if $\kappa$ is the uniform distribution on $\GG_s(d,d-1)$, and if $n\geq d+1$ then also
\begin{align*}
\bE U_\ell(W_{n,d}^{(k)})&={(n-d+k)!\pi^{d-n-k}\over 2}\\
&\qquad\times\sum_{s=0}^{\lfloor {k-\ell\over 2}\rfloor}B\{n+\ell-d+k,k-2s\}(k-2s-1)^2A[k-2s-2,\ell-2].
\end{align*}
\end{corollary}
\begin{proof}
The result for the typical spherical $k$-face $Z_{n,d}^{(k)}$ can be concluded by combining Corollary \ref{cor:TypicalIso} with \cite[Corollary 4.2]{HugSchneiderConicalTessellations}. In fact,
\begin{align*}
\bE U_\ell(Z_{n,d}^{(k)}) &= \bE U_\ell(Z_{n-d+k,k}) = {C(n-d+k,k-\ell)\over 2C(n-d+k,k)}
\end{align*}
For the weighted typical spherical $k$-face $W_{n,d}^{(k)}$ we use Theorem \ref{thm:Efron} and Theorem \ref{thm:fVector} and obtain
\begin{align*}
\bE U_\ell(W_{n,d}^{(k)}) &= {1\over 2}{n-d+k+\ell\choose\ell}^{-1}\bE f_{k-\ell}(W_{n+\ell,d}^{(k)})\\
&={1\over 2}{n-d+k+\ell\choose\ell}^{-1}{(n+\ell-d+k)!\pi^{d-n-k}\over\ell!}\\
&\qquad\times\sum_{s=0}^{\lfloor {k-\ell\over 2}\rfloor}B\{n+\ell-d+k,k-2s\}(k-2s-1)^2A[k-2s-2,\ell-2]\\
&={(n-d+k)!\pi^{d-n-k}\over 2}\\
&\qquad\times\sum_{s=0}^{\lfloor {k-\ell\over 2}\rfloor}B\{n+\ell-d+k,k-2s\}(k-2s-1)^2A[k-2s-2,\ell-2].
\end{align*}
The proof is thus complete.
\end{proof}

Given the relation between the expected number of vertices and the expected $k$th spherical Querma\ss integral of $W_{n,d}^{(k)}$, we are now prepared to give a proof of Proposition \ref{thm:MonotoneF0}.

\begin{proof}[Proof of Proposition \ref{thm:MonotoneF0}]
	Since the equality for $n=d+1$ is clear, we concentrate on the case that $n\geq d+2$.
	
	Since $W_{n,d}^{(k)}$ is the size ($\cH^k$-) biased version of $Z_{n,d}^{(k)}$, $\cH^k(W_{n,d}^{(k)})$ has the size biased distribution of $\cH^k(Z_{n,d}^{(k)})$ and we thus have the stochastic monotonicity
	$$
	\bP(\cH^k(W_{n,d}^{(k)})\leq x) \leq \bP(\cH^k(Z_{n,d}^{(k)})\leq x)
	$$
	for all $0\leq x\leq\omega_{k+1}$, see \cite[Section 2.2.4]{ArratiaGoldsteinKochman}. In particular, this implies monotonicity for all moments of $\cH^k(W_{n,d}^{(k)})$ and $\cH^k(Z_{n,d}^{(k)})$. Especially
	\begin{align}\label{eq:Mono1}
	\bE U_k(W_{n,d}^{(k)}) \geq \bE U_k(Z_{n,d}^{(k)}),
	\end{align}
	where we used additionally the definition of $U_k$ and its relation to $\cH^k$. Next, the Efron-type identity in Theorem \ref{thm:Efron} yields that
	\begin{align}\label{eq:Mono2}
	\bE U_k(W_{n,d}^{(k)}) = {1\over 2}{n-d+2k\choose k}^{-1} \bE f_0(W_{n+k,d}^{(k)}).
	\end{align}
	Moreover, combining Theorem \ref{thm:fvectortypical} with Corollary \ref{cor:U} leads to
	\begin{align}\label{eq:Mono3}
	\bE U_k(Z_{n,d}^{(k)}) = {C(n-d+2k,k)\over 2^{k+1}{n-d+2k\choose k}C(n-d+k,k)}\bE f_0(Z_{n+k,d}^{(k)}).
	\end{align}
	Plugging \eqref{eq:Mono2} and \eqref{eq:Mono3} into \eqref{eq:Mono1} we arrive at
	\begin{align*}
	\bE f_0(W_{n+k,d}^{(k)}) \geq {C(n-d+2k,k)\over 2^k C(n-d+k,k)}\bE f_0(Z_{n+k,d}^{(k)}),
	\end{align*}
	which is equivalent to
	\begin{align*}
	\bE f_0(W_{n,d}^{(k)}) \geq {C(n-d+k,k)\over 2^k C(n-d,k)}\bE f_0(Z_{n,d}^{(k)}).
	\end{align*}
	This completes the argument.
\end{proof}

In the context of Proposition \ref{thm:MonotoneF0} we also note the following implication, which shows that an inequality between the expected number of $(k-\ell)$-dimensional spherical faces of $W_{n+\ell,d}^{(k)}$ and $Z_{n+\ell,d}^{(k)}$ is stronger than the corresponding inequality for the expected spherical Querma\ss integrals of order $\ell$.

\begin{proposition}\label{prop:MonoUMonof}
	Let $d\geq 1$, $k\in\{1,\ldots,d\}$, $\ell\in\{1,\ldots,k\}$ and consider an isotropic great hypersphere tessellation of $\SS^d$ with intensity $n\geq d+1$. Then
	$$
	\bE f_{k-\ell}(W_{n+\ell,d}^{(k)})\geq \bE f_{k-\ell}(Z_{n+\ell,d}^{(k)})\quad\text{implies that}\quad\bE U_\ell(W_{n,d}^{(k)})\geq \bE U_\ell(Z_{n,d}^{(k)}).
	$$
\end{proposition}
\begin{proof}
	First, we start by noting that Theorem \ref{thm:Efron} yields that
	$$
	\bE U_\ell(W_{n,d}^{(k)}) = {1\over 2}{n-d+k+\ell\choose\ell}^{-1}\bE f_{k-\ell}(W_{n+\ell,d}^{(k)}).
	$$
	Similarly, combining Theorem \ref{thm:fvectortypical} with Corollary \ref{cor:U} we conclude that
	$$
	\bE U_\ell(Z_{n,d}^{(k)}) = {C(n-d+k+\ell)\over 2^{\ell+1}C(n-d+k,k)}{n-d+k+\ell\choose\ell}^{-1}\bE f_{k-\ell}(Z_{n+\ell,d}^{(k)}).
	$$
	Using these two identities, our assumption that $\bE f_{k-\ell}(W_{n+\ell,d}^{(k)})\geq \bE f_{k-\ell}(Z_{n+\ell,d}^{(k)})$ is equivalent to
	$$
	2{n-d+k+\ell\choose\ell}\bE U_\ell(W_{n,d}^{(k)}) \geq {2^{\ell+1}C(n-d+k,k)\over C(n-d+k+\ell,k)}{n-d+k+\ell\choose\ell}\bE U_\ell(Z_{n,d}^{(k)}),
	$$
	which in turn can be rewritten as
	$$
	\bE U_\ell(W_{n,d}^{(k)}) \geq {2^{\ell}C(n-d+k,k)\over C(n-d+k+\ell,k)}\bE U_\ell(Z_{n,d}^{(k)}).
	$$
	From the geometric interpretation of the constant $C(n,d)$ (recall \eqref{eq:Cnd}), it is evident that $C(n-d+k,k)\leq C(n-d+k+\ell,k)$, which implies the result.
\end{proof}

\begin{remark}
	The proof of Proposition \ref{prop:MonoUMonof} actually shows that under the same conditions $\bE f_{k-\ell}(W_{n+\ell,d}^{(k)})\geq \bE f_{k-\ell}(Z_{n+\ell,d}^{(k)})$ implies the  inequality $\bE U_\ell(W_{n,d}^{(k)})\geq {2^{\ell}C(n-d+k,k)\over C(n-d+k+\ell,k)} \bE U_\ell(Z_{n,d}^{(k)})$.
\end{remark}

\subsection{Spherical intrinsic volumes}

From the two formulas presented in the previous section for the spherical Querma\ss integrals also the expected so-called spherical intrinsic volumes of $W_{n,d}^{(k)}$ and $Z_{n,d}^{(k)}$ can be determined explicitly. To introduce them, for a spherical convex set $K\in\KK_s(d)$ and $0\leq r<\pi/2$ we define the $r$-parallel set $K_r=\{x\in\SS^{d}:0<d_g(K,x)\leq r\}$, where $d_g(\,\cdot\,,\,\cdot\,)$ denotes the geodesic distance on $\SS^{d}$ and $d_g(K,x)=\min\{d_g(y,x):y\in K\}$ stands for the geodesic distance of $x$ to $K$. The spherical Steiner formula (a special case of \cite[Theorem 6.5.1]{SW}) says that
$$
\cH^{d}(K_r) = \sum_{\ell=0}^{d-1}v_\ell(K)\,\omega_{\ell+1}\omega_{d-\ell}\int_0^r \cos^\ell\varphi\sin^{d-\ell-1}\varphi\,\dint\varphi.
$$
The coefficients $v_0(K),v_1(K),\ldots,v_{d-1}(K)$ are the \textit{spherical intrinsic volumes} of $K$ and it is convenient to complement them by putting $v_{d}(K)=\omega_{d+1}^{-1}\cH^{d}(K)$ and $v_{-1}(K)=v_d(K^\circ)$. However, since $\sum_{i=-1}^dv_i(K)=1$ by \cite[Theorem 6.5.5]{SW}, $v_{-1}(K)$ is determined once $v_0(K),v_1(K),\ldots,v_d(K)$ are known. If $P\in\PP_s(d)$ is a spherical polytope the spherical intrinsic volumes admit the representation
\begin{align}\label{eq:SphIntVolPolytope}
	v_\ell(P) = {1\over\omega_{\ell+1}}\sum_{F\in\cF_{\ell}(P)}\cH^\ell(F)\gamma(F,P),\qquad\ell\in\{0,1,\ldots,d-1\},
\end{align}
where $\gamma(F,P)$ stands for the external angle of $P$ at $F$, see \cite[page 250]{SW}.

We remark that the spherical intrinsic volumes are closely related to the notion of \textit{conical intrinsic volumes} $\cv_0(C),\cv_1(C),\ldots,\cv_{d+1}(C)$ associated with a convex cone $C\subset\RR^{d+1}$. In fact, one has that
\begin{align}\label{eq:RelationVcoV}
v_\ell(C\cap\SSd) = \cv_{\ell+1}(C)\qquad\text{for}\qquad \ell\in\{0,1,\ldots,d\},
\end{align}
see \cite{AmelunxenLotzMcCoyTropp,McCoyTropp}. In this paper we prefer to work with the spherical intrinsic volumes.

The spherical (or conical) intrinsic volumes are related to the spherical Querma\ss integrals $U_\ell(K)$ by
\begin{align}\label{eq:RelationUV}
v_{d}(K) = U_{d}(K),\qquad v_{d-1}(K) = U_{d-1}(K),\qquad v_\ell(K) = U_{\ell}(K)-U_{\ell+2}(K)
\end{align}
for $\ell\in\{0,1,\ldots,d-2\}$, see \cite[Equation (7)]{HugSchneiderConicalTessellations}. This leads to the following formulas of which the first one is known from \cite[Corollary 4.3]{HugSchneiderConicalTessellations} for $k=d$, the second one is new even in this case. Some particular values for small $d$ and $n$ are collected in Appendix \ref{app:IntVol}, see also Figure \ref{fig:VTypicalWeighted}.

\begin{figure}
\centering
\includegraphics[width=0.45\columnwidth]{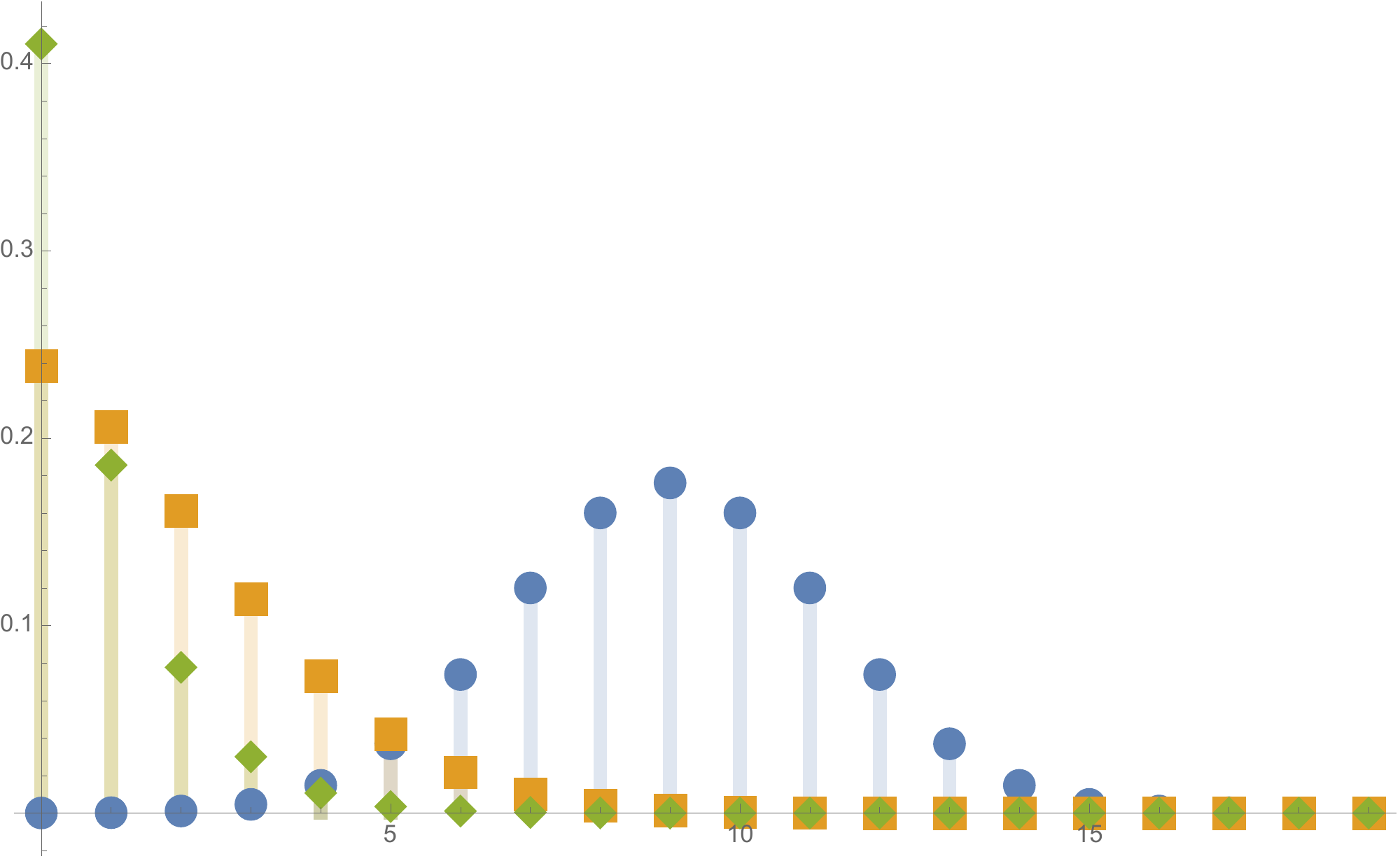}\quad
\includegraphics[width=0.45\columnwidth]{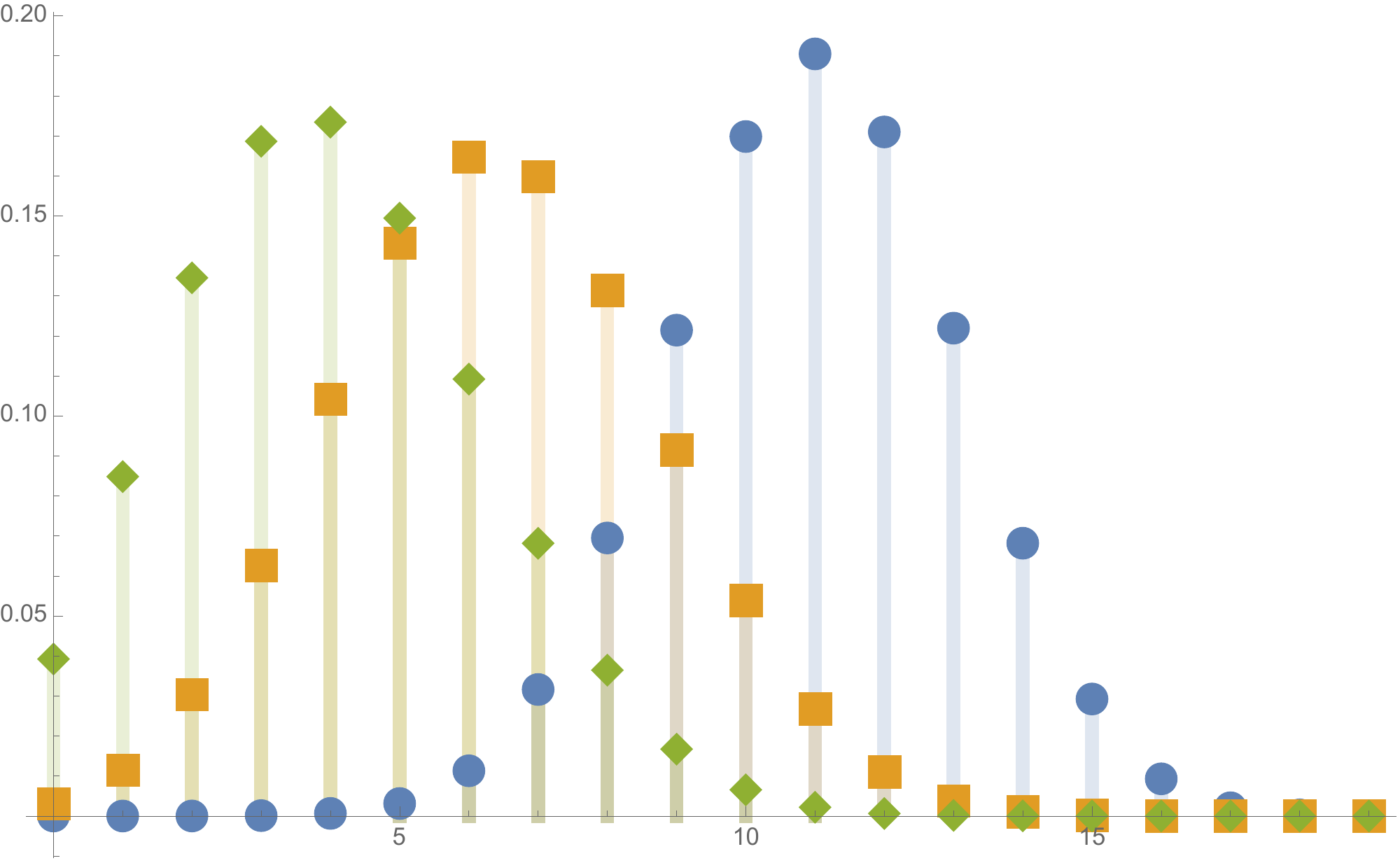}
\caption{The expected spherical intrinsic volumes $\bE v_\ell(\,\cdot\,)$ of $Z_{n,d}$ (left) and $W_{n,d}$ (right) as a function of $\ell$ for $d=19$ and $n=20$ (blue dots), $n=40$ (orange squares) and $n=60$ (green diamonds).}
\label{fig:VTypicalWeighted}
\end{figure}

\begin{corollary}\label{cor:IntVol}
Let $d\geq 1$, $k\in\{0,1,\ldots,d\}$ and consider a great hypersphere tessellation of $\SS^d$ with non-degenerate directional distribution $\kappa$ and intensity $n\geq d-k$. For $\ell\in\{0,1,\ldots,k\}$ one has that
$$
\bE v_\ell(Z_{n,d}^{(k)}) = {{n-d+k\choose k-\ell}\over C(n-d+k,k)}.
$$
In the isotropic case, that is, if $\kappa$ is the uniform distribution on $\GG_s(d,d-1)$, and if $n\geq d+1$ and $\ell\in\{0,1,\ldots,k\}$ then also
$$
\bE v_\ell(W_{n,d}^{(k)}) = {(n-d+k)!\over 2\pi^{n-d+k}}\,B\{n-d+k+\ell,k\}A[k,\ell].
$$
\end{corollary}
\begin{proof}
We apply twice Corollary \ref{cor:U} with $k=d$, and use \eqref{eq:RelationUV} and \eqref{eq:Cnd} to see that
\begin{align*}
\bE v_\ell(Z_{n,d}) &= \bE U_\ell(Z_{n,d})-\bE U_{\ell+2}(Z_{n,d})\\
&={C(n,d-\ell)-C(n,d-\ell-2)\over 2C(n,d)}\\
&={{n-1\choose d-\ell}+{n-1\choose d-\ell-1}\over C(n,d)} = {{n\choose d-\ell}\over C(n,d)}
\end{align*}
(we use here that formally $\bE U_{d+1}(Z_{n,d})=\bE U_{d+2}(Z_{n,d})=0$). Moreover, from Corollary \ref{cor:TypicalIso} we conclude that, for $k\in\{0,1,\ldots,d\}$,
$$
\bE v_\ell(Z_{n,d}^{(k)}) = \bE v_\ell(Z_{n-d+k,k}) = {{n-d+k\choose k-\ell}\over C(n-d+k,k)}.
$$
For the case of weighted faces, we first choose again $k=d$ and make use of the following two recurrence relations from \cite[Equations (1.8) and (1.10)]{kabluchko_poisson_zero}:
\begin{eqnarray}
A[n+2,k] - A[n,k]  &=& (n+1)^2A[n,k-2],\qquad\;\, n\geq 0,k\in\ZZ,\label{eq:RecurrenceA}\\
B\{n,k-2\}-B\{n,k\} &=& (k-1)^2B\{n+2,k\},\qquad n\geq 1,k\geq 2.\label{eq:RecurrenceB}
\end{eqnarray}
Using Corollary \ref{cor:U} and applying \eqref{eq:RecurrenceB} to $\bE U_\ell(W_{n,d})$ we first have that
\begin{align*}
\bE U_\ell(W_{n,d}) = {n!\over 2\pi^n}\sum_{s\geq 0}B\{n+\ell,d-2s\}(A[d-2s,\ell]-A[d-2s-2,\ell]),
\end{align*}
where we use here and below that summation over \textit{all} $s\geq 0$ is possible, since the other terms appearing in the representation for $\bE U_\ell(W_{n,d})$ and $\bE U_{\ell+2}(W_{n,d})$ in Corollary \ref{cor:U} vanish.
Applying now \eqref{eq:RecurrenceA} to $\bE U_{\ell+2}(W_{n,d})$  we obtain
\begin{align*}
\bE U_{\ell+2}(W_{n,d}) = {n!\over 2\pi^n}\sum_{s\geq 0}(B\{n+\ell,d-2s-2\}-B\{n+\ell,d-2s\})A[d-2s-2,\ell].
\end{align*}
Subtracting both expression from each other we get
\begin{align*}
\bE v_\ell(W_{n,d}) &= \bE U_\ell(W_{n,d}) - \bE U_{\ell+2}(W_{n,d})\\
&={n!\over 2\pi^n}\sum_{s\geq 0}\big(B\{n+\ell,d-2s\}A[d-2s,\ell]\\
&\hspace{5cm}-B\{n+\ell,d-2s-2\}A[d-2s-2,\ell]\big)\\
&={n!\over 2\pi^n}\,B\{n+\ell,d\}A[d,\ell].
\end{align*}
Finally, we make use of Corollary \ref{cor:CellSection} to conclude that
\begin{align*}
\bE v_\ell(W_{n,d}^{(k)}) &= \bE v_\ell(W_{n-d+k,k}) = {(n-d+k)!\over 2\pi^{n-d+k}}B\{n-d+k+\ell,k\}A[k,\ell].
\end{align*}
This completes the argument.
\end{proof}

\begin{remark}
Note that Corollary~\ref{cor:IntVol} does not provide a formula for $\bE v_{-1}(W_{n,d}^{(k)})$. To derive a formula for this quantity, observe first that $\bE v_{-1}(W_{n,d}^{(k)}) = \bE v_{-1}(W_{n-d+k,k})$ by Corollary~\ref{cor:CellSection}.
Up to rotations, $W_{n-d+k,k}$ can be identified with the north pole cell in the great hypersphere tessellation $T_{n-d+k,k}$. The polar spherical convex body of this cell can be identified with the spherical convex hull $D_{n-d+k,k}$ of $n-d+k$ random points sampled independently and uniformly on the $k$-dimensional lower half-sphere $\SS^k_-$. Since $v_{-1}(W_{n-d+k,k})$ is defined as $v_d$ of the polar convex body (or just as the solid angle of the corresponding cone), Theorem~2.5 of~\cite{kabluchko_poisson_zero} yields
\begin{align*}
\bE v_{-1}(W_{n,d}^{(k)}) 
&= 
\bE v_{-1}(W_{n-d+k,k}) \\
&= 
\frac{(n-d+k)!}{2\pi^{n-d+k}} \sum_{\substack{m\in \{k+2,\ldots,n-d+k+1\}\\m-k \text{ is even}}} B\{n-d+k+1,m\} (m-1)^2 A[m-2,-1]. 
\end{align*}
\end{remark}

\subsection{The Euclidean case as the limit for $n\to\infty$}

In this section we analyse the asymptotic behaviour of the (suitably rescaled) expected spherical intrinsic volumes of the typical and the weighted typical spherical $k$-face of a great hypersphere tessellation in $\SS^d$, as $n\to\infty$. It will turn out that the limits can be identified with the expected Euclidean intrinsic volumes of the typical and the weighted typical $k$-face of a Poisson hyperplane tessellation in the $d$-dimensional Euclidean space $\RR^d$. We denote by $V_0(K),V_1(K),\ldots,V_d(K)$ the Euclidean intrinsic volumes of a convex set $K\subset\RR^d$, which, similarly to the spherical case, may formally be defined as the coefficients of the Steiner formula, see \cite[Equation (14.5)]{SW}. In particular, if $P\subset\RR^d$ is a polytope, then
\begin{align}\label{eq:EuclIntVol}
V_\ell(P) = \sum_{F\in\cF_{\ell}(P)}\cH^\ell(F)\gamma(F,P),\qquad\ell\in\{0,1,\ldots,d\},
\end{align}
according to \cite[Equation (14.14)]{SW}, where, as in the spherical case, we use the symbol $\cF_{\ell}(P)$ for the set of $\ell$-dimensional face of $P$ and $\gamma(P,F)$ to denote the external angle of $P$ at $F$.

Consider a stationary and isotropic Poisson hyperplane tessellation in $\RR^d$, {$d\geq 1$}, with intensity $\gamma>0$ as in \cite[Chapter 10.3]{SW}. By ${Z}_{\gamma,d}^{(k)}$, $k\in\{0,1,\ldots,d\}$, we denote its typical $k$-face and by $W_{\gamma,d}^{(k)}$ its weighted typical $k$-face, where the weight is given by the $k$-dimensional Hausdorff measure. Formally, $Z_{\gamma,d}^{(k)}$ and $W_{\gamma,d}^{(k)}$ are defined by means of Palm distributions as in \cite{SchneiderWeightedFaces}. It is well known that
\begin{align}\label{eq:VTypicalEuclidean}
\bE V_\ell(Z_{\gamma,d}^{(k)}) = \Big({2\over\gamma}\Big)^\ell\bigg({\Gamma({d+1\over 2})\over\Gamma({d\over 2})}\bigg)^\ell\Gamma\Big({\ell\over 2}+1\Big){k\choose\ell}
\end{align}
for $\ell\in\{0,1,\ldots,k\}$, see \cite[p.\ 490]{SW}. Our next goal is to derive a similar formula for the weighted typical $k$-face, which has not been stated in the existing literature (but see \cite[Theorem 10.4.9]{SW} for the case $d=k=\ell$, which will turn out to be the ``simplest'' possible case).

\begin{theorem}\label{thm:VEuclidean}
For all $d\geq 1$, $k\in\{0,1,\ldots,d\}$,  $\ell\in\{0,1,\ldots,k\}$ and $\gamma>0$  the expected $\ell$-th Euclidean intrinsic volume of the weighted typical $k$-cell of the stationary and isotropic Poisson hyperplane tessellation in $\RR^d$ is given by
\begin{align}\label{eq:VWeightedEuclidean}
\bE V_\ell(W_{\gamma,d}^{(k)}) = \Big({2\pi\over\gamma}\Big)^\ell
\bigg({\Gamma({d+1\over 2})\over\Gamma({d\over 2})}\bigg)^\ell\,{\Gamma({\ell\over 2}+1)\over\ell!}A[k,\ell].
\end{align}
\end{theorem}
\begin{proof}
Using \cite[Theorem 1]{SchneiderWeightedFaces} together with the Efron-type identity for the weighted typical cell from \cite[Section 5]{SchneiderWeightedFaces} we conclude that
\begin{align}\label{eq:Poisson1}
\bE f_{k-\ell}(W_{\gamma,d}^{(k)}) = \kappa_\ell\Big({\gamma\kappa_{d-1}\over d\kappa_d}\Big)^\ell\bE V_\ell(W_{\gamma,d}^{(k)})
\end{align}
for $\ell\in\{0,1,\ldots,k\}$, where $\kappa_\ell={\pi^{\ell/2}\over\Gamma(1+\ell/2)}$ is the volume of the $\ell$-dimensional Euclidean unit ball. On the other hand, in \cite[Theorem 1.1]{kabluchko_poisson_zero} the values for $\bE f_{k-\ell}(W_{\gamma,d}^{(d)})$ have been computed explicitly in terms of the constants $A[m,\ell]$ defined at \eqref{eq:Aml}. In combination with \cite[Theorem 1]{SchneiderWeightedFaces}  again this leads to
\begin{align}\label{eq:Poisson2}
\bE f_{k-\ell}(W_{\gamma,d}^{(k)}) = {\pi^\ell\over\ell!}A[k,\ell].
\end{align}
Plugging \eqref{eq:Poisson2} into \eqref{eq:Poisson1} we conclude that
\begin{align*}
\bE V_\ell(W_{\gamma,d}^{(k)}) = \Big({d\kappa_d\over\gamma\kappa_{d-1}}\Big)^\ell{\pi^\ell\over\ell!\,\kappa_\ell}\,A[k,\ell].
\end{align*}
Using the definition of $\kappa_\ell$ and simplifying the resulting constants, the result follows.
\end{proof}

\begin{remark}
In the special case when $k=\ell=d$, taking into account that by \cite[Proposition 1.2]{kabluchko_poisson_zero}, $A[d,d]={(d!)^2\over 2^d\Gamma({d\over 2}+1)^2}$, one can easily verify that
$$
\bE V_{d}(W_{\gamma, d}^{(d)})=d!\kappa_d\Big({d\kappa_d\over 2\gamma\kappa_{d-1}}\Big)^d
$$
in accordance with \cite[Theorem 10.4.9]{SW}. It is also interesting to compare the result of Theorem \ref{thm:VEuclidean} with \eqref{eq:VTypicalEuclidean}.
\end{remark}

We can now present the announced limit relation for the expected spherical intrinsic volumes for the typical and the weighted typical spherical $k$-face. After the proof we explain the geometric reason behind the rescaling with the factor $n^\ell\omega_{\ell+1}$.

\begin{theorem}\label{thm:EuclidAsLimitSphere}
	Let $d\geq 1$, $k\in\{0,\ldots,d\}$, $\ell\in\{0,\ldots,k\}$ and consider an isotropic great hypersphere tessellation of $\SS^d$ with intensity $n\geq d+1$. Then
	\begin{align*}
	\lim_{n\to\infty} n^\ell\omega_{\ell+1}\bE v_\ell(Z_{n,d}^{(k)}) = \bE V_\ell(Z_{\gamma,d}^{(k)})\quad\text{and}\quad\lim_{n\to\infty} n^\ell\omega_{\ell+1}\bE v_\ell(W_{n,d}^{(k)}) = \bE V_\ell(W_{\gamma,d}^{(k)})
	\end{align*}
	with $\gamma={1\over\sqrt{\pi}}{\Gamma({d+1\over 2})\over\Gamma({d\over 2})}$.
\end{theorem}
\begin{proof}
	From Corollary \ref{cor:IntVol} we have that
	$$
	\bE v_\ell(Z_{n,d}^{(k)}) = {{n-d+k\choose k-\ell}\over C(n-d+k,k)}.
	$$
	Now, we need to observe that, as $n\to\infty$, ${n-d+k\choose k-\ell}$ is asymptotically equivalent to ${n^{k-\ell}\over(k-\ell)!}$ and that $C(n-d+k,k)$ is asymptotically equivalent to ${2n^k\over k!}$, recall \eqref{eq:Cnd}. This shows that
	$$
	\lim_{n\to\infty} n^\ell\omega_{\ell+1}\bE v_\ell(Z_{n,d}^{(k)}) = \omega_{\ell+1}{\ell!\over 2}{k\choose \ell}.
	$$
	Using \eqref{eq:VTypicalEuclidean} with the intensity $\gamma$ as in the statement of the proposition we see that
	$$
	\bE V_\ell(Z_{\gamma,d}^{(k)}) = (2\sqrt{\pi})^\ell\Gamma\Big({\ell\over 2}+1\Big){k\choose\ell}.
	$$
	However, from the definition of $\omega_{\ell+1}$ and Legendre's duplication formula for the gamma function we have that
	\begin{align}\label{eq:omega_legendre}
	\omega_{\ell+1}{\ell!\over 2} = \ell\pi^{\ell+1\over 2}\,{\Gamma(\ell)\over \Gamma({\ell+1\over 2})} = \ell\pi^{\ell+1\over 2}\,{2^{\ell-1}\Gamma({\ell\over 2})\over\sqrt{\pi}} = (2\sqrt{\pi})^\ell\Gamma\Big({\ell\over 2}+1\Big).
	\end{align}
	This proves the first claim. The second one follows similarly. In fact, from Theorem \ref{cor:IntVol} we have that
	\begin{align*}
	\lim_{n\to\infty} n^\ell\omega_{\ell+1}\bE v_\ell(W_{n,d}^{(k)}) = \lim_{n\to\infty} n^\ell\omega_{\ell+1}{(n-d+k)!\over 2\pi^{n-d+k}}B\{n-d+k+\ell,k\}A[k,\ell],
	\end{align*}
	and from \cite[page 8]{kabluchko_poisson_zero} it follows that $B\{n-d+k+\ell,k\}$ is asymptotically equivalent to ${\pi^{n-d+k+\ell}\over(n-d+k+\ell)!}$, as $n\to\infty$. Thus,
	\begin{align*}
	\lim_{n\to\infty} n^\ell\omega_{\ell+1}\bE v_\ell(W_{n,d}^{(k)}) = {\omega_{\ell+1}\pi^\ell\over 2}A[k,\ell].
	\end{align*}
	On the other hand, using \eqref{eq:VWeightedEuclidean} with the intensity $\gamma$ as in the statement of the theorem we see that
	\begin{align*}
	\bE V_\ell(Z_{\gamma,d}^{(k)}) = (2\pi^{3/2})^\ell{\Gamma({\ell\over 2}+1)\over\Gamma(\ell+1)} A[k,\ell].
	\end{align*}
Applying once more the definition of $\omega_{\ell+1}$ and Legendre's duplication formula for the gamma function, we finally see that the right hand sides of the two last expressions are identical:
	$$
	(2\pi^{3/2})^\ell{\Gamma({\ell\over 2}+1)\over\Gamma(\ell+1)} = {(2\pi^{3/2})^\ell\over 2}{\Gamma({\ell\over 2})\over\Gamma(\ell)} = {(2\pi^{3/2})^\ell\over 2}{\sqrt{\pi}\over 2^{\ell-1}\Gamma({\ell+1\over 2})} = {\pi^{3\ell+1\over 2}\over\Gamma({\ell+1\over 2})} = {\omega_{\ell+1}\pi^\ell\over 2}.
	$$
	This proves the second claim as well.
\end{proof}

Let us now give a non-rigorous explanation of the rescaling that appeared in Theorem~\ref{thm:EuclidAsLimitSphere}. The idea is that the sphere $\SSd$ , on small scales, is almost flat, and that in a small window, the great hypersphere tessellation looks essentially like the Euclidean Poisson hyperplane tessellation in the appropriate tangent space, which is identified with $\RR^d$.
For simplicity we consider here only the full-dimensional cells in what follows, that is, we put $k=d$. If $n$ is large, then both the typical cell $Z_{n,d}$ and the weighted typical cell $W_{n,d}$ become ``small'' spherical polytopes (this will be made precise also in the next section when we study the statistical dimension). Since the spherical content of $\SS^d$ is $\omega_{d+1}$  and since the total number of cells in the spherical tessellation is $C(n,d) \sim 2n^d/d!$, the number of cells per unit spherical volume is $\sim 2/ (d! \omega_{d+1}) n^{d}$, where we write $\sim$ for asymptotic equivalence as $n\to\infty$. Since ``small'' spherical polytopes are essentially flat, we can multiply $Z_{n,d}$ and $W_{n,d}$ by $n$ to obtain flat polytopes which are close in distribution to the cells $Z_{\gamma, d} = Z_{\gamma,d}^{(d)}$ and $W_{\gamma,d}= W_{\gamma,d}^{(d)}$ of the Euclidean Poisson hyperplane tessellation with a parameter $\gamma$ that has to be chosen so that the following condition is satisfied. The mean number of cells per unit volume in the Euclidean Poisson hyperplane tessellation should match the spherical case, i.e., it should be $2/ (d! \omega_{d+1})$.  According to~\eqref{eq:VTypicalEuclidean} with $\ell=d$, this yields the condition
$$
\Big({2\over\gamma}\Big)^d\bigg({\Gamma({d+1\over 2})\over\Gamma({d\over 2})}\bigg)^d\Gamma\Big({d\over 2}+1\Big)
=
\frac 12 d! \omega_{d+1}.
$$
Using~\eqref{eq:omega_legendre} with $\ell$ replaced by $d$ it is easy to check that the condition is satisfied for the value of $\gamma$ given in Theorem~\ref{thm:EuclidAsLimitSphere}. 
Thus, in the large $n$ limit, the typical and the weighted spherical cells ``look like'' the corresponding Euclidean ones (with the above choice of $\gamma$) divided by $n$. In fact, it is possible to state and prove such results rigorously, see~\cite{KabluchkoTemesvariThaeleCones}.

Now, consider a small spherical polytope $P_n$ (for example, $Z_{n,d}$ or $W_{n,d}$) which is close to $\frac 1n P$, where $P$ is fixed Euclidean polytope. It remains to understand the asymptotics of the spherical intrinsic volumes $v_\ell(P_n)$, as $n\to\infty$. By the representation \eqref{eq:SphIntVolPolytope}  of $v_\ell$ we have that
$$
v_\ell(P_n) = {1\over \omega_{\ell+1}} \sum_{F_n\in \cF_\ell(P_n)} \cH^\ell(F_n)  \gamma(F,P_n).
$$
In the large $n$ limit, $F_n\in\cF_\ell(P_n)$ is approximated by $\frac 1n F$ with some $F\in \cF_{\ell}(P)$, implying that $\cH^\ell(F_n) \sim n^{-\ell} V_\ell (F)$, while  the external angle converges to its Euclidean counterpart. Comparing this to \eqref{eq:EuclIntVol}, it follows that, as $n\to\infty$,
$$
v_{\ell}(P_n) \sim \frac{V_\ell (P)}{n^\ell \omega_{\ell+1}}
$$
for all $\ell \in \{0,\ldots,d\}$.
Specifying this to the case when $P_n$ is either $Z_{n,d}$ or $W_{n,d}$, explains the rescaling  used in Theorem~\ref{thm:EuclidAsLimitSphere}.

\subsection{Statistical dimension of typical and weighted typical spherical faces}

\begin{figure}[t]
\includegraphics[width=0.3\columnwidth]{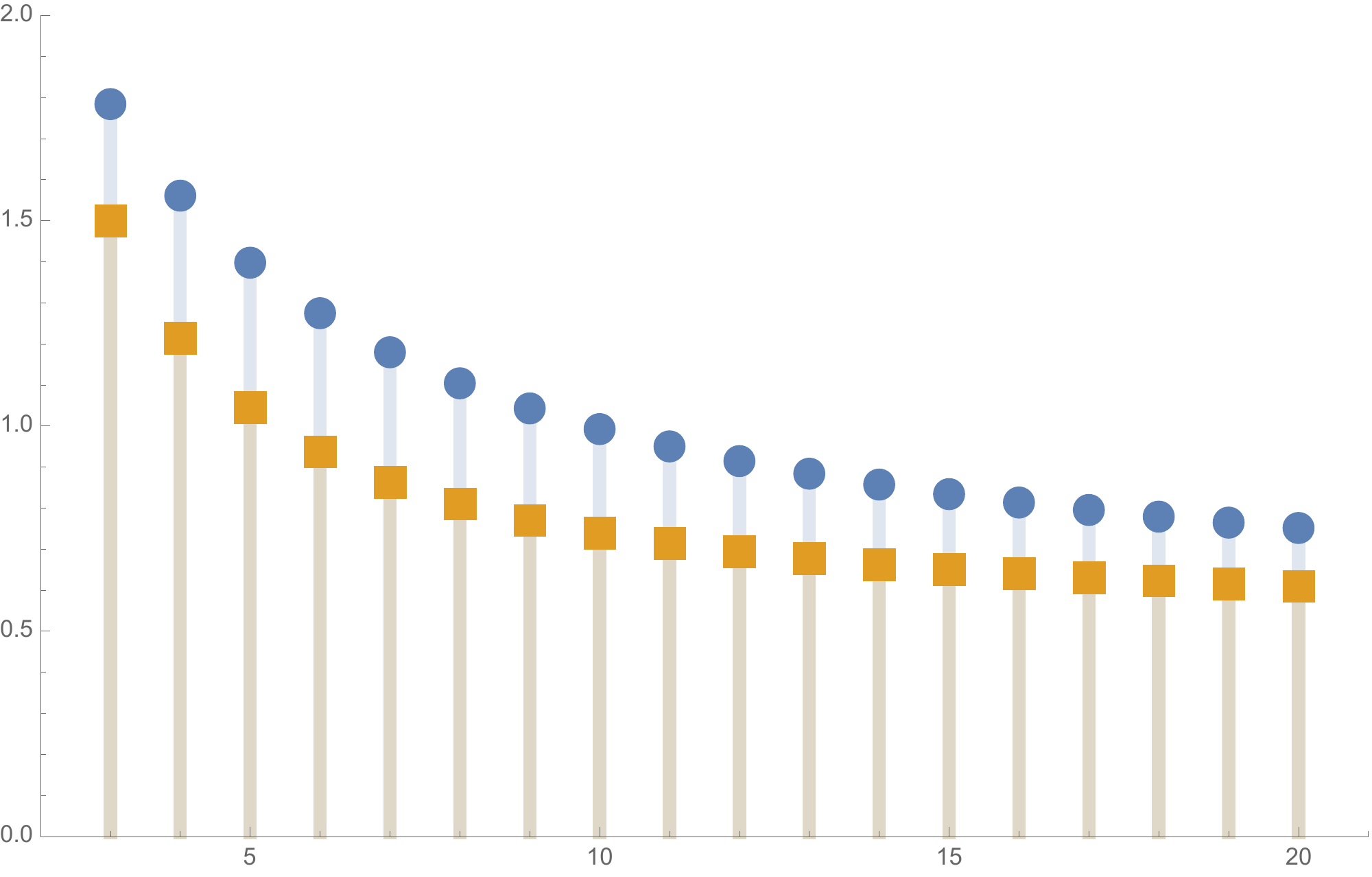}\quad
\includegraphics[width=0.3\columnwidth]{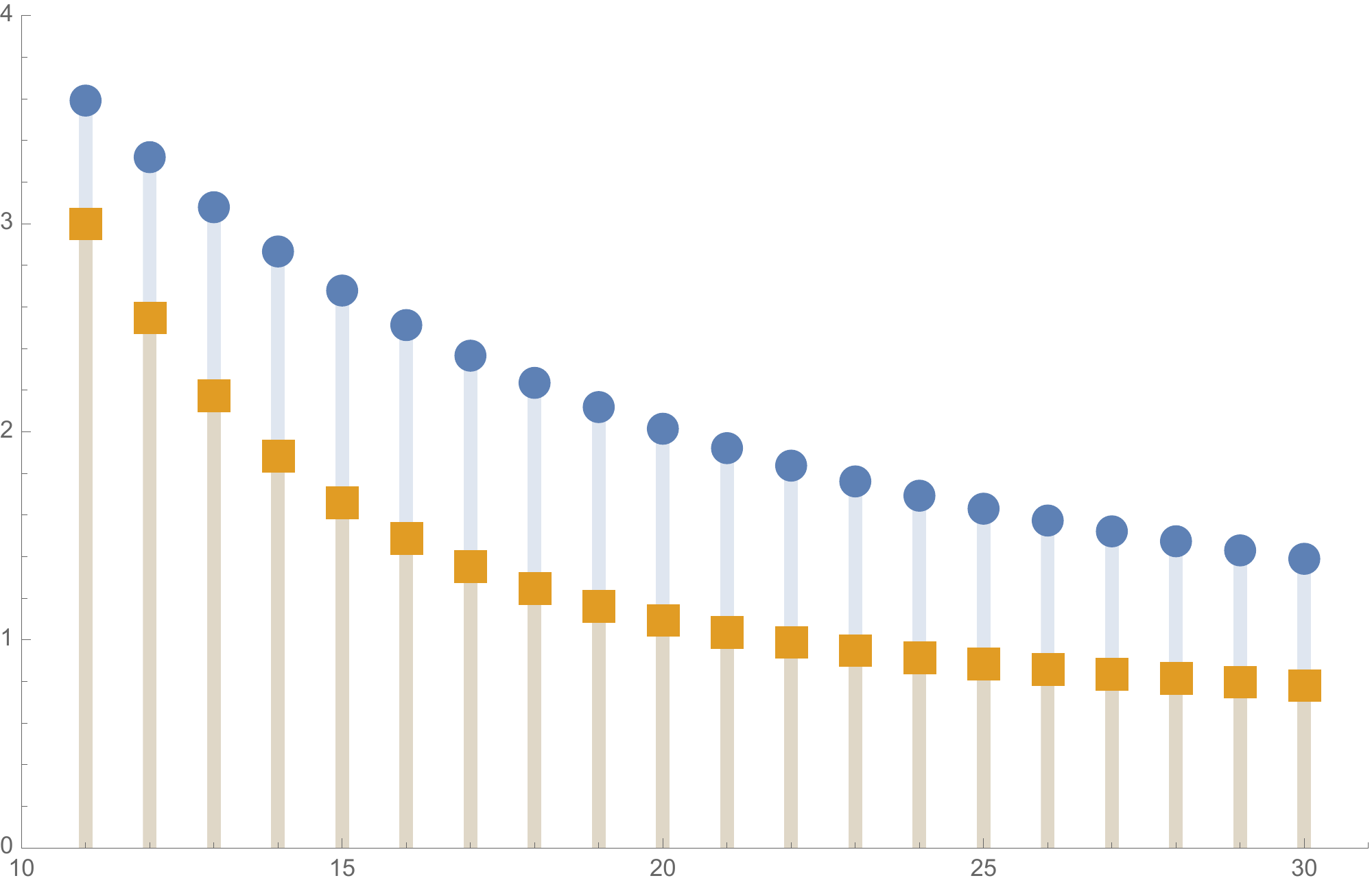}\quad
\includegraphics[width=0.3\columnwidth]{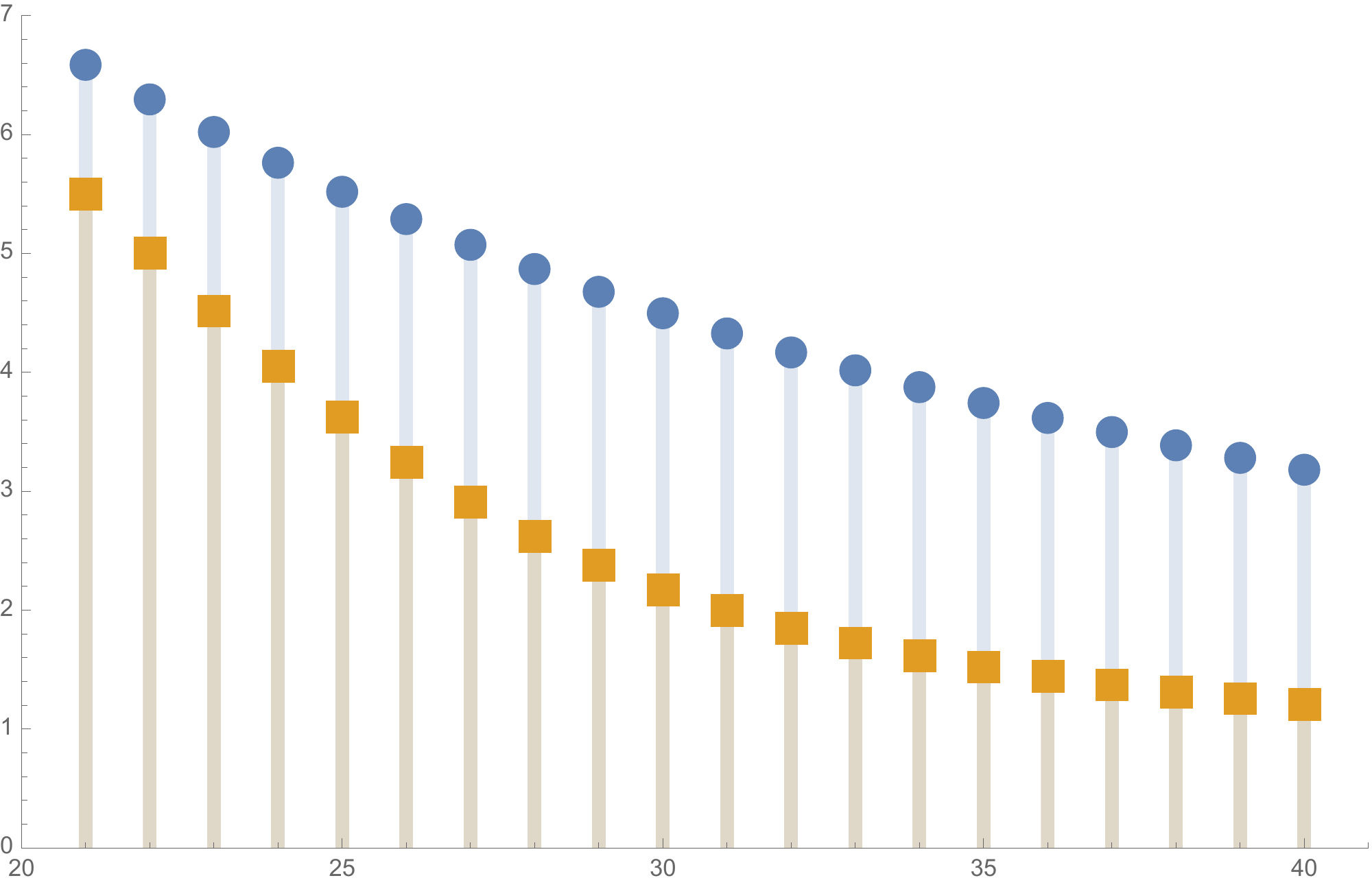}
\caption{Expected statistical dimensions of $\breve{W}_{n,d}^{(k)}$ (blue dots) and $\breve{Z}_{n,d}^{(k)}$ (orange squares) with $d=k=2$ (left panel), $d=10$ and $k=5$ (middle panel) and $d=20$ and $k=10$ (right panel) for $n\in\{d+1,\ldots,d+20\}$.}
\label{fig:statsim}
\end{figure}

The statistical dimension $\delta(C)$ of a convex cone $C\subset\RR^{d+1}$ is a highly important quantity in conical optimization or high dimensional probability. In a sense, it measures the `true' dimension or size of $C$. It is also closely related to the widely used notion of Gaussian width and to concentration phenomena for conical intrinsic volumes, see \cite{AmelunxenBuergisser,AmelunxenLotzMcCoyTropp,McCoyTropp}. By definition, $\delta(C)$ equals $\bE\|\Pi_Cg\|^2$, where $\Pi_C$ denotes the metric projection (or nearest-point map) to $C$ and $g$ is a standard Gaussian random vector in $\RR^{d+1}$. For example, if $C\subset\RR^{d+1}$ is a $k$-dimensional linear subspace, then $\delta(C)=k$. On the other hand, if $C=\pos(u)$ for some $u\in\SS^d$ is a ray, one has $\delta(C)=1/2$. In this section we study the expected statistical dimension of the random cones generated by the typical and the weighted typical spherical $k$-face of the great hypersphere tessellation $T_{n,d}$. Formally, for $k\in\{0,1,\ldots,d\}$ we define the random polyhedral cones
$$
\breve{Z}_{n,d}^{(k)} := \pos(Z_{n,d}^{(k)})\subset\RRd1\qquad\text{and}\qquad\breve{W}_{n,d}^{(k)}:=\pos(W_{n,d}^{(k)})\subset\RRd1.
$$
Using the representation \eqref{eq:RelationVcoV} of conical intrinsic volumes via spherical intrinsic volumes, their expected statistical dimensions can be expressed as
\begin{align*}
\bE\delta(\breve{Z}_{n,d}^{(k)}) = \sum_{j=1}^{k+1}j\,\bE\cv_j(\breve{Z}_{n,d}^{(k)}) = \sum_{j=0}^k(j+1)\bE v_j(Z_{n,d}^{(k)})
\end{align*}
and
\begin{align*}
\bE\delta(\breve{W}_{n,d}^{(k)}) = \sum_{j=1}^{k+1}j\,\bE\cv_j(\breve{W}_{n,d}^{(k)}) = \sum_{j=0}^k(j+1)\bE v_j(W_{n,d}^{(k)}),
\end{align*}
respectively. Unfortunately, there are no simple closed form expressions for $\bE\delta(\breve{Z}_{n,d}^{(k)})$ and $\bE\delta(\breve{W}_{n,d}^{(k)})$. However, we note the following special cases of $\bE\delta(\breve{Z}_{n,d}^{(k)})$ for $d=k\in\{2,3,4,5\}$:
{\small
\begin{alignat*}{2}
\bE\delta(\breve{Z}_{n,2}^{(2)}) &= {n^2+3n+6\over 2n^2-2n+4}, &&\bE\delta(\breve{Z}_{n,3}^{(3)}) = {n^3+3n^2+14n+24\over 2n^3-6n^2+16n}\\
\bE\delta(\breve{Z}_{n,4}^{(4)}) &= \frac{n^4+2 n^3+23 n^2+70 n+120}{2 n^4-12 n^3+46 n^2-36 n+48}, \qquad&&\bE\delta(\breve{Z}_{n,5}^{(5)}) = \frac{n^5+35 n^3+120 n^2+444 n+720}{2 n^5-20 n^4+110 n^3-220 n^2+368 n}.
\end{alignat*}}
The corresponding formulas for $\bE\delta(\breve{W}_{n,d}^{(k)})$ are even more involved.
For example, if $d=k=2$ we claim that
\begin{equation}\label{eq:stat_dim_d=2}
\begin{split}
\bE\delta(\breve{W}_{n,2}^{(2)})
&=
\frac 12 +
\frac{n!}{2\pi^n}
\Bigg(
\sum_{\substack{k\in \{0,\ldots,n\}\\n-k \text{ is even}}} (-1)^{\frac{n-k}{2}} \frac{k+2}{k!}  \pi^k
\\
&\hspace{2cm}+
2 (-1)^{n/2}\ind_{\{n \text{ is even}\}} + \pi (-1)^{\frac{n-1}2}\ind_{\{n \text{ is odd}\}}
\Bigg).
\end{split}
\end{equation}
This formula can be derived as follows. By the definition of the statistical dimension and by Corollary~\ref{cor:IntVol}, we have
\begin{equation}\\
\begin{split}
\bE\delta(\breve{W}_{n,2}^{(2)})
&=
\bE v_0({W}_{n,2}^{(2)}) + 2 \bE v_1({W}_{n,2}^{(2)}) + 3 \bE v_2({W}_{n,2}^{(2)})\\
&=
\frac{n!}{2\pi^n} \left( B\{n,2\} + \pi B\{n+1,2\} + 3 B\{n+2,2\}\right),\label{eq:bE_delta_d=2_def}
\end{split}
\end{equation}
where we also used the values $A[2,0] = 1$, $A[2,1] = \pi/2$ and $A[2,2] = 1$. Using the definition of $B\{n,2\}$ given in~\eqref{eq:def_B}, we obtain
$$
B\{n,2\} = {1\over(n-2)!} \int_0^\pi (\sin x) x^{n-2}\,\dint x
=
-\sum_{\substack{k\in \{0,\ldots,n-2\}\\n-k \text{ is even} }} (-1)^{\frac{n-k}{2}} \frac{\pi^k}{k!}  - (-1)^{n/2} \ind_{\{n \text{ is even}\}}
$$
for $n\geq 2$; see Entries 4,5,6 in Section 1.5.40 of~\cite{brychkov} for the value of the integral. Inserting this formula three times into~\eqref{eq:bE_delta_d=2_def} and performing straightforward  but lengthy transformations, we arrive at~\eqref{eq:stat_dim_d=2}. Similarly, one can obtain an explicit expression for
\begin{align*}
\bE\delta(\breve{W}_{n,3}^{(3)})
&=
\bE v_0({W}_{n,3}^{(3)}) + 2 \bE v_1({W}_{n,3}^{(3)}) + 3 \bE v_2({W}_{n,3}^{(3)}) +  4 \bE v_3({W}_{n,3}^{(3)})\\
&=
\frac{n!}{2\pi^n} \Big( B\{n,3\} + \Big(\frac{4}{\pi }+\frac{4 \pi }{3}\Big) B\{n+1,3\} + 12 B\{n+2,3\} + \frac{32}{\pi} B\{n+3,3\}\Big),
\end{align*}
using the formula
$$
B\{n,3\} = - \sum_{\substack{k\in \{0,\ldots,n-2\}\\ n-k\text{ is even}}}(-1)^{\frac{n-k}{2}}\frac{\pi^k}{k! 2^{n-k}} + \frac{(-1)^{n/2}}{2^{n}}\ind_{\{n \text{ is even}\}},
\quad n\geq 3,
$$
which follows from~\eqref{eq:def_B} and Entry~12 in Section 1.5.40 of~\cite{brychkov}.


Note that since $\bE U_\ell(Z_{n,d}^{(k)})\to 0$ and $\bE U_\ell(W_{n,d}^{(k)})\to 0$ for $\ell\in\{1,\ldots,k\}$ and $n\to\infty$, and since almost surely $U_0(Z_{n,d}^{(k)})=U_0(W_{n,d}^{(k)})=1/2$, the limit relations
$$
\lim_{n\to\infty}\bE\delta(\breve{Z}_{n,d}^{(k)})={1\over 2}\qquad\text{and}\qquad\lim_{n\to\infty}\bE\delta(\breve{W}_{n,d}^{(k)})={1\over 2}
$$
follow from \eqref{eq:RelationUV}. This is consistent with the observation that, as $n\to\infty$, $\breve{Z}_{n,d}^{(k)}$ and $\breve{W}_{n,d}^{(k)}$ asymptotically behave like rays emanating from the origin, whose statistical dimension equals $1/2$.

\subsection{Intersection probabilities for weighted typical cells}

\begin{figure}
	\centering
	\begin{tikzpicture}[]
	\pgftext{\includegraphics[width=250pt]{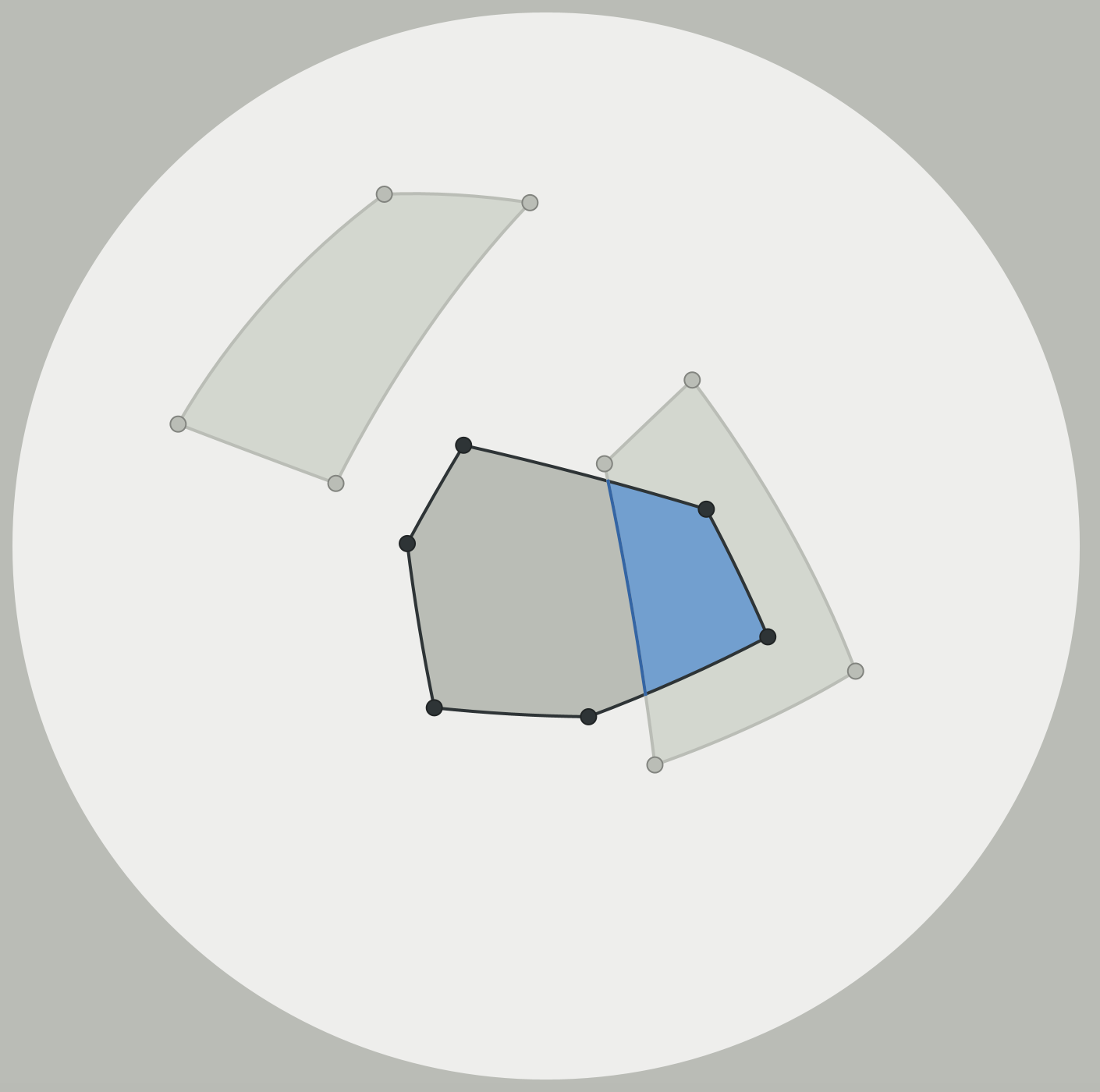}} at (0pt,0pt);
	\node at (-1.6,-1) {$W_{n,d}$};
	\node at (2,1) {$V_{n,d}$};
	\node at (-0.9,3.1) {$V_{n,d}$};
	\end{tikzpicture}
	\caption{Illustration of the two different intersection situations for a realization of $W_{n,d}$ and two different realizations of $V_{m,d}$ with $d=2$.}
\end{figure}

Intersection probabilities for random cones have recently moved into the focus of  attention in stochastic geometry because of their relevance in conical optimization problems, see \cite{AmelunxenBuergisser,AmelunxenLotzMcCoyTropp,McCoyTropp}. In particular, it is of interest in these works to evaluate the probability that a fixed cone and a randomly rotated cone share a common ray. However, we note that randomness enters in this problem only via a random rotation. The natural question now arises whether there are mathematically tractable models for cones having a \textit{random} shape that allow an exact determination of intersection probabilities. In this context, the following question has been studied in \cite{SchneiderKinematicCones}, which we rephrase in our equivalent spherical set-up. For fixed $d\geq 1$ and $n,m\in\NN$ let $P_{n,d}$ be a spherical random polytope with the same distribution as the typical cell $Z_{n,d}$ of $T_{m,d}$ and $Q_{m,d}$ be a spherical random polytope with the same distribution as the typical cell $Z_{m,d}$ of $T_{m,d}$, and assume that $P_{n,d}$ and $Q_{m,d}$ are independent. What is the probability $\bP(P_{n,d}\cap Q_{m,d}\neq\varnothing)$ that $P_{n,d}$ and $Q_{m,d}$ have a non-empty intersection? Using the spherical (or conical) kinematic formula and the explicitly known values for the spherical (or conical) intrinsic volumes of $P_{n,d}$ and $Q_{m,d}$ this probability was explicitly determined in \cite[Theorem 1.4]{SchneiderKinematicCones}. For example, for $d=2$ and $d=3$ one has that
\begin{align*}
\bP(P_{n,2}\cap Q_{m,2}\neq\varnothing) &= \frac{m^2+2 m n-m+n^2-n+2}{\left(m^2-m+2\right) \left(n^2-n+2\right)},\\
\bP(P_{n,3}\cap Q_{m,3}\neq\varnothing) &= \frac{3 (m+n) \left(m^2+2 m n-3 m+n^2-3 n+8\right)}{m \left(m^2-3
	m+8\right) n \left(n^2-3 n+8\right)}.
\end{align*}
Our goal is to complement the result in \cite{SchneiderKinematicCones} by studying the corresponding intersection probability for weighted typical cells. Passing to their conical versions, this adds another tractable model to the question addressed above. However, in contrast to the model studied in \cite{SchneiderKinematicCones} we would like to point out that the intersection probability for weighted typical cells is not just a purely combinatorial quantity.

\begin{figure}[t]
	\centering
	\includegraphics[width=0.3\columnwidth]{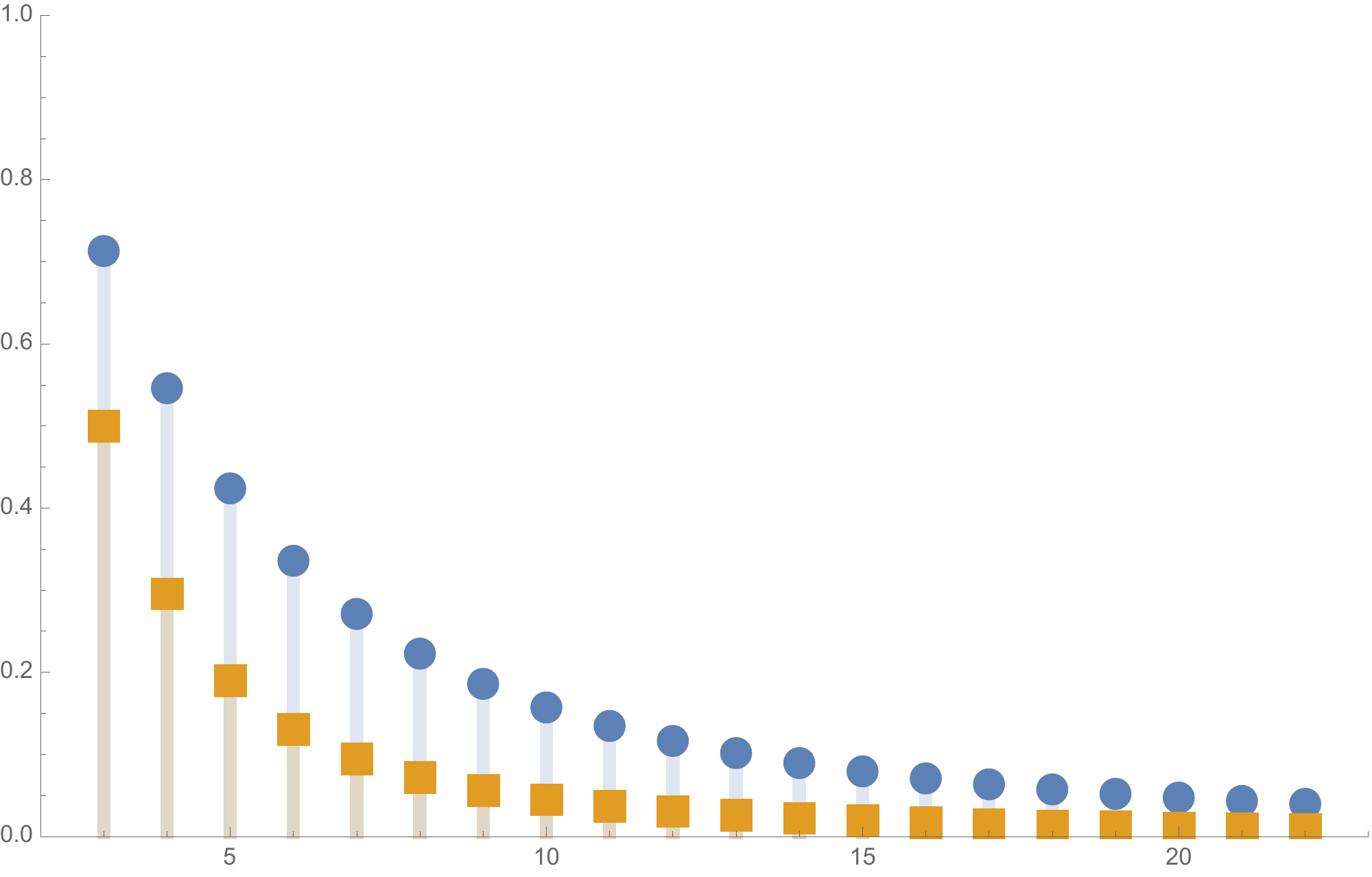}\quad
	\includegraphics[width=0.3\columnwidth]{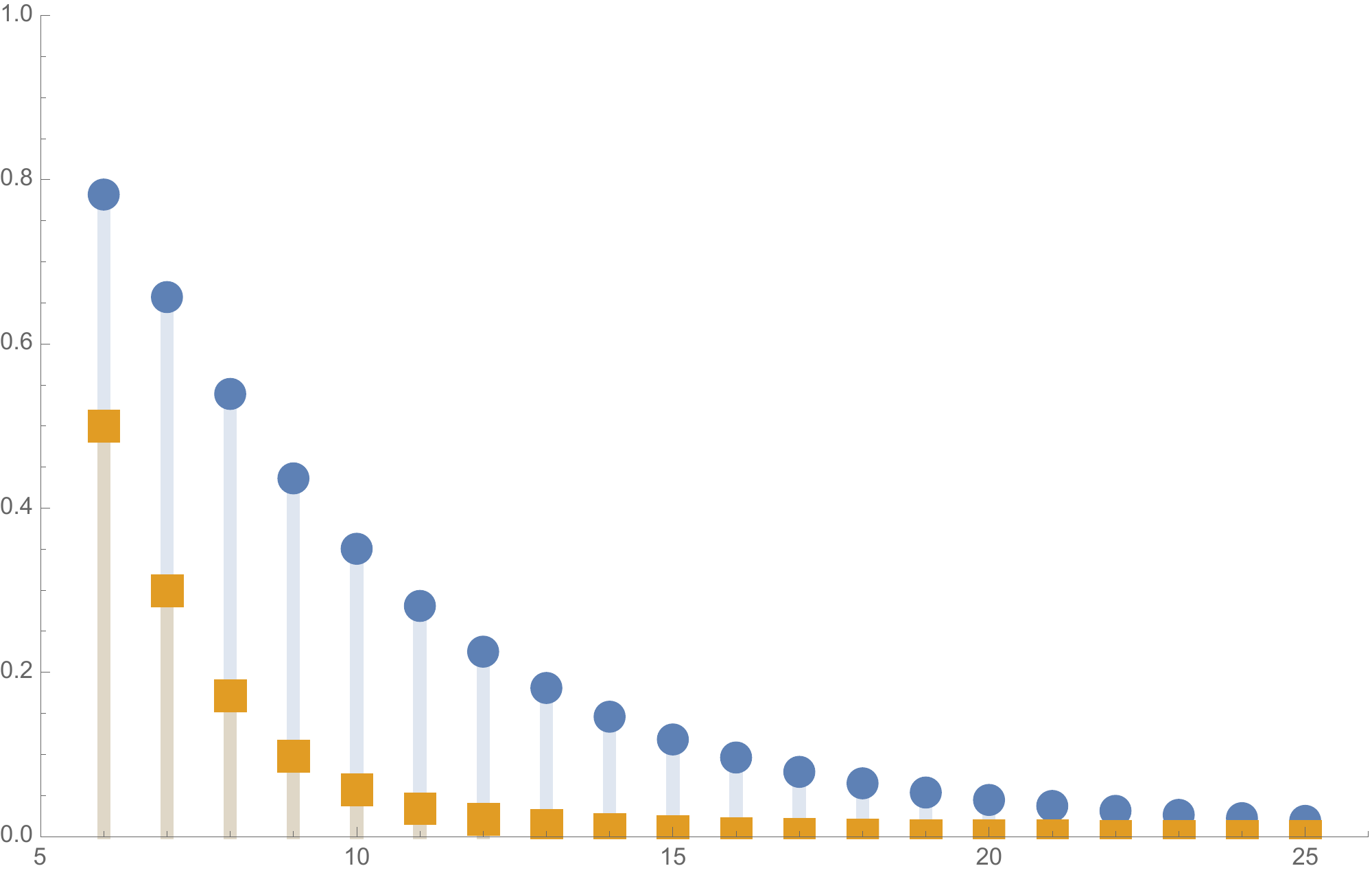}\quad
	\includegraphics[width=0.3\columnwidth]{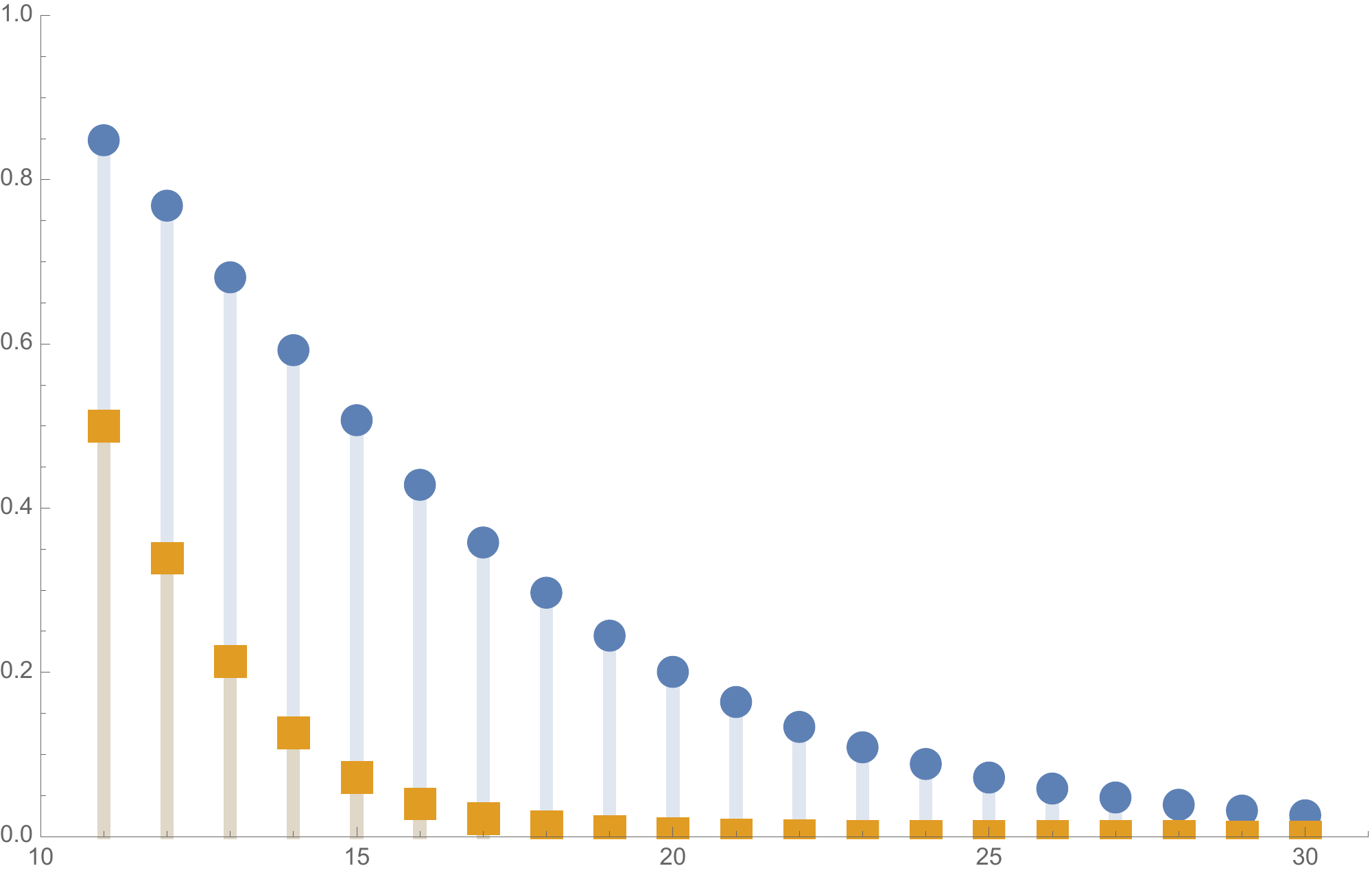}
	\caption{Intersection probabilities $\bP(W_{n,d}\cap V_{n,d}\neq\varnothing)$ and $\bP(P_{n,d}\cap Q_{n,d}\neq\varnothing)$ for the weighted typical cell (blue dots) and typical cell (orange squares) for $d\in\{2,5,10\}$ and $n\in\{d+1,\ldots,d+20\}$.}
	\label{fig:InterPlot}
\end{figure}

\begin{theorem}\label{thm:IntersectionProb}
For $d\geq  1$ and $n,m>d$ consider two independent isotropic great hypersphere tessellations $T_{n,d}$ and $\widetilde{T}_{m,d}$ of $\SS^d$. Let $W_{n,d}$ be the weighted typical cell of $T_{n,d}$ and $V_{m,d}$ be the weighted typical cell of $\widetilde{T}_{m,d}$. Then
$$
\bP(W_{n,d}\cap V_{m,d}\neq\varnothing)\! =\! {n!m!\over 2\pi^{n+m}}\sum_{k=0}^{\lfloor{d\over 2}\rfloor}\sum_{i=2k}^d\!\! B\{n+d-i+2k,d\}B\{m+i,d\}A[d,d-i+2k]A[d,i].
$$
\end{theorem}
\begin{proof}
We use the isotropy assumption and Fubini's theorem to see that
\begin{align*}
\bP(W_{n,d}\cap V_{m,d}\neq\varnothing) &= \bE{\bf 1}\{W_{n,d}\cap V_{m,d}\neq\varnothing\}\\
&=\int_{\SO(d+1)}\bE{\bf 1}\{\varrho W_{n,d}\cap V_{m,d}\neq\varnothing\}\,\nu(\dint \varrho)\\
&=\bE\int_{\SO(d+1)}{\bf 1}\{\varrho W_{n,d}\cap V_{m,d}\neq\varnothing\}\,\nu(\dint \varrho),
\end{align*}
where we denote by $\nu$ the rotation invariant Haar probability measure on $\SO(d+1)$. To the last expression we apply the spherical principal kinematic formula, see \cite[p.\ 261]{SW}. In our case it says that, for almost all given realizations of $W_{n,d}$ and $V_{m,d}$,
$$
\int_{\SO(d+1)}{\bf 1}\{\varrho W_{n,d}\cap V_{m,d}\neq\varnothing\}\,\nu(\dint \varrho) = 2\sum_{k=0}^{\lfloor{d\over 2}\rfloor}\sum_{i=2k}^d v_{d-i+2k}(W_{n,d})\, v_i(V_{m,d}).
$$
The result now follows by taking expectations, using the independence of $W_{n,d}$ and $V_{m,d}$ and finally Corollary \ref{cor:IntVol}.
\end{proof}

Using the previous result it is in principle possible to obtain a fully explicit formula for the intersection probability $\bP(W_{n,d}\cap V_{m,d}\neq\varnothing)$ by combining Corollary \ref{cor:U} with \eqref{eq:RelationUV}. For example, if we denote by $_pF_q(a_1,\ldots,a_p;b_1,\ldots,b_q;z)$ the usual hypergeometric function, then $\bP(W_{n,2}\cap V_{m,2}\neq\varnothing)$ can be expressed as
\begin{align*}
&{\pi^2\over 8(n+1)} \bigg[{\pi^2nm\over m+1} {_1F_2}\Big({m+1\over 2};{3\over 2},{m+3\over 2};-{\pi^2\over 4}\Big){_1F_2}\Big({n+1\over 2};{3\over 2},{n+3\over 2};-{\pi^2\over 4}\Big)\\
&+{4(n+1)\over (n+2)(m+2)}{_1F_2}\Big(1+{m\over 2};{3\over 2},1+{m\over 2};-{\pi^2\over 4}\Big)\Big(n+2-\pi^2{_1F_2}\Big(1+{n\over 2};{3\over 2},2+{n\over 2};-{\pi^2\over 4}\Big)\Big)\\
&+{4(n+1)\over n+2}{_1F_2}\Big(1+{n\over 2};{3\over 2},1+{n\over 2};-{\pi^2\over 4}\Big)\bigg]
\end{align*}
for $m,n\geq 3$. However, since such formulas become rather involved in general, we refrain from presenting them. Instead, we collect some particular values for small $n$, $m$ and $d$ in Appendix \ref{app:IntesectionProb} and compare in Figure \ref{fig:InterPlot} the intersection probabilities $\bP(W_{n,d}\cap V_{n,d}\neq\varnothing)$ for the weighted typical cell with those for the typical cell for $d\in\{2,5,10\}$.

\begin{remark}
	It is also possible to determine the intersection probability $\bP(P\cap W_{n,d}\neq\varnothing)$, where $P\in\PP_s(d)$ is now a fixed spherical polytope. In fact, repeating the proof of Theorem \ref{thm:IntersectionProb}, one shows that for $d\geq 1$ and $n\geq d+1$,
	$$
	\bP(P\cap W_{n,d}\neq\varnothing) = {n!\over\pi^n}\sum_{k=0}^{\lfloor{d\over 2}\rfloor}\sum_{i=2k}^dB\{n+d-2+2k\}A[d,d-i+2k]\,v_i(P).
	$$
\end{remark}

\subsection*{Acknowledgement}

ZK has been supported by the German Research Foundation under Germany's Excellence Strategy  EXC 2044 -- 390685587, Mathematics M\"unster: Dynamics - Geometry - Structure.

\addcontentsline{toc}{section}{References}


\begin{landscape}

\appendix

\section{Expected spherical face numbers}\label{app:FaceNumbers}

{\small Values of $\bE f_0(W_{n,2}^{(2)})=\bE f_1(W_{n,2}^{(2)})$ and $\bE f_0(Z_{n,2}^{(2)})=\bE f_1(Z_{n,2}^{(2)})$ for $n\in\{4,\ldots,10\}$:

\medspace

\resizebox{\linewidth}{!}{	\begin{tabular}{|c||c|c|c|c|c|c|c|c|}
		\hline
		\parbox[0pt][2em][c]{0cm}{}		$n$ & $3$ & $4$ & $5$ & $6$ & $7$ & $8$ & $9$ & $10$ \\
		\hline
		\hline
		\parbox[0pt][2em][c]{0cm}{}		$\bE f_0(W_{n,2}^{(2)})$ & $3$ & $6-\frac{24}{\pi ^2}$ & $10-\frac{60}{\pi ^2}$ & $15+\frac{720}{\pi ^4}-\frac{180}{\pi ^2}$ & $21+\frac{2520}{\pi ^4}-\frac{420}{\pi ^2}$ & $28-\frac{40320}{\pi ^6}+\frac{10080}{\pi ^4}-\frac{840}{\pi ^2}$ & $36-\frac{181440}{\pi ^6}+\frac{30240}{\pi ^4}-\frac{1512}{\pi ^2}$ & $45+\frac{3628800}{\pi ^8}-\frac{907200}{\pi ^6}+\frac{75600}{\pi ^4}-\frac{2520}{\pi ^2}$\\
		\hline
		\hline
		\parbox[0pt][2em][c]{0cm}{}		$\bE f_0(Z_{n,2}^{(2)})$	& $3$ & $3$ & ${24\over 7}$ & ${40\over 11}$ & ${15\over 4}$ & ${42\over 11}$ & $112\over 29$ & $144\over 37$\\
		\hline
\end{tabular}}}

\vspace{2cm}

{\small Values of $\bE f_\ell(W_{n,3}^{(3)})$ and $\bE f_\ell(Z_{n,3}^{(3)})$ for $\ell\in\{0,1,2\}$ and $n\in\{4,\ldots,10\}$:

\medspace

\resizebox{\linewidth}{!}{	\begin{tabular}{|c||c|c|c|c|c|c|c|}
		\hline
		\parbox[0pt][2em][c]{0cm}{}		$n$ & $4$ & $5$ & $6$ & $7$ & $8$ & $9$ & $10$ \\
		\hline
		\hline
		\parbox[0pt][2em][c]{0cm}{}		$\bE f_0(W_{n,3}^{(3)})$	& $4$ & $\frac{20}{3}-\frac{10}{\pi ^2}$ & $10-\frac{30}{\pi ^2}$ & $14+\frac{105}{\pi ^4}-\frac{70}{\pi ^2}$ & $\frac{56}{3}+\frac{420}{\pi ^4}-\frac{140}{\pi ^2}$ & $24-\frac{1890}{\pi ^6}+\frac{1260}{\pi ^4}-\frac{252}{\pi ^2}$ & $30-\frac{9450}{\pi ^6}+\frac{3150}{\pi ^4}-\frac{420}{\pi ^2}$\\
		\hline
		\parbox[0pt][2em][c]{0cm}{}		$\bE f_1(W_{n,3}^{(3)})$ & $6$ & $10-\frac{15}{\pi ^2}$ & $15-\frac{45}{\pi ^2}$ & $21+\frac{315}{2 \pi ^4}-\frac{105}{\pi ^2}$ & $28+\frac{630}{\pi ^4}-\frac{210}{\pi ^2}$ & $36-\frac{2835}{\pi ^6}+\frac{1890}{\pi ^4}-\frac{378}{\pi ^2}$ & $45-\frac{14175}{\pi ^6}+\frac{4725}{\pi ^4}-\frac{630}{\pi ^2}$\\
		\hline
		\parbox[0pt][2em][c]{0cm}{}		$\bE f_2(W_{n,3}^{(3)})$ & $4$ & $\frac{16}{3}-\frac{5}{\pi ^2}$ & $7-\frac{15}{\pi ^2}$ & $9+\frac{105}{2 \pi ^4}-\frac{35}{\pi ^2}$ & $\frac{34}{3}+\frac{210}{\pi ^4}-\frac{70}{\pi ^2}$ & $14-\frac{945}{\pi ^6}+\frac{630}{\pi ^4}-\frac{126}{\pi ^2}$ & $17-\frac{4725}{\pi ^6}+\frac{1575}{\pi ^4}-\frac{210}{\pi ^2}$\\
		\hline
		\hline
		\parbox[0pt][2em][c]{0cm}{}		$\bE f_0(Z_{n,3}^{(3)})$	&  $4$ & ${16\over 3}$ & ${80\over 13}$ & ${20\over 3}$ & $7$ & $224\over 31$ & $96\over 13$ \\
		\hline
		\parbox[0pt][2em][c]{0cm}{}		$\bE f_1(Z_{n,3}^{(3)})$	& $6$ & $8$ & ${120\over 13}$ & $10$ & $21\over 2$ & $336\over 31$ & $114\over 13$\\
		\hline
		\parbox[0pt][2em][c]{0cm}{}		$\bE f_2(Z_{n,3}^{(3)})$	& $4$ & $14\over 3$ & ${66\over 13}$ & $16\over 3$ & $11\over 2$ & $174\over 31$ & $74\over 13$\\
		\hline
\end{tabular}}}

\newpage

\section{Expected spherical Querma\ss integrals}\label{app:Quermass}

{\small Values of $\bE U_1(W_{n,2}^{(2)})$, $\bE U_2(W_{n,2}^{(2)})$ and $\bE U_1(Z_{n,2}^{(2)})$, $\bE U_2(Z_{n,2}^{(2)})$ for $n\in\{3,\ldots,9\}$:}

\medspace

\resizebox{\linewidth}{!}{	\begin{tabular}{|c||c|c|c|c|c|c|c|}
		\hline
		\parbox[0pt][2em][c]{0cm}{}		$n$ & $3$ & $4$ & $5$ & $6$ & $7$ & $8$ & $9$\\
		\hline
		\hline
		\parbox[0pt][2em][c]{0cm}{}		$\bE U_1(W_{n,2}^{(2)})$ & $\frac{3}{4}-\frac{3}{\pi ^2}$	& $1-\frac{6}{\pi ^2}$ & $\frac{5}{4}+\frac{60}{\pi ^4}-\frac{15}{\pi ^2}$ & $\frac{3}{2}+\frac{180}{\pi ^4}-\frac{30}{\pi ^2}$ & $\frac{7}{4}-\frac{2520}{\pi ^6}+\frac{630}{\pi ^4}-\frac{105}{2 \pi ^2}$ & $2-\frac{10080}{\pi ^6}+\frac{1680}{\pi ^4}-\frac{84}{\pi ^2}$ & $\frac{9}{4}+\frac{181440}{\pi ^8}-\frac{45360}{\pi ^6}+\frac{3780}{\pi ^4}-\frac{126}{\pi
			^2}$\\
		\hline
		\parbox[0pt][2em][c]{0cm}{}	$\bE U_2(W_{n,2}^{(2)})$ & $\frac{1}{2}-\frac{3}{\pi ^2}$ & $\frac{1}{2}+\frac{24}{\pi ^4}-\frac{6}{\pi ^2}$ & $\frac{1}{2}+\frac{60}{\pi ^4}-\frac{10}{\pi ^2}$ & $\frac{1}{2}-\frac{720}{\pi ^6}+\frac{180}{\pi ^4}-\frac{15}{\pi ^2}$ & $\frac{1}{2}-\frac{2520}{\pi ^6}+\frac{420}{\pi ^4}-\frac{21}{\pi ^2}$ & $\frac{1}{2}+\frac{40320}{\pi ^8}-\frac{10080}{\pi ^6}+\frac{840}{\pi ^4}-\frac{28}{\pi
			^2}$ & $\frac{1}{2}+\frac{181440}{\pi ^8}-\frac{30240}{\pi ^6}+\frac{1512}{\pi ^4}-\frac{36}{\pi
			^2}$\\
		\hline
		\hline
		\parbox[0pt][2em][c]{0cm}{}		$\bE U_1(Z_{n,2}^{(2)})$	& $3\over 8$ & $2\over 7$ & $5\over 22$ & $3\over 16$ & $7\over 44$ & $4\over 29$ & $9\over 74$\\
		\hline
		\parbox[0pt][2em][c]{0cm}{}	$\bE U_2(Z_{n,2}^{(2)})$ & $1\over 8$ & $1\over 14$ & $1\over 22$ & $1\over 32$ & $1\over 44$ & $1\over 58$ & $1\over 74$\\
		\hline
\end{tabular}}

\vspace{1cm}

{\small Values of $\bE U_1(W_{n,3}^{(3)})$, $\bE U_2(W_{n,3}^{(3)})$, $\bE U_3(W_{n,3}^{(3)})$ and $\bE U_1(Z_{n,3}^{(3)})$, $\bE U_2(Z_{n,3}^{(3)})$, $\bE U_3(Z_{n,3}^{(3)})$ for $n\in\{4,\ldots,9\}$:

\medspace

\resizebox{\linewidth}{!}{	\begin{tabular}{|c||c|c|c|c|c|c|}
		\hline
		\parbox[0pt][2em][c]{0cm}{}		$n$ & $4$ & $5$ & $6$ & $7$ & $8$ & $9$\\
		\hline
		\hline
		\parbox[0pt][2em][c]{0cm}{}		$\bE U_1(W_{n,3}^{(3)})$ & $\frac{8}{15}-\frac{1}{2 \pi ^2}$ & $\frac{7}{12}-\frac{5}{4 \pi ^2}$ & $\frac{9}{14}+\frac{15}{4 \pi ^4}-\frac{5}{2 \pi ^2}$ & $\frac{17}{24}+\frac{105}{8 \pi ^4}-\frac{35}{8 \pi ^2}$ & $\frac{7}{9}-\frac{105}{2 \pi ^6}+\frac{35}{\pi ^4}-\frac{7}{\pi ^2}$ & $\frac{17}{20}-\frac{945}{4 \pi ^6}+\frac{315}{4 \pi ^4}-\frac{21}{2 \pi ^2}$\\
		\hline
		\parbox[0pt][2em][c]{0cm}{}	$\bE U_2(W_{n,3}^{(3)})$ & $\frac{1}{2}-\frac{3}{2 \pi ^2}$ & $\frac{1}{2}+\frac{15}{4 \pi ^4}-\frac{5}{2 \pi ^2}$ & $\frac{1}{2}+\frac{45}{4 \pi ^4}-\frac{15}{4 \pi ^2}$ & $\frac{1}{2}-\frac{315}{8 \pi ^6}+\frac{105}{4 \pi ^4}-\frac{21}{4 \pi ^2}$ & $\frac{1}{2}-\frac{315}{2 \pi ^6}+\frac{105}{2 \pi ^4}-\frac{7}{\pi ^2}$ & $\frac{1}{2}+\frac{2835}{4 \pi ^8}-\frac{945}{2 \pi ^6}+\frac{189}{2 \pi ^4}-\frac{9}{\pi
			^2}$\\
		\hline
		\parbox[0pt][2em][c]{0cm}{}	$\bE U_3(W_{n,3}^{(3)})$ & $\frac{1}{5}+\frac{3}{2 \pi ^4}-\frac{1}{\pi ^2}$ & $\frac{1}{6}+\frac{15}{4 \pi ^4}-\frac{5}{4 \pi ^2}$ & $\frac{1}{7}-\frac{45}{4 \pi ^6}+\frac{15}{2 \pi ^4}-\frac{3}{2 \pi ^2}$ & $\frac{1}{8}-\frac{315}{8 \pi ^6}+\frac{105}{8 \pi ^4}-\frac{7}{4 \pi ^2}$ & $\frac{1}{9}+\frac{315}{2 \pi ^8}-\frac{105}{\pi ^6}+\frac{21}{\pi ^4}-\frac{2}{\pi ^2}$ & $\frac{1}{10}+\frac{2835}{4 \pi ^8}-\frac{945}{4 \pi ^6}+\frac{63}{2 \pi ^4}-\frac{9}{4
			\pi ^2}$\\
		\hline
		\hline
		\parbox[0pt][2em][c]{0cm}{}		$\bE U_1(Z_{n,3}^{(3)})$	&  $7\over 16$ & $11\over 30$ & $4\over 13$ & $11\over 42$ & $29\over 128$ & $37\over 186$\\
		\hline
		\parbox[0pt][2em][c]{0cm}{}	$\bE U_2(Z_{n,3}^{(3)})$ & $1\over 4$ & $1\over 6$ & $3\over 26$ & $1\over 12$ & $1\over 16$ & $3\over 62$\\
		\hline
		\parbox[0pt][2em][c]{0cm}{}	$\bE U_3(Z_{n,3}^{(3)})$ & $1\over 16$ & $1\over 30$ & $1\over 52$ & $1\over 84$ & $1\over 128$ & $1\over 186$\\
		\hline
\end{tabular}}}

\newpage

\section{Expected spherical intrinsic volumes}\label{app:IntVol}

{\footnotesize Values of $\bE v_0(W_{n,2}^{(2)})$, $\bE v_1(W_{n,2}^{(2)})$, $\bE v_2(W_{n,2}^{(2)})$ and $\bE v_0(Z_{n,2}^{(2)})$, $\bE v_1(Z_{n,2}^{(2)})$ and $\bE v_2(Z_{n,2}^{(2)})$ for $n\in\{3,\ldots,9\}$:

\medspace

\resizebox{\linewidth}{!}{\tiny	\begin{tabular}{|c||c|c|c|c|c|c|c|}
		\hline
		\parbox[0pt][2em][c]{0cm}{}		$n$ & $3$ & $4$ & $5$ & $6$ & $7$ & $8$ & $9$\\
		\hline
		\hline
		\parbox[0pt][2em][c]{0cm}{}		$\bE v_0(W_{n,2}^{(2)})$ & $\frac{3}{\pi ^2}$	& $\frac{6}{\pi ^2}-\frac{24}{\pi ^4}$ & $\frac{10}{\pi ^2}-\frac{60}{\pi ^4}$ & $\frac{720}{\pi ^6}-\frac{180}{\pi ^4}+\frac{15}{\pi ^2}$ & $\frac{2520}{\pi ^6}-\frac{420}{\pi ^4}+\frac{21}{\pi ^2}$ & $-\frac{40320}{\pi ^8}+\frac{10080}{\pi ^6}-\frac{840}{\pi ^4}+\frac{28}{\pi ^2}$ & $-\frac{181440}{\pi ^8}+\frac{30240}{\pi ^6}-\frac{1512}{\pi ^4}+\frac{36}{\pi ^2}$\\
		\hline
		\parbox[0pt][2em][c]{0cm}{}		$\bE v_1(W_{n,2}^{(2)})$ & $\frac{3}{2 \pi }$	& $\frac{3}{\pi }-\frac{12}{\pi ^3}$ & $\frac{5}{\pi }-\frac{30}{\pi ^3}$ & $\frac{360}{\pi ^5}-\frac{90}{\pi ^3}+\frac{15}{2 \pi }$ & $\frac{1260}{\pi ^5}-\frac{210}{\pi ^3}+\frac{21}{2 \pi }$ & $-\frac{20160}{\pi ^7}+\frac{5040}{\pi ^5}-\frac{420}{\pi ^3}+\frac{14}{\pi }$ & $-\frac{90720}{\pi ^7}+\frac{15120}{\pi ^5}-\frac{756}{\pi ^3}+\frac{18}{\pi }$\\
		\hline
		\parbox[0pt][2em][c]{0cm}{}		$\bE v_2(W_{n,2}^{(2)})$ & $\frac{3}{\pi ^2}$	& $\frac{6}{\pi ^2}-\frac{24}{\pi ^4}$ & $\frac{10}{\pi ^2}-\frac{60}{\pi ^4}$ & $\frac{720}{\pi ^6}-\frac{180}{\pi ^4}+\frac{15}{\pi ^2}$ & $\frac{2520}{\pi ^6}-\frac{420}{\pi ^4}+\frac{21}{\pi ^2}$ & $-\frac{40320}{\pi ^8}+\frac{10080}{\pi ^6}-\frac{840}{\pi ^4}+\frac{28}{\pi ^2}$ & $-\frac{181440}{\pi ^8}+\frac{30240}{\pi ^6}-\frac{1512}{\pi ^4}+\frac{36}{\pi ^2}$\\
		\hline
		\hline
		\parbox[0pt][2em][c]{0cm}{}		$\bE v_0(Z_{n,2}^{(2)})$	& $3\over 8$ & $3\over 7$ & $5\over 11$ & $15\over 32$ & $21\over 44$ & $14\over 29$ & $18\over 37$\\
		\hline
		\parbox[0pt][2em][c]{0cm}{}		$\bE v_1(Z_{n,2}^{(2)})$	& $3\over 8$ & $2\over 7$ & $5\over 22$ & $3\over 16$ & $7\over 44$ & $4\over 29$ & $9\over 74$\\
		\hline
		\parbox[0pt][2em][c]{0cm}{}		$\bE v_2(Z_{n,2}^{(2)})$	& $1\over 8$ & $1\over 14$ & $1\over 22$ & $1\over 32$ & $1\over 44$ & $1\over 58$ & $1\over 74$\\
		\hline
\end{tabular}}}

\bigskip


{\footnotesize Values of $\bE v_0(W_{n,3}^{(3)})$, $\bE v_1(W_{n,3}^{(3)})$, $\bE v_2(W_{n,3}^{(3)})$, $\bE v_3(W_{n,3}^{(3)})$ and $\bE v_0(Z_{n,3}^{(3)})$, $\bE v_1(Z_{n,3}^{(3)})$, $\bE v_2(Z_{n,3}^{(3)})$, $\bE v_3(Z_{n,3}^{(3)})$ for $n\in\{4,\ldots,9\}$:
	
	\medspace
	
	\resizebox{\linewidth}{!}{\tiny	\begin{tabular}{|c||c|c|c|c|c|c|}
			\hline
			\parbox[0pt][2em][c]{0cm}{}		$n$ & $4$ & $5$ & $6$ & $7$ & $8$ & $9$\\
			\hline
			\hline
			\parbox[0pt][2em][c]{0cm}{}		$\bE v_0(W_{n,3}^{(3)})$ & $\frac{3}{2 \pi ^2}$ & $\frac{5}{2 \pi ^2}-\frac{15}{4 \pi ^4}$ & $\frac{15}{4 \pi ^2}-\frac{45}{4 \pi ^4}$ & $\frac{315}{8 \pi ^6}-\frac{105}{4 \pi ^4}+\frac{21}{4 \pi ^2}$ & $\frac{315}{2 \pi ^6}-\frac{105}{2 \pi ^4}+\frac{7}{\pi ^2}$ & $-\frac{2835}{4 \pi ^8}+\frac{945}{2 \pi ^6}-\frac{189}{2 \pi ^4}+\frac{9}{\pi ^2}$\\
			\hline
			\hline
			\parbox[0pt][2em][c]{0cm}{}		$\bE v_1(W_{n,3}^{(3)})$ & $\frac{3}{\pi ^3}+\frac{1}{\pi }$ & $-\frac{15}{2 \pi ^5}+\frac{5}{2 \pi ^3}+\frac{5}{3 \pi }$ & $\frac{5}{2 \pi }-\frac{45}{2 \pi ^5}$ & $\frac{315}{4 \pi ^7}-\frac{105}{4 \pi ^5}-\frac{7}{\pi ^3}+\frac{7}{2 \pi }$ & $\frac{315}{\pi ^7}-\frac{21}{\pi ^3}+\frac{14}{3 \pi }$ & $-\frac{2835}{2 \pi ^9}+\frac{945}{2 \pi ^7}+\frac{126}{\pi ^5}-\frac{45}{\pi ^3}+\frac{6}{\pi }$\\
			\hline
			\parbox[0pt][2em][c]{0cm}{}		$\bE v_2(W_{n,3}^{(3)})$ & $\frac{6}{\pi ^2}$ & $\frac{10}{\pi ^2}-\frac{15}{\pi ^4}$ & $\frac{15}{\pi ^2}-\frac{45}{\pi ^4}$ & $\frac{315}{2 \pi ^6}-\frac{105}{\pi ^4}+\frac{21}{\pi ^2}$ & $\frac{630}{\pi ^6}-\frac{210}{\pi ^4}+\frac{28}{\pi ^2}$ & $-\frac{2835}{\pi ^8}+\frac{1890}{\pi ^6}-\frac{378}{\pi ^4}+\frac{36}{\pi ^2}$\\
			\hline
			\parbox[0pt][2em][c]{0cm}{}		$\bE v_3(W_{n,3}^{(3)})$ & $\frac{12}{\pi ^3}$ & $\frac{20}{\pi ^3}-\frac{30}{\pi ^5}$ & $\frac{30}{\pi ^3}-\frac{90}{\pi ^5}$ & $\frac{315}{\pi ^7}-\frac{210}{\pi ^5}+\frac{42}{\pi ^3}$ & $\frac{1260}{\pi ^7}-\frac{420}{\pi ^5}+\frac{56}{\pi ^3}$ & $-\frac{5670}{\pi ^9}+\frac{3780}{\pi ^7}-\frac{756}{\pi ^5}+\frac{72}{\pi ^3}$\\
			\hline
			\hline
			\parbox[0pt][2em][c]{0cm}{}		$\bE v_0(Z_{n,3}^{(3)})$ &   $1\over 4$ & $1\over 3$ & $5\over 13$ & $5\over 12$ & $7\over 16$ & $14\over 31$\\
			\hline
			\parbox[0pt][2em][c]{0cm}{}		$\bE v_1(Z_{n,3}^{(3)})$ & $3\over 8$ & $1\over 3$ & $15\over 52$ & $1\over 4$ & $7\over 32$ & $6\over 31$\\
			\hline
			\parbox[0pt][2em][c]{0cm}{}		$\bE v_2(Z_{n,3}^{(3)})$ & $1\over 4$ & $1\over 6$ & $3\over 26$ & $1\over 12$ & $1\over 16$ & $3\over 62$\\
			\hline
			\parbox[0pt][2em][c]{0cm}{}		$\bE v_3(Z_{n,3}^{(3)})$ & $1\over 16$ & $1\over 30$ & $1\over 52$ & $1\over 84$ & $1\over 128$ & $1\over 186$\\
			\hline
\end{tabular}}}

\newpage

\section{Statistical Dimensions}

{\small Values of $\bE\delta(\breve{W}_{n,2}^{(2)})$ and $\bE\delta(\breve{Z}_{n,2}^{(2)})$ for $n\in\{4,\ldots,10\}$:}

\medspace

\resizebox{\linewidth}{!}{
	\begin{tabular}{|c||c|c|c|c|c|c|c|c|}
		\hline
		\parbox[0pt][2em][c]{0cm}{}	$n$ & $3$ & $4$ & $5$ & $6$ & $7$ & $8$ & $9$ & $10$\\
		\hline
		\hline
		\parbox[0pt][2em][c]{0cm}{} $\bE\delta(\breve{W}_{n,2}^{(2)})$ & $3-\frac{12}{\pi ^2}$ & $\frac{7}{2}+\frac{48}{\pi ^4}-\frac{24}{\pi ^2}$ & $4+\frac{240}{\pi ^4}-\frac{50}{\pi ^2}$ & $\frac{9}{2}-\frac{1440}{\pi ^6}+\frac{720}{\pi ^4}-\frac{90}{\pi ^2}$ & $5-\frac{10080}{\pi ^6}+\frac{2100}{\pi ^4}-\frac{147}{\pi ^2}$ & $\frac{11}{2}+\frac{80640}{\pi ^8}-\frac{40320}{\pi ^6}+\frac{5040}{\pi ^4}-\frac{224}{\pi ^2}$ & $6+\frac{725760}{\pi ^8}-\frac{151200}{\pi ^6}+\frac{10584}{\pi ^4}-\frac{324}{\pi ^2}$ & $\frac{13}{2}-\frac{7257600}{\pi ^{10}}+\frac{3628800}{\pi ^8}-\frac{453600}{\pi
			^6}+\frac{20160}{\pi ^4}-\frac{450}{\pi ^2}$\\
		\hline
		\parbox[0pt][2em][c]{0cm}{}	$\bE\delta(\breve{Z}_{n,2}^{(2)})$ & $\frac{3}{2}$ & $\frac{17}{14}$ & $\frac{23}{22}$ & $\frac{15}{16}$ & $\frac{19}{22}$ & $\frac{47}{58}$ & $\frac{57}{74}$ & $\frac{17}{23}$\\
		\hline
	\end{tabular}
}

\bigskip

{\small Values of $\bE\delta(\breve{W}_{n,3}^{(3)})$ and $\bE\delta(\breve{Z}_{n,3}^{(3)})$ for $n\in\{4,\ldots,10\}$:}

\medspace

\resizebox{\linewidth}{!}{
	\begin{tabular}{|c||c|c|c|c|c|c|c|}
		\hline
		\parbox[0pt][2em][c]{0cm}{} $n$	& $4$ & $5$ & $6$ & $7$ & $8$ & $9$ & $10$ \\
		\hline
		\hline
		\parbox[0pt][2em][c]{0cm}{} $\bE\delta(\breve{W}_{n,3}^{(3)})$ & $\frac{89}{30}+\frac{3}{\pi ^4}-\frac{6}{\pi ^2}$ & $3+\frac{15}{\pi ^4}-\frac{10}{\pi ^2}$ & $\frac{43}{14}-\frac{45}{2 \pi ^6}+\frac{45}{\pi ^4}-\frac{31}{2 \pi ^2}$ & $\frac{19}{6}-\frac{315}{2 \pi ^6}+\frac{105}{\pi ^4}-\frac{91}{4 \pi ^2}$ & $\frac{59}{18}+\frac{315}{\pi ^8}-\frac{630}{\pi ^6}+\frac{217}{\pi ^4}-\frac{32}{\pi ^2}$ & $\frac{17}{5}+\frac{2835}{\pi ^8}-\frac{1890}{\pi ^6}+\frac{819}{2 \pi ^4}-\frac{87}{2 \pi ^2}$ & $\frac{233}{66}-\frac{14175}{2 \pi ^{10}}+\frac{14175}{\pi ^8}-\frac{9765}{2 \pi ^6}+\frac{720}{\pi
			^4}-\frac{115}{2 \pi ^2}$\\
		\hline
		\parbox[0pt][2em][c]{0cm}{}	$\bE\delta(\breve{Z}_{n,3}^{(3)})$ & $2$ & $\frac{49}{30}$ & $\frac{18}{13}$ & $\frac{17}{14}$ & $\frac{35}{32}$ & $\frac{187}{186}$ & $\frac{61}{65}$\\
		\hline
	\end{tabular}
}

\section{Intersection probabilities}\label{app:IntesectionProb}

{\small Values of $\bP(W_{n,2}\cap V_{m,2}\neq\varnothing)$ for $n,m\in\{3,\ldots,8\}$:}

\medspace

\resizebox{\linewidth}{!}{
\begin{tabular}{|c||c|c|c|c|c|c|}
\hline
 \parbox[0pt][2em][c]{0cm}{} $n$	& $3$ & $4$ & $5$ & $6$ & $7$ & $8$ \\
 \hline
 \hline
 \parbox[0pt][2em][c]{0cm}{}	$m=3$ & $\frac{13}{8}-\frac{9}{\pi ^2}$ & $2+\frac{144}{\pi ^6}-\frac{15}{\pi ^2}$ & $\frac{19}{8}+\frac{120}{\pi ^4}-\frac{30}{\pi ^2}$ & $\frac{11}{4}-\frac{4320}{\pi ^8}+\frac{360}{\pi ^4}-\frac{54}{\pi ^2}$ & $\frac{25}{8}-\frac{5040}{\pi ^6}+\frac{1134}{\pi ^4}-\frac{357}{4 \pi ^2}$ & $\frac{7}{2}+\frac{241920}{\pi ^{10}}-\frac{20160}{\pi ^6}+\frac{2856}{\pi
 	^4}-\frac{138}{\pi ^2}$\\
 \hline
 \parbox[0pt][2em][c]{0cm}{} $m=4$ & $2+\frac{144}{\pi ^6}-\frac{15}{\pi ^2}$ & $\frac{5}{2}-\frac{1152}{\pi ^8}+\frac{576}{\pi ^6}-\frac{24}{\pi ^2}$ & $3-\frac{2880}{\pi ^8}+\frac{480}{\pi ^6}+\frac{180}{\pi ^4}-\frac{45}{\pi ^2}$ & $\frac{7}{2}+\frac{34560}{\pi ^{10}}-\frac{17280}{\pi ^8}+\frac{720}{\pi
 	^6}+\frac{540}{\pi ^4}-\frac{78}{\pi ^2}$ & $4+\frac{120960}{\pi ^{10}}-\frac{20160}{\pi ^8}-\frac{6552}{\pi ^6}+\frac{1638}{\pi
 	^4}-\frac{126}{\pi ^2}$ & $\frac{9}{2}-\frac{1935360}{\pi ^{12}}+\frac{967680}{\pi ^{10}}-\frac{40320}{\pi
 	^8}-\frac{28896}{\pi ^6}+\frac{4032}{\pi ^4}-\frac{192}{\pi ^2}$\\
 \hline
  \parbox[0pt][2em][c]{0cm}{} $m=5$ & $\frac{19}{8}+\frac{120}{\pi ^4}-\frac{30}{\pi ^2}$ & $3-\frac{2880}{\pi ^8}+\frac{480}{\pi ^6}+\frac{180}{\pi ^4}-\frac{45}{\pi ^2}$ & $\frac{29}{8}-\frac{1200}{\pi ^6}+\frac{550}{\pi ^4}-\frac{75}{\pi ^2}$ & $\frac{17}{4}+\frac{86400}{\pi ^{10}}-\frac{14400}{\pi ^8}-\frac{3600}{\pi
  	^6}+\frac{1230}{\pi ^4}-\frac{120}{\pi ^2}$ & $\frac{39}{8}+\frac{50400}{\pi ^8}-\frac{20580}{\pi ^6}+\frac{2940}{\pi ^4}-\frac{735}{4
  	\pi ^2}$ & $\frac{11}{2}-\frac{4838400}{\pi ^{12}}+\frac{806400}{\pi ^{10}}+\frac{201600}{\pi
  	^8}-\frac{65520}{\pi ^6}+\frac{6400}{\pi ^4}-\frac{270}{\pi ^2}$\\
  \hline
  \parbox[0pt][2em][c]{0cm}{} $m=6$ & $\frac{11}{4}-\frac{4320}{\pi ^8}+\frac{360}{\pi ^4}-\frac{54}{\pi ^2}$ & $\frac{7}{2}+\frac{34560}{\pi ^{10}}-\frac{17280}{\pi ^8}+\frac{720}{\pi
  	^6}+\frac{540}{\pi ^4}-\frac{78}{\pi ^2}$ & $\frac{17}{4}+\frac{86400}{\pi ^{10}}-\frac{14400}{\pi ^8}-\frac{3600}{\pi
  	^6}+\frac{1230}{\pi ^4}-\frac{120}{\pi ^2}$ & $5-\frac{1036800}{\pi ^{12}}+\frac{518400}{\pi ^{10}}-\frac{43200}{\pi
  	^8}-\frac{10800}{\pi ^6}+\frac{2430}{\pi ^4}-\frac{180}{\pi ^2}$ & $\frac{23}{4}-\frac{3628800}{\pi ^{12}}+\frac{604800}{\pi ^{10}}+\frac{120960}{\pi
  	^8}-\frac{44100}{\pi ^6}+\frac{5040}{\pi ^4}-\frac{525}{2 \pi ^2}$ & $\frac{13}{2}+\frac{58060800}{\pi ^{14}}-\frac{29030400}{\pi ^{12}}+\frac{2419200}{\pi
  	^{10}}+\frac{564480}{\pi ^8}-\frac{126000}{\pi ^6}+\frac{9960}{\pi ^4}-\frac{372}{\pi
  	^2}$\\
   \hline
  \parbox[0pt][2em][c]{0cm}{} $m=7$ & $\frac{25}{8}-\frac{5040}{\pi ^6}+\frac{1134}{\pi ^4}-\frac{357}{4 \pi ^2}$ & $4+\frac{120960}{\pi ^{10}}-\frac{20160}{\pi ^8}-\frac{6552}{\pi ^6}+\frac{1638}{\pi
  	^4}-\frac{126}{\pi ^2}$ & $\frac{39}{8}+\frac{50400}{\pi ^8}-\frac{20580}{\pi ^6}+\frac{2940}{\pi ^4}-\frac{735}{4
  	\pi ^2}$ & $\frac{23}{4}-\frac{3628800}{\pi ^{12}}+\frac{604800}{\pi ^{10}}+\frac{120960}{\pi
  	^8}-\frac{44100}{\pi ^6}+\frac{5040}{\pi ^4}-\frac{525}{2 \pi ^2}$ & $\frac{53}{8}-\frac{2116800}{\pi ^{10}}+\frac{758520}{\pi ^8}-\frac{114660}{\pi
  	^6}+\frac{18081}{2 \pi ^4}-\frac{735}{2 \pi ^2}$ & $\frac{15}{2}+\frac{203212800}{\pi ^{14}}-\frac{33868800}{\pi ^{12}}-\frac{6773760}{\pi
  	^{10}}+\frac{2328480}{\pi ^8}-\frac{268800}{\pi ^6}+\frac{16044}{\pi
  	^4}-\frac{504}{\pi ^2}$\\
   \hline
  \parbox[0pt][2em][c]{0cm}{} $m=8$ & $\frac{7}{2}+\frac{241920}{\pi ^{10}}-\frac{20160}{\pi ^6}+\frac{2856}{\pi
  	^4}-\frac{138}{\pi ^2}$ & $\frac{9}{2}-\frac{1935360}{\pi ^{12}}+\frac{967680}{\pi ^{10}}-\frac{40320}{\pi
  	^8}-\frac{28896}{\pi ^6}+\frac{4032}{\pi ^4}-\frac{192}{\pi ^2}$ & $\frac{11}{2}-\frac{4838400}{\pi ^{12}}+\frac{806400}{\pi ^{10}}+\frac{201600}{\pi
  	^8}-\frac{65520}{\pi ^6}+\frac{6400}{\pi ^4}-\frac{270}{\pi ^2}$ & $\frac{13}{2}+\frac{58060800}{\pi ^{14}}-\frac{29030400}{\pi ^{12}}+\frac{2419200}{\pi
  	^{10}}+\frac{564480}{\pi ^8}-\frac{126000}{\pi ^6}+\frac{9960}{\pi ^4}-\frac{372}{\pi
  	^2}$ & $\frac{15}{2}+\frac{203212800}{\pi ^{14}}-\frac{33868800}{\pi ^{12}}-\frac{6773760}{\pi
  	^{10}}+\frac{2328480}{\pi ^8}-\frac{268800}{\pi ^6}+\frac{16044}{\pi
  	^4}-\frac{504}{\pi ^2}$ & $\frac{17}{2}-\frac{3251404800}{\pi ^{16}}+\frac{1625702400}{\pi
  	^{14}}-\frac{135475200}{\pi ^{12}}-\frac{29352960}{\pi ^{10}}+\frac{6491520}{\pi
  	^8}-\frac{551040}{\pi ^6}+\frac{25984}{\pi ^4}-\frac{672}{\pi ^2}$\\
  \hline
\end{tabular}
}

\bigskip

{\small Values of $\bP(W_{n,3}\cap V_{m,3}\neq\varnothing)$ for $n,m\in\{4,\ldots,8\}$:}

\medspace

\resizebox{\linewidth}{!}{
	\begin{tabular}{|c||c|c|c|c|c|}
		\hline
		\parbox[0pt][2em][c]{0cm}{}	$n$ & $4$ & $5$ & $6$ & $7$ & $8$\\
		\hline
		\hline
		\parbox[0pt][2em][c]{0cm}{}	$m=4$ & $\frac{16}{15}+\frac{9}{\pi ^6}-\frac{3}{\pi ^4}-\frac{3}{\pi ^2}$ & $\frac{67}{60}-\frac{45}{4 \pi ^8}+\frac{45}{2 \pi ^6}-\frac{14}{3 \pi ^2}$ & $\frac{247}{210}-\frac{135}{2 \pi ^8}+\frac{135}{4 \pi ^6}+\frac{21}{2 \pi
			^4}-\frac{7}{\pi ^2}$ & $\frac{149}{120}+\frac{945}{8 \pi ^{10}}-\frac{945}{4 \pi ^8}+\frac{63}{4 \pi
			^6}+\frac{133}{4 \pi ^4}-\frac{81}{8 \pi ^2}$ & $\frac{59}{45}+\frac{945}{\pi ^{10}}-\frac{945}{2 \pi ^8}-\frac{126}{\pi ^6}+\frac{78}{\pi
			^4}-\frac{85}{6 \pi ^2}$ \\
		\hline
		\parbox[0pt][2em][c]{0cm}{}	$m=5$ & $\frac{67}{60}-\frac{45}{4 \pi ^8}+\frac{45}{2 \pi ^6}-\frac{14}{3 \pi ^2}$ & $\frac{7}{6}-\frac{225}{4 \pi ^8}+\frac{75}{2 \pi ^6}+\frac{25}{4 \pi ^4}-\frac{20}{3 \pi
			^2}$ & $\frac{103}{84}+\frac{675}{8 \pi ^{10}}-\frac{675}{4 \pi ^8}+\frac{315}{8 \pi
			^6}+\frac{175}{8 \pi ^4}-\frac{75}{8 \pi ^2}$ & $\frac{31}{24}+\frac{4725}{8 \pi ^{10}}-\frac{1575}{4 \pi ^8}-\frac{105}{8 \pi
			^6}+\frac{105}{2 \pi ^4}-\frac{155}{12 \pi ^2}$ & $\frac{49}{36}-\frac{4725}{4 \pi ^{12}}+\frac{4725}{2 \pi ^{10}}-\frac{2205}{4 \pi
			^8}-\frac{955}{4 \pi ^6}+\frac{435}{4 \pi ^4}-\frac{209}{12 \pi ^2}$\\
		\hline
		\parbox[0pt][2em][c]{0cm}{}	$m=6$ & $\frac{247}{210}-\frac{135}{2 \pi ^8}+\frac{135}{4 \pi ^6}+\frac{21}{2 \pi
			^4}-\frac{7}{\pi ^2}$ & $\frac{103}{84}+\frac{675}{8 \pi ^{10}}-\frac{675}{4 \pi ^8}+\frac{315}{8 \pi
			^6}+\frac{175}{8 \pi ^4}-\frac{75}{8 \pi ^2}$ & $\frac{9}{7}+\frac{2025}{4 \pi ^{10}}-\frac{675}{2 \pi ^8}+\frac{45}{4 \pi
			^6}+\frac{45}{\pi ^4}-\frac{25}{2 \pi ^2}$ & $\frac{227}{168}-\frac{14175}{16 \pi ^{12}}+\frac{14175}{8 \pi ^{10}}-\frac{8505}{16 \pi
			^8}-\frac{1785}{16 \pi ^6}+\frac{1383}{16 \pi ^4}-\frac{33}{2 \pi ^2}$ & $\frac{179}{126}-\frac{14175}{2 \pi ^{12}}+\frac{4725}{\pi ^{10}}-\frac{315}{\pi
			^8}-\frac{480}{\pi ^6}+\frac{631}{4 \pi ^4}-\frac{43}{2 \pi ^2}$\\
		\hline
		\parbox[0pt][2em][c]{0cm}{}	$m=7$ & $\frac{149}{120}+\frac{945}{8 \pi ^{10}}-\frac{945}{4 \pi ^8}+\frac{63}{4 \pi
			^6}+\frac{133}{4 \pi ^4}-\frac{81}{8 \pi ^2}$ & $\frac{31}{24}+\frac{4725}{8 \pi ^{10}}-\frac{1575}{4 \pi ^8}-\frac{105}{8 \pi
			^6}+\frac{105}{2 \pi ^4}-\frac{155}{12 \pi ^2}$ & $\frac{227}{168}-\frac{14175}{16 \pi ^{12}}+\frac{14175}{8 \pi ^{10}}-\frac{8505}{16 \pi
			^8}-\frac{1785}{16 \pi ^6}+\frac{1383}{16 \pi ^4}-\frac{33}{2 \pi ^2}$ & $\frac{17}{12}-\frac{99225}{16 \pi ^{12}}+\frac{33075}{8 \pi ^{10}}-\frac{6615}{16 \pi
			^8}-\frac{735}{2 \pi ^6}+\frac{1141}{8 \pi ^4}-\frac{21}{\pi ^2}$ & $\frac{107}{72}+\frac{99225}{8 \pi ^{14}}-\frac{99225}{4 \pi ^{12}}+\frac{59535}{8 \pi
			^{10}}+\frac{6825}{8 \pi ^8}-\frac{7791}{8 \pi ^6}+\frac{1869}{8 \pi ^4}-\frac{637}{24
			\pi ^2}$\\
			\hline
		\parbox[0pt][2em][c]{0cm}{}	$m=8$ & $\frac{59}{45}+\frac{945}{\pi ^{10}}-\frac{945}{2 \pi ^8}-\frac{126}{\pi ^6}+\frac{78}{\pi
			^4}-\frac{85}{6 \pi ^2}$ & $\frac{49}{36}-\frac{4725}{4 \pi ^{12}}+\frac{4725}{2 \pi ^{10}}-\frac{2205}{4 \pi
			^8}-\frac{955}{4 \pi ^6}+\frac{435}{4 \pi ^4}-\frac{209}{12 \pi ^2}$ & $\frac{179}{126}-\frac{14175}{2 \pi ^{12}}+\frac{4725}{\pi ^{10}}-\frac{315}{\pi
			^8}-\frac{480}{\pi ^6}+\frac{631}{4 \pi ^4}-\frac{43}{2 \pi ^2}$ & $\frac{107}{72}+\frac{99225}{8 \pi ^{14}}-\frac{99225}{4 \pi ^{12}}+\frac{59535}{8 \pi
			^{10}}+\frac{6825}{8 \pi ^8}-\frac{7791}{8 \pi ^6}+\frac{1869}{8 \pi ^4}-\frac{637}{24
			\pi ^2}$ & $\frac{14}{9}+\frac{99225}{\pi ^{14}}-\frac{66150}{\pi ^{12}}+\frac{6615}{\pi
			^{10}}+\frac{4620}{\pi ^8}-\frac{1967}{\pi ^6}+\frac{350}{\pi ^4}-\frac{98}{3 \pi ^2}$\\
		  \hline
	\end{tabular}
}

\end{landscape}


\begin{thebibliography}{30}\small


\bibitem{AmelunxenBuergisser}
Amelunxen, D. and B\"urgisser, P: Intrinsic volumes of symmetric cones and applications in convex programming. Math. Programm., Ser. A \textbf{149}, 105--130 (2015).

\bibitem{AmelunxenLotzMcCoyTropp}
Amelunxen, D., Lotz, M., McCoy, M.B. and Tropp, J.A.: Living on the edge: phase transitions in convex programs with random data. Inf. Inference \textbf{3}, 224--298 (2014).

\bibitem{ArbeiterZaehle}
Arbeiter, E. and Z\"ahle, M.: Geometric measures for random mosaics in spherical spaces. Stochastic Stochastics Rep. \textbf{46}, 63--77 (1994).

\bibitem{ArratiaGoldsteinKochman}
Arratia, R., Goldstein, L. and Kochman, F.: Size bias for one and all. Probab. Surveys \textbf{16}, 1--61 (2019).

\bibitem{BaccelliOReilly}
Baccelli, F. and O'Reilly, E.: The stochastic geometry of unconstrained one-bit data compression. Electron. J. Probab. \textbf{24}, article 138 (2019).


\bibitem{BaranyHugReitznerSchneider}
B\'ar\'any, I., Hug, D., Reitzner, M. and Schneider, R.: Random points in halfspheres. Random Structures Algorithms \textbf{50}, 3--22 (2017).




\bibitem{BilykLacey}
Bilyk, D. and Lacey, M.T.: Random tessellations, restricted isometric embeddings, and one bit sensing. arXiv: 1512.06697.


\bibitem{CoverEfron}
Cover, T.M. and Efron, B.: Geometrical probability and random points on a hypersphere. Ann. Math. Stat. \textbf{38}, 213--220 (1967).


\bibitem{DeussHoerrmannThaele}
Deu\ss, C.: H\"orrmann, J. and Th\"ale, C.: A random cell splitting scheme on the sphere. Stochastic Processes Appl. \textbf{127}, 1554--1564 (2017).


\bibitem{GodlandKabluchko}
Godland, T. and Kabluchko, Z.: Conical tessellations associated with Weyl chambers. arXiv: 2004.10466.


\bibitem{HeroldHugThaele}
Herold, F., Hug, D. and Th\"ale, C.: Does a central limit theorem hold for the $k$-skeleton of Poisson hyperplanes in hyperbolic space? arXiv: 1911.02120.

\bibitem{HoldenPeresZhai}
Holden, N., Peres, Y. and Zhai, A.: Gravitational allocation on the sphere. Proc. Natl. Acad. Sci. USA \textbf{115}, 9666--9671 (2018).

\bibitem{HoermannHugReitznerThaele}
H\"orrmann, J., Hug, D., Reitzner, M. and Th\"ale, C.: Poisson polyhedra in high dimensions. Adv. Math. \textbf{281}, 1--39 (2015).


\bibitem{HugReichenbacher}
Hug, D. and Reichenbacher, A.: Geometric inequalities, stability results and Kendall's problem in spherical space. arXiv 1709.06522.

\bibitem{HugSchneiderFacesDirection}
Hug, D. and Schneider, R.: Faces with given directions in anisotropic Poisson hyperplane mosaics. Adv. in Appl. Probab. \textbf{43}, 308--321 (2011).

\bibitem{HugSchneiderConicalTessellations}
Hug, D. and Schneider, R.: Random conical tessellations. Discrete Comput. Geom. \textbf{56}, 395--426 (2016).


\bibitem{HugThaele}
Hug, D. and Th\"ale, C.: Splitting tessellations in spherical spaces. Electron. J. Probab. \textbf{24}, article 24, 60pp, (2019).

\bibitem{kabluchko_poisson_zero}
Kabluchko, Z.: Expected $f$-vector of the {P}oisson zero polytope and random convex hulls in the half-sphere.  arXiv: 1901.10528.




\bibitem{KabluchkoTemesvariThaele}
Kabluchko, Z., Temesvari, D. and Th\"ale, C.: Expected intrinsic volumes and facet numbers of random beta-polytopes. Math. Nachr. \textbf{292}, 79--105 (2019).

\bibitem{KabluchkoTemesvariThaeleCones}
Kabluchko, Z., Temesvari, D. and Th\"ale, C.: A new approach to weak convergence of random cones and polytopes. arXiv: 2003.04001.

\bibitem{KabluchkoThaele_VoronoiSphere}
Kabluchko, Z. and Th\"ale, C.: The typical cell of a {V}oronoi tessellation on the sphere. arXiv: 1911.07221.


\bibitem{KabluchkoThaeleZaporozhets}
Kabluchko, Z., Th\"ale, C. and Zaporozhets, D.: Beta polytopes and Poisson polyhedra: $f$-vectors and angles. arXiv: 1805.01338.

\bibitem{MarynychKabluchkoTemesvariThaele}
Kabluchko, Z., Marynych, A., Temesvari, D. and Th\"ale, C.: Cones generated by random points on half-spheres and convex hulls of Poisson point processes. Probab. Theory Related Fields {\bf 175}, 1021--1061 (2019).

\bibitem{Matheron}
Matheron, G.: {\em Random Sets and Integral Geometry}. Wiley (1975).

\bibitem{McCoyTropp}
McCoy, M.B. and Tropp, J.A.: From Steiner formulas for cones to concentration of intrinsic volumes. Discrete Comput. Geom. \textbf{51}, 926--963 (2014).

\bibitem{MilesFlats}
Miles, R.E.: A synopsis of `Poisson flats in Euclidean spaces'. Izw. Akad. Nauk. Arm. SSR Ser. Mat. \textbf{5}, 263--285 (1970).

\bibitem{MilesSphere}
Miles, R.E.: Random points, sets and tessellations on the surface of a sphere. Sankhya Ser. A \textbf{33}, 145--174 (1971).

\bibitem{brychkov}
Prudnikov, A.P.,  Brychkov, Yu.A.  and Marichev, O.I.:
{\em Integrals and Series. {V}ol.\ 1. Elementary Functions.}  Translated from the Russian. Gordon \& Breach Science Publishers (1986).







\bibitem{SchneiderWeightedFaces}
Schneider, R.: Weighted faces of Poisson hyperplane tessellations. Adv. in Appl. Probab. \textbf{41}, 682--694 (2009).

\bibitem{SchneiderKinematicCones}
Schneider, R.: Intersection probabilities and kinematic formulas for polyhedral cones. Acta Math. Hungar.  \textbf{155}, 3--24 (2018).


\bibitem{SW}
Schneider, R. and Weil, W.: {\em Stochastic and Integral Geometry}. Springer (2008).




\end{thebibliography}
\end{document}